\documentclass[11pt]{article}
\newif \ifreview \reviewfalse
\usepackage{currfile,datetime}
\usepackage{amsthm,amsmath,amssymb,bbm,geometry,epsfig,listings,
paralist,enumerate,comment,nicefrac,verbatim,multirow,xcolor,wasysym}
\usepackage[pagewise]{lineno}[4.41]
\usepackage{mathrsfs}  
\usepackage{marginnote}\reversemarginpar
\usepackage{soul}

\usepackage[colorlinks,linkcolor=red,citecolor=green]{hyperref}

\usepackage[notref,notcite,final]{showkeys}
\usepackage[capitalise,compress]{cleveref} 

\ifreview

\linenumbers

\newcommand*\patchAmsMathEnvironmentForLineno[1]{%
    \expandafter\let\csname old#1\expandafter\endcsname\csname #1\endcsname
    \expandafter\let\csname oldend#1\expandafter\endcsname\csname end#1\endcsname
    \renewenvironment{#1}%
    {\linenomath\csname old#1\endcsname}%
    {\csname oldend#1\endcsname\endlinenomath}}%
  \newcommand*\patchBothAmsMathEnvironmentsForLineno[1]{%
    \patchAmsMathEnvironmentForLineno{#1}%
    \patchAmsMathEnvironmentForLineno{#1*}}%
  \patchBothAmsMathEnvironmentsForLineno{equation}%
  \patchBothAmsMathEnvironmentsForLineno{align}%
  \patchBothAmsMathEnvironmentsForLineno{gather}%
  \patchBothAmsMathEnvironmentsForLineno{multline}%
\else 

\fi

\crefname{equation}{}{}

\newtheorem{lemma}{Lemma}[section]

\newtheorem{proposition}[lemma]{Proposition}
\newtheorem{theorem}[lemma]{Theorem}

\newtheorem{corollary}[lemma]{Corollary}

\newtheorem{setting}[lemma]{Setting}

\crefname{subsection}{Subsection}{Subsections}
\crefname{enumi}{item}{items}
\creflabelformat{enumi}{(#2\textup{#1}#3)}

\newcommand{\1}{\ensuremath{\mathbbm{1}}}
\providecommand{\N}{{\ensuremath{\mathbbm{N}}}}
\providecommand{\Z}{{\ensuremath{\mathbbm{Z}}}}
\providecommand{\R}{{\ensuremath{\mathbbm{R}}}}

\providecommand{\E}{{\ensuremath{\mathbbm{E}}}}

\newcommand{\xeqref}[1]{}

\providecommand{\var}{{\ensuremath{\operatorname{\mathbb{V}ar}}}}
\renewcommand{\P}{{\ensuremath{\mathbbm{P}}}}
\renewcommand{\gets}{\curvearrowleft}
\newcommand{\F}{{\ensuremath{\mathbbm{F}}}}
\renewcommand{\limsup}{\varlimsup}
\renewcommand{\liminf}{\varliminf}

\newcommand{\unif}{\ensuremath{\mathfrak{r}}}
\newcommand{\totalD}{\mathsf{D}}

\newcommand{\approximationX}{\mathcal{X}}
\newcommand{\exponentLP}{p}
\newcommand{\exponentV}{p_1}
\newcommand{\exponentX}{p_2}

\newcommand{\expFirstNorm}{q_1}
\newcommand{\expSecondNorm}{q_2}

\newcommand{\FEU}{\mathcal{C}}
\newcommand{\Yappr}{\mathscr{Y}}\newcommand{\thetaBar}{\theta}
\newcommand{\FEY}{\mathfrak{C}}
\newcommand{\tnorm}[2]{{\left\vert\kern-0.25ex\left\vert\kern-0.25ex\left\vert #1     \right\vert\kern-0.25ex\right\vert\kern-0.25ex\right\vert}_{#2}}
\renewcommand{\epsilon}{\varepsilon}

\title{Multilevel Picard approximations for high-dimensional\\
decoupled forward-backward stochastic differential equations}
\author
{Martin Hutzenthaler$^{1}$ \\
 Tuan Anh Nguyen$^{2}$\bigskip\\
\small{$^1$ Faculty of Mathematics, University of Duisburg-Essen,}\\
\small{Essen, Germany; e-mail: \texttt{martin.hutzenthaler}\textcircled{\texttt{a}}\texttt{uni-due.de}}\\
\small{$^2$ Faculty of Mathematics, University of Duisburg-Essen,}\\
\small{Essen, Germany; e-mail: \texttt{tuan.nguyen}\textcircled{\texttt{a}}\texttt{uni-due.de}}
}

\AtBeginDocument{%
\sloppy
\allowdisplaybreaks
\linespread{1.1}
}

\begin{document}

\maketitle
\begin{abstract}
Backward stochastic differential equations (BSDEs)
appear in numeruous applications.
Classical approximation methods suffer from the curse of dimensionality
and deep learning-based approximation methods are not known to converge to the BSDE solution.
Recently, Hutzenthaler et al.\ \cite{hutzenthaler2021overcoming} introduced a new
approximation method for BSDEs whose forward diffusion is Brownian motion
and proved that this method converges with essentially optimal rate without suffering from
the curse of dimensionality.
The central object of this article is to extend this result to general forward diffusions.
The main challenge is that we need to establish convergence in temporal-spatial H\"older norms
since the forward diffusion cannot be sampled exactly in general.
\end{abstract}

\tableofcontents
{%
\makeatletter
\let\@makefnmark\relax
\let\@thefnmark\relax
\@footnotetext{\emph{Key words and phrases:}
stochastic differential equation, strong convergence, Euler--Maruyama approximation,  Lipschitz condition, Lyapunov function,
curse of dimensionality, high-dimensional SDEs, high-dimensional PDEs,
high-dimensional BSDEs, multilevel Picard approximations, multilevel Monte Carlo method. }
\@footnotetext{\emph{AMS 2010 subject classification}:Primary 60H30; Secondary 65C05, 65C30.} 
\makeatother
}%
%
%
%
%
%
\section{Introduction}
Backward differential equations (BSDEs) belong to the most frequently studied equations
in stochastic analysis and computational stochastics.
BSDEs arise in the solution of stochastic optimal control problems (see, e.g., \cite{touzi2012optimal, pham2009continuous, yong1999stochastic}),
BSDEs appear in the approximative valuation of financial products such as financial derivative contracts (see, e.g., \cite{ElKarouiPengQuenez1997,crepey2013financial,delong2013backward}),
and BSDEs are strongly linked to nonlinear partial differential equations (PDEs) which themselves arise naturally in many applications (see, e.g., \cite{PardouxPeng1992,peng1991probabilistic,pardoux1999bsdes,pardoux2014stochastic}). 
BSDEs in applications typically do not have a solution in closed form
and are often high-dimensional (where dimension refers to the underlying noise).

In view of the importance of BSDEs,
it has been an active research field to contruct efficient approximation methods 
for the last decades.
We refer, for example, to
\cite{
  abbas2020conditional,  
BallyPages2003,
 BenderDenk2007,
  Bender2015Primal,
BouchardTouzi2004,
  BriandLabart2014,
  ChassagneuxRichou2015, 
 CrisanManolarakis2012,
  FuZhaoZhou,
gobet2007error,
 GobetLabart2010,
 GobetLemorWarin2005,
  LionnetDosReisSzpruch2015, 	      
  Pham2015, 
  Turkedjiev2015,
  Zhang2004,
 zhao2006new}.
These methods, however, suffer from the curse of dimensionality
(cf., e.g., Bellman \cite{Bellman}, Novak \& Wozniakowski \cite[Chapter~1]{NovakWozniakowski2008I}, and Novak \& Ritter \cite{MR1485004})
in the sense that the number of computational operations 
to approximatively compute one sample path of the BSDE solution
grows at least exponentially in the reciprocal $\nicefrac{1}{\varepsilon}$
of the prescribed approximation accuracy $\varepsilon \in (0,\infty)$
or the dimension $d\in \N=\{1,2,3,\ldots\}$.
Moreover,
we refer, for example, to \cite{chen2019deep,EHanJentzen2017CMStat,FujiiTakahashiTakahashi2017,HenryLabordere2017}
and the references mentioned in the overview articles \cite{beck2020overview,han2020algorithms}
for deep learning-based approximation methods for BSDEs. These deep learning-based approximation methods, however,
are not known to converge to the BSDE solution.

Recently, \cite{EHutzenthalerJentzenKruse2016, HJKNW2018} introduced full history recursive multilevel Picard (MLP) approximations.
This nonlinear Monte Carlo method indeed overcomes the curse of dimensionality in the numerical approximation of
a large class of semilinear
partial differential equations (PDEs); see, e.g., \cite{HJKNW2018, hutzenthaler2019overcoming, HJKN20, HutzenthalerKruse2017, hjk2019overcoming, beck2019overcoming, beck2020overcoming,hutzenthaler2021strong,hutzenthaler2022multilevel} and,
 e.g., \cite{becker2020numerical} for simulations.
Based on the MLP method and on the multilevel approach of \cite{h98,Heinrich01},
\cite{hutzenthaler2021overcoming} then introduced an approximation method for BSDEs
and proved that this new method converges essentially with rate $\nicefrac{1}{2}$ to the BSDE solution
without suffering from the curse of dimensionality.
The main assumptions of \cite[Theorem 5.1]{hutzenthaler2021overcoming}
are that the terminal condition is a Lipschitz function of a forward diffusion which can be sampled exactly
 (e.g. Brownian motion or the Ornstein-Uhlenbeck process)
and that the nonlinearity is $z$-independent and has a bounded and globally Lipschitz continuous derivative.

In this article we extend the results
of \cite{hutzenthaler2021overcoming} to decoupled forward-backward stochastic differential equations (FBSDEs).
More precisley the contribution of this article is as follows:
\begin{enumerate}[(i)]
  \item We extend the approximation method
   of \cite{hutzenthaler2021overcoming} to more general forward diffusions
   (note for this
   that the Euler approximations in \eqref{eq:intro.Euler}
   agree with the solution of \eqref{eq:intro.X} if $\mu_d=0$ and $\sigma_d$ is constant).
   We note that it is more efficient that
   the  Euler approximations in \eqref{eq:intro.U} 
   are based on $m^l$ (resp.\ $m^{l-1}$)
   subintervals than using $m^n$ subintervals 
   (as was done in \cite{HJKN20}).
  \item The improved (compared to \cite{HJKN20}) efficiency has the
  drawback that the arguments on the right-hand side of
  \eqref{eq:intro.U} are not identical
  so that we need to analyze H\"older regularity of the MLP approximations.
  We prove for the first time that MLP approximations of semilinear PDEs
  converge in H\"older norms; see \Cref{p05} below.
  The proof of \Cref{s02}
  also uses that Euler approximations of SDEs converge
  in H\"older norms under suitable assumptions which was recently established 
  in \cite{HN21};
  cf.\ also \cref{m33b} below.
  We note that we assume higher regularity of all coefficient functions
  for our analysis with temporal-spatial H\"older norms.
 \item Our error criterion is the $L^2$-norm of the path distance where
  supremum over time is inside expectation; cf.\ \Cref{h01}.
  This improves the error criterion of \cite{hutzenthaler2021overcoming}
  where the $L^2$-distance at single time points was estimated.
  To achieve these improved error estimates, we analyze $L^p$-distances
  for MLP approximations
  (this was first done in \cite{hutzenthaler2021strong})
  for arbitrary $p\in[2,\infty)$ and apply an argument of
  \cite{CHJvNW21} of Kolmogorov-Chentsov-type.
\end{enumerate}
The following theorem is the main result of this article.

\begin{theorem}\label{h01}
Let  $T, c,\delta \in (0,\infty)$, 
 $  \Theta = \bigcup_{ n \in \N }\! \Z^n$,
for every $d\in\N=\{1,2,\ldots\}$ let
$f_d\in C^2( \R,\R)$, 
$g_d\in C^2( \R^d,\R)$,
$\mu_d\in C^2(\R^d,\R^d)$,
$\sigma_d \in C^2(\R^d,\R^{d\times d})$,
for every $d\in\N$ let $\|\cdot\|$ denote the standard norm on $\R^d$
and the Frobenius norm on $\R^{d\times d}$,
assume for all 
$d\in\N$, $w\in\R$,
$x,y,z\in\R^d$ that
$\lvert Tf_d(0)\rvert+ \lVert \mu_d(0)\rVert+\lVert \sigma_d(0)\rVert\leq cd^c$,
$\lvert f_d'(w)\rvert\leq c$,
$\lvert f_d''(w)\rvert\leq c$,
$  \lvert (\totalD g_d(x))(y)\rvert\leq c d^c\lVert y\rVert$, 
$  \lVert (\totalD \mu_d(x))(y)\rVert\leq c d^c\lVert y\rVert$,
$  \lVert (\totalD \sigma_d(x))(y)\rVert\leq c d^c\lVert y\rVert$,
$  \lvert (\totalD^2 g_d(x))(y,z)\rvert\leq c d^c\lVert y\rVert
\lVert  z\rVert$, 
$
\lVert (\totalD^2 \mu_d (x))(y,z)\rVert\leq c d^c\lVert y\rVert\lVert z\rVert$, 
and
$ 
\lVert (\totalD^2 \sigma_d (x))(y,z)\rVert\leq c d^c\lVert y\rVert\lVert z\rVert$, 
let $(\Omega,\mathcal{F},\P, (\F_t)_{t\in [0,T]})$ be a filtered probability space which satisfies the usual conditions,
let
$\unif^\theta\colon \Omega\to[0,1]$,
$\theta\in \Theta$, be i.i.d.\ random variables which satisfy for all
 $t\in [0,1]$
that
$\P(\unif^0 \leq t)= t$,
let $W^{d,\theta} \colon [0,T]\to \R^d$, $d\in\N$, $\theta\in\Theta$, be independent standard $(\F_t)_{t\in [0,T]}$-Brownian motions, assume that
$(\unif^{\theta})_{\theta\in\Theta} $ and
$(W^{d,\theta})_{d\in\N,\theta\in\Theta}$ are independent, for every $d,n\in\N$,
$\theta\in\Theta$,
$s\in[0,T]$,
$x\in\R^d$
let
$(\approximationX^{d,n,\theta,x}_{s,t})_{t\in[s,T]}
\colon [s,T] \times\R^d\times \Omega\to\R^d$ satisfy for all 
$k\in [0,n-1]\cap\Z$,
$t\in\bigl[\max\{s,\frac{kT}{n}\},\max\{s,\frac{(k+1)T}{n}\}\bigr]$
 that $\approximationX^{d,n,\theta,x}_{s,s}=x$ and
\begin{align}\small\begin{split}\label{eq:intro.Euler}
\approximationX^{d,n,\theta,x}_{s,t}&=
\approximationX^{d,n,\theta,x}_{s,\max\{s,\frac{kT}{n}\}}+ \mu_d\bigl(
\approximationX^{d,n,\theta,x}_{s,\max\{s,\frac{kT}{n}\}} \bigr)\bigl(t-\max\{s,\tfrac{kT}{n}\}\bigr) +
\sigma_d\bigl(
\approximationX^{d,n,\theta,x}_{s,\max\{s,\frac{kT}{n}\}} \bigr)\bigl(W^{d,\theta}_t-W^{d,\theta}_{\max\{s,\frac{kT}{n}\}}\bigr),\end{split}
\end{align}
let 
$ 
  {U}_{ n,m}^{d,\theta } \colon [0, T] \times \R^d \times \Omega \to \R
$,
$d\in\N$,
$n,m\in\Z$, $\theta\in\Theta$, satisfy for all $d,n,m\in \N$, $\theta\in\Theta$, $t\in[0,T]$, $x\in\R^d$ that
$
{U}_{-1,m}^{d,\theta}(t,x)={U}_{0,m}^{d,\theta}(t,x)=0$ and
\begin{equation}  \begin{split}\label{eq:intro.U}
  {U}_{n,m}^{d,\theta}(t,x)
=  &\sum_{\ell=0}^{n-1}\frac{1}{m^{n-\ell}}
    \sum_{i=1}^{m^{n-\ell}}
\Biggl[
      g_d\bigl(\approximationX^{d,m^\ell,(\theta,\ell,i),x}_{t,T}\bigr)-\1_{\N}(\ell)
g_d \bigl(\approximationX^{d,m^{\ell-1},(\theta,\ell,i),x}_{t,T}\bigr) \\
 &+(T-t)
     \bigl( f_d\circ {U}_{\ell,m}^{d,(\theta,\ell,i)}\bigr)
\left(t+(T-t)\unif^{(\theta,\ell,i)},\approximationX_{t,t+(T-t)\unif^{(\theta,\ell,i)}}^{d,m^\ell,(\theta,\ell,i),x}\right) \\
&-\1_{\N}(\ell)(T-t)
\bigl(f_d\circ {U}_{\ell-1,m}^{d,(\theta,\ell,-i)}\bigr)
    \left(t+(T-t)\unif^{(\theta,\ell,i)},\approximationX_{t,t+(T-t)\unif^{(\theta,\ell,i)}}^{d,m^{\ell-1},(\theta,\ell,i),x}   \right) \Biggr],
\end{split}     \end{equation}
 let $\lfloor \cdot \rfloor_m \colon \R \to \R$, $ m \in \N $, and 
$\lceil \cdot \rceil_m \colon  \R \to \R$, $ m \in \N $, 
satisfy for all $m \in \N$, $t \in [0,T]$ that
$\lfloor t \rfloor_m = \max( ([0,t]\backslash \{T\}) \cap \{ 0, \frac{ T }{ m }, \frac{ 2T }{ m }, \ldots \} )$
and 
$\lceil t \rceil_m = \min(((t,\infty) \cup \{T\})\cap  \{ 0, \frac{ T }{ m }, \frac{ 2T }{ m }, \ldots \} )$,
let $\Yappr^{d,n,m}=(\Yappr^{d,n,m}_t)_{t\in[0,T]}
\colon [0,T]\times\Omega\to\R $,
 $d,n,m\in\N$, 
satisfy  for all $d,n,m\in\N$, $t\in[0,T]$  that
\begin{equation}  \begin{split}\label{Y.intro}
\Yappr^{d,n,m}_{t}
= &\sum_{\ell=0}^{n-1}\Biggl[
\left[ \tfrac{ \lceil t \rceil_{m^{\ell+1}} - t }{ ( T / m^{ \ell + 1 } ) } \right]
U^{d,\ell}_{n-\ell,m}(\lfloor t \rfloor_{m^{\ell+1}}, \approximationX^{d,m^n,0,0}_{0,\lfloor t \rfloor_{m^{\ell+1}}})+
\left[\tfrac{ t-\lfloor t \rfloor_{m^{\ell+1}} }{ ( T / m^{ \ell + 1 } ) }\right]
U^{d,\ell}_{n-\ell,m}(\lceil t \rceil_{m^{\ell+1}}, \approximationX^{d,m^n,0,0}_{0,\lceil t \rceil_{m^{\ell+1}}}) \\
&-\1_{\N}(\ell)\biggl(
\left[ \tfrac{ \lceil t \rceil_{m^{\ell}} - t }{ ( T / m^{ \ell  } ) } \right]U^{d,\ell}_{n-\ell,m}(\lfloor t \rfloor_{m^{\ell}}, \approximationX^{d,m^n,0,0}_{0,\lfloor t \rfloor_{m^{\ell}}})+
\left[ \tfrac{ t-\lfloor  t \rfloor_{m^{\ell}} }{ ( T / m^{ \ell  } ) } \right]
U^{d,\ell}_{n-\ell,m}(\lceil t \rceil_{m^{\ell}}, \approximationX^{d,m^n,0,0}_{0,\lceil t \rceil_{m^{\ell}}})
\biggr)\Biggr],
\end{split}     \end{equation}
for every 
$ d,n,m\in\N$
let
$\FEY_{d,n,m}\in\N_0$ be the number of realizations of scalar random variables
and the number of function evaluations of $(f_d,g_d,\mu_d,\sigma_d)$ 
which are used to compute one realization of
$(\Yappr^{d,n,m}_{kT/m^n})_{k\in \{0,1,\ldots, m^n\}} $ (cf.\ \eqref{c10}--\eqref{c09} for a precise definition),
let $M\colon\N\to\N$ satisfy for all $n\in\N$
that $M(n)-\sqrt{\log(n)})\in[0,1]$,
and for every $d\in\N$ let $X^d,Z^d\colon[0,T]\times\Omega\to\R^d$
and $Y^d\colon[0,T]\times\Omega\to\R$
be
 $(\F_{t})_{t\in[0,T]}$-predictable stochastic processes
 such that
$
\int_0^T  \E \bigl[\|X_s^d\|+\lvert Y_s^d\rvert+\| Z_s^{d}\|^2\bigr]ds<\infty
$ and such that
for all  $t\in[0,T]$ it holds a.s.\ that
\begin{align}\label{eq:intro.X}
X^{d}_{t}&=  \int_{0}^{t} \mu_d(X^{d}_{r})\,dr+\int_{0}^{t} \sigma_d(X^{d}_{r})\,dW^{d,0}_r,\\
Y^d_t&=g_d(X^{d}_{T})+\int_t^T f_d(Y^d_s)\,ds-\int_t^T (Z_s^{d})^{T}\, dW_s^{d,0}.
\end{align}
Then 
there exist
$\gamma\in(0,\infty)$,
 $\mathsf{n}\colon \N\times(0,1)\to \N$ such that for all 
$d\in\N$,
$\epsilon\in (0,1)$, $n\in [\mathsf{n}(d,\epsilon),\infty)\cap\N$ it holds that
$
\left(\E\!\left[
\sup_{t\in[0,T]}
\left\lvert \Yappr^{d,n,M(n)}_{t}-  Y^d_t \right\rvert^{2}\right]\right)^{\nicefrac{1}{2}}<\epsilon$
 and 
$
\epsilon^{2+\delta}
{\FEY}_{d,n,M(n)}
\leq \gamma d^\gamma
$.
\end{theorem}
\noindent
\cref{h01} follows directly from \cref{s02}.

Now we heuristically motivate our approximation method
\cref{Y.intro} for fixed $d\in\N$.
By the nonlinear Feynman-Kac formula of \cite{PardouxPeng1992}
the FBSDE solution $Y^d$ satisfies a.s.\ that $Y^d=(u(t,X_t^d))_{t\in[0,T]}$
where $u$ is the viscosity solution of the associated PDE.
This PDE solution also solves the stochastic fixed-point equation
\begin{equation}  \begin{split}\label{fixed-point}
  \forall (t,x)\in[0,T]\times\R^d\colon
  u(t,x)=(\Phi(u))(t,x):=
  \E\Big[g(X_T^{t,x})+\int_t^T f_d(u(s,X_s^{t,x}))\,ds\Big]
\end{split}     \end{equation}
where $X^{t,x}$ is the solution of SDE \eqref{eq:intro.X} starting
in point $x\in\R^d$ at time $t\in[0,T]$; see \cite{beck2021nonlinear}.
The Picard approximation theorem suggests that the Picard iterates
$u_n:=\Phi^n(0)$, $n\in\N$, converge to $u$.
Now in the telescope expansion
\begin{equation}  \begin{split}
  u\approx u_n=\sum_{l=1}^{n-1}\big(u_{l+1}-\1_{\N}(l)u_l\big)
  =\sum_{l=0}^{n-1}\big(\Phi(u_{l}-\1_{\N}(l)\Phi(u_{l-1})\big)
\end{split}     \end{equation}
we approximate expectations and temporal integrals by
Monte Carlo averages where for large $l$ we need fewer Monte Carlo
samples (say $m^{n-l}$) since $u_l-u_{l-1}$ is then already small.
This roughly motivates \cref{eq:intro.U}.
To motivate \cref{Y.intro} let $\mathscr{L}_n$ be the piecewise
affine-linear interpolation of its argument restricted to the uniform
subgrid of $[0,T]$ with $m^n$ subintervals.
Now in the telescope expansion
\begin{equation}  \begin{split}
  Y\approx \mathscr{L}_n(Y)=\sum_{l=0}^{n-1}
  \big(\mathscr{L}_{l+1}(Y)-\1_{\N}(l)\mathscr{L}_{l}(Y)\big)
\end{split}     \end{equation}
it suffices to approximate $Y$ in the $l$-th summand by
$U_{n-l,m}(\cdot,X)$
since $\big(\mathscr{L}_{l+1}-\mathscr{L}_{l}\big)\big(Y-U_{n-l,m}(\cdot,X)\big)$
is already of the desired order $m^{-l}\cdot m^{l-n}=m^{-n}$
as the MLP approximations converge to $u$ in suitable H\"older norms.

The remainder of this article is organized as follows. \Cref{sec2} establishes
temporal-spatial H\"older-type regularity of solutions of the stochastic fixed-point equation
\cref{fixed-point}.
We derive error estimates for MLP approximations of PDEs in $L^p$-norms
in \Cref{sec3}
and in temporal-spatial H\"older norms in \Cref{sec4}.
Finally, in
\Cref{sec5}  
we prove \Cref{h01}.

\section{H\"older regularity of solutions to stochastic fixed-point equations}
\label{sec2}
In this section we derive the H\"older-type regularity estimate \cref{hoelder.regularity}
for $f\circ u$. This could in a second step be used to derive H\"older regularity of $u$.
Later we only need the regularity of $f\circ u$ and the regularity of $f\circ u$ is slightly
easier to derive from the fixed-point equations \cref{fixed-point}.
For \cref{l05} below we assume a general forward diffusion which satisfies a flow property,
a Lyapunov-type estimate,
and certain strong continuity in the starting point.

\begin{setting}\label{s12}
Let $d\in \N$,  $\exponentV,\exponentX\in [1,\infty)$,
 $c,T\in (0,\infty)$,
$f\in C(\R,\R)$,
$g\in C(\R^d,\R)$,
${V}\in C([0,T]\times  \R^d,   [1,\infty)) $
 satisfy that 
$ \tfrac{6}{\exponentV}+\tfrac{2}{\exponentX} \leq 1$,
let $\lVert\cdot \rVert\colon \R^d\to[0,\infty)$ be a norm, 
let $(\Omega,\mathcal{F},\P)$ be a probability space,
for every random variable 
$\mathfrak{X}\colon\Omega\to\R$
 let $\lVert\mathfrak{X}\rVert_{r} \in [0,\infty]$, $r\in[1,\infty)$,
satisfy for all $r\in[1,\infty)$ that $\lVert\mathfrak{X}\rVert_{r}^r= \E[\lvert\mathfrak{X}\rvert^r]$,
let
$(X^{x}_{s,t})_{s\in[0,T],t\in[s,T],x\in\R^d}
\colon \{(\mathfrak{s},\mathfrak{t})\in [0,T]^2\colon \mathfrak{s}\leq \mathfrak{t}\} \times\R^d\times \Omega\to\R^d
$ be measurable,
and
assume for all 
$s\in [0,T]$, $t\in[s,T]$,
$r\in [t,T]$, $x,y,\tilde{x},\tilde{y}\in\R^d$, 
$v,w,\tilde{v},\tilde{w}\in\R$,
$A\in ( \mathcal{B}(\R^d))^{\otimes \R^d}$,
$B\in (\mathcal{B}(\R^d))^{\otimes ([t,T]\times\R^d)}
$
 that
\begin{equation}
\lvert g(x)\rvert\leq V(T,x),\quad
\lvert Tf(0)\rvert\leq   V(s,x),\label{l06}
\end{equation}
\begin{equation}%
\begin{split}\label{l01}
& 
\left\lvert 
(g(x)-g(y))-( g(\tilde{x})-g(\tilde{y}))
\right\rvert \\
&
\leq 
\tfrac{V(T,x)+V(T,y)+V(T,\tilde{x})+V(T,\tilde{y})}{4}
\left(
\tfrac{\lVert(x-y)-(\tilde{x}-\tilde{y})\rVert}{\sqrt{T}} + \tfrac{(\lVert x-y\rVert+\lVert\tilde{x}-\tilde{y}\rVert)\lVert x-\tilde{x}\rVert}{2T}\right),
\end{split}
\end{equation}%
\begin{equation}
\label{l02}
\begin{split}
&
\left\lvert 
(f(v)-f(w))-( f(\tilde{v})-f(\tilde{w}))
\right\rvert 
\leq c\left\lvert (v-w)-(\tilde{v}-\tilde{w})\right\rvert  + \tfrac{c\left(\lvert v-w\rvert + \lvert \tilde{v}-\tilde{w}\rvert\right)\lvert v-\tilde{v}\rvert}{2},
\end{split}
\end{equation}%
\begin{equation}\label{d01}\begin{split}
\P \!\left(
X_{ t,r}^{X_{ s,t }^x } = X_{s,r}^x\right)=1,
\quad 
\P \!\left(X_{ s,t }^{\cdot} \in A,
X_{ t,\cdot}^{\cdot} \in B\right)= 
\P \!\left(X_{ s,t }^{\cdot} \in A\right)\P \!\left(
X_{ t,\cdot}^{\cdot} \in B\right)
,\end{split}
\end{equation}%
\begin{equation}
\begin{split}\label{l03}
&
\left\lVert \left\lVert 
(X_{s,t}^{x}-X_{s,t}^{y})
-(X_{s,t}^{\tilde{x}}-X_{s,t}^{\tilde{y}})\right\rVert
\right\rVert_{\exponentX}\\
&\leq 
\tfrac{V(s,x)+V(s,y)+V(s,\tilde{x})+V(s,\tilde{y})}{4}
\left[
\left\lVert 
(x-y)-(\tilde{x}-\tilde{y})
\right\rVert
+ \tfrac{\left(
\lVert x-y\rVert+\lVert\tilde{x}-\tilde{y}\rVert\right)
\lVert x-\tilde{x}\rVert}{2\sqrt{T}}\right]
,
\end{split}
\end{equation}%
and
\begin{equation}\begin{split}
\label{l04}
\left\lVert \left\lVert \bigl( X_{s,t}^x- 
X_{s,t}^y\bigr) -
( x-y)\right\rVert\right\rVert_{\exponentX}\leq \tfrac{(V(s,x)+V(s,y))}{2} \lVert x-y\rVert\tfrac{\lvert t-s\rvert^{\nicefrac{1}{2}}}{\sqrt{T}}.
\end{split}\end{equation}%
\end{setting}

\begin{lemma}\label{b02}
Assume \cref{s12} and assume for all $s\in [0,T]$, $t\in[s,T]$, $x\in\R^d$ that
\begin{equation}\label{l05}\begin{split}
 \left\lVert \left\lVert X_{s,t}^x-x\right\rVert\right\rVert_{\exponentX}\leq V(s,x)\lvert t-s\rvert^{\nicefrac{1}{2}}\quad\text{and}\quad\left\lVert V(t,X_{s,t}^x)\right\rVert_{\exponentV}\leq  V(s,x).
\end{split}\end{equation}
Then
\begin{enumerate}[(i)]\itemsep0pt
\item \label{k04}
there exists a unique measurable  $u\colon [0,T]\times\R^d\to\R$ which satisfies for all $t\in[0,T]$, $x\in \R^d$ that
$\E\bigl[\lvert
g(X_{t,T}^x)\rvert\bigr]+
\int_{t}^{T} \E\bigl[\lvert f( u(r,X_{t,r}^x))\rvert
\bigr]\,d r+\sup_{r\in[0,T],\xi\in\R^d}\frac{\lvert u(r,\xi)\rvert}{V(r,\xi)}<\infty$ and
$
u(t,x)=
\E\bigl[
g(X_{t,T}^x)\bigr]+
\int_{t}^{T} \E\bigl[f( u(r,X_{t,r}^x))
\bigr]\,dr,
$
\item\label{k03} it holds for all
$s\in[0,T]$,
$t\in[s,T]$,
$x,y\in\R^d$ 
that
 $\lvert u(t,x)\rvert\leq 2e^{c(T-t)}V(t,x)$ and
\begin{equation}
\lvert u(s,x)-u(t,y)\rvert\leq 4 e^{2 c T}
\left(\tfrac{V(s,x)+V(t,y)}{2}\right)^2
\tfrac{V(s,x)\lvert t-s\rvert ^{\nicefrac{1}{2}}+
\lVert x-y\rVert}{\sqrt{T}},
\end{equation} and
\item 
\label{k15}it holds
for all $s\in[0,T]$, $t\in [s,T]$,
 $x,y,\tilde{x},\tilde{y}\in\R^d$ that
\begin{align}\small\begin{split}\label{hoelder.regularity}
&\left\lvert \bigl[
f(
u(s,x))-f(u(s,y))\bigr]-\bigl[f(u(t,\tilde{x}))-f(u(t,\tilde{y}))\bigr]
\right\rvert \leq 896 ce^{6cT}
\left(\tfrac{V(s,x)+V(s,y)+V(t,\tilde{x})+V(t,\tilde{y})}{4}\right)^7\\
&\quad \cdot 
\left(
\tfrac{\lVert x-y-(\tilde{x}-\tilde{y})\rVert}{ \sqrt{T} }+
\tfrac{\lVert x-y\rVert+\lVert\tilde{x}-\tilde{y}\rVert}{\sqrt{T}}
\left[2
\tfrac{V(s,x)+V(s,y)}{2}\tfrac{\lvert t-s\rvert^{\nicefrac{1}{2}}}{\sqrt{T}}+\tfrac{\lVert x-\tilde{x}\rVert}{\sqrt{T}}\right]\right).\end{split}
\end{align}
\end{enumerate}
\end{lemma}
\begin{proof}[Proof of \cref{b02}]
First, observe that \eqref{l02}, \eqref{l01}, and \eqref{l03} show for all 
$s\in [0,T]$, $t\in [s,T]$,
$x,y\in\R^d$, $v,w\in\R$ that
\begin{equation}\label{l02b}
\lvert f(v)-f(w)\rvert =
\lvert (f(v)-f(w))-(f(v)-f(v))\rvert\leq c \lvert v-w\rvert,
\end{equation}
\begin{equation}\begin{split}
&
\lvert g(x)-g(y)\rvert=
\tfrac{\left\lvert (g(x)-g(y)) -(g(x)-g(x))
\right\rvert}{2}
+
\tfrac{\left\lvert (g(y)-g(x)) -(g(y)-g(y))
\right\rvert}{2}\\
&\leq\tfrac{1}{2} \tfrac{3V(T,x)+V(T,y)}{4}
\tfrac{\lVert x-y\rVert}{\sqrt{T}}
+ \tfrac{1}{2}
\tfrac{3V(T,y)+V(T,x)}{4}
\tfrac{\lVert y-x\rVert}{\sqrt{T}}= \tfrac{V(T,x)+V(T,y)}{2}\tfrac{\lVert x-y\rVert}{\sqrt{T}},
\end{split}\end{equation}
and
\begin{equation}
\begin{split}
&
\left\lVert \left\lVert 
X_{s,t}^{x}-X_{s,t}^{y}
\right\rVert
\right\rVert_{\exponentX}
=
\tfrac{\left\lVert \left\lVert 
(X_{s,t}^{x}-X_{s,t}^{y})
-(X_{s,t}^{x}-X_{s,t}^{x})\right\rVert
\right\rVert_{\exponentX}}{2}
+
\tfrac{\left\lVert \left\lVert 
(X_{s,t}^{y}-X_{s,t}^{x})
-(X_{s,t}^{y}-X_{s,t}^{y})\right\rVert
\right\rVert_{\exponentX}}{2}
\\
&\leq \tfrac{1}{2}
\tfrac{3V(s,x)+V(s,y)}{4}
\left\lVert 
x-y
\right\rVert+
\tfrac{1}{2}\tfrac{3V(s,y)+V(s,x)}{4}
\left\lVert 
y-x
\right\rVert= \tfrac{V(s,x)+V(s,y)}{2} \lVert x-y\rVert.
\end{split}\label{l03b}
\end{equation}
This, \cite[Lemma~3.1]{hutzenthaler2021overcoming}  (applied with 
$L\gets c$,
$p_3\gets\exponentV  $, 
$f\gets ([0,T]\times \R^d\times \R \ni (t,x,w) \mapsto f(w)\in\R)$,
$ \phi \gets V$,
$\psi\gets V$
 in the notation of \cite[Lemma~3.1]{hutzenthaler2021overcoming}), the fact that
$\frac{2}{\exponentV}+\frac{1}{\exponentX}\leq 1$, the fact that $1\leq V$,
the assumptions on measurability, 
\eqref{l06},
\eqref{d01}, and \eqref{l05}
prove \eqref{k04} and \eqref{k03}.

Next,
\eqref{l01}, H\"older's inequality, 
the fact that $\frac{1}{\exponentV}+\tfrac{2}{\exponentX}\leq 1$,
the triangle inequality, 
\eqref{l05}, \eqref{l03}, 
\eqref{l03b}, and the fact that $1\leq V$ show for all $s\in[0,T]$,  
$x_1,x_2,y_1,y_2\in\R^d$
 that
\begin{align}
&\left\lVert 
(g(X^{x_1}_{s,T})-g(X^{y_1}_{s,T}))-( g(X^{x_2}_{s,T})-g(X^{y_2}_{s,T}))
\right\rVert_1\nonumber \\
&
\leq\xeqref{l01}\biggl\lVert
\tfrac{V(T,X^{x_1}_{s,T})+V(T,X^{y_1}_{s,T})
+V(T,X^{x_2}_{s,T})+V(T,X^{y_2}_{s,T})
}{4}\nonumber \\
&\qquad\qquad\cdot\left(
\tfrac{\left\lVert (X_{s,T}^{x_1}-X_{s,T}^{y_1})-(X_{s,T}^{x_2}-
X_{s,T}^{y_2})\right\rVert}{\sqrt{T}} + 
\tfrac{\left(\left\lVert X_{s,T}^{x_1}-X_{s,T}^{y_1}\right\rVert+\left\lVert X_{s,T}^{x_2}-X_{s,T}^{y_2}\right\rVert\right)\left\lVert X_{s,T}^{x_1}-X_{s,T}^{x_2}\right\rVert}{2T}\right)\biggr\rVert_1\nonumber\\
&\leq 
 \left\lVert 
\tfrac{V(T,X^{x_1}_{s,T})+V(T,X^{y_1}_{s,T})
+V(T,X^{x_2}_{s,T})+V(T,X^{y_2}_{s,T})
}{4}\right\rVert_{\exponentV}\nonumber \\
&\qquad\cdot \left(
\tfrac{\left\lVert \left\lVert (X_{s,T}^{x_1}-X_{s,T}^{y_1})-(X_{s,T}^{x_2}-
X_{s,T}^{y_2})\right\rVert\right\rVert_{\exponentX}}{\sqrt{T}}+
\tfrac{\left(\left\lVert \left\lVert X_{s,T}^{x_1}-X_{s,T}^{y_1}\right\rVert\right\rVert_{\exponentX}
+
\left\lVert \left\lVert X_{s,T}^{x_2}-X_{s,T}^{y_2}\right\rVert\right\rVert_{\exponentX}\right)\left\lVert \left\lVert X_{s,T}^{x_1}-X_{s,T}^{x_2}\right\rVert\right\rVert_{\exponentX}}{2T}\right)\nonumber \\
&\leq\xeqref{l05} \tfrac{V(s,x_1)+V(s,y_1)+V(s,x_2)+V(s,y_2)}{4}
\Biggl(\xeqref{l03}\tfrac{V(s,x_1)+V(s,y_1)+V(s,x_2)+V(s,y_2)}{4}\nonumber \\
&\qquad\qquad\cdot 
\tfrac{\left\lVert 
(x_1-y_1)-(x_2-y_2)
\right\rVert
+ \tfrac{\left(
\lVert x_1-y_1\rVert+ \lVert x_2-y_2\rVert \right)
\lVert x_1-x_2\rVert }{2\sqrt{T}}}{\sqrt{T}}\nonumber \\
&+ \xeqref{l03b}\tfrac{
2\tfrac{V(s,x_1)+V(s,y_1)+V(s,x_2)+V(s,y_2)}{4}
\left(
\lVert x_1-y_1\rVert +
\lVert x_2-y_2\rVert \right)
2\tfrac{V(s,x_1)+V(s,y_1)+V(s,x_2)+V(s,y_2)}{4}\lVert x_1-x_2\rVert
}{2T}\Biggr)\nonumber \\
&\leq 5
\left(\tfrac{V(s,x_1)+V(s,y_1)+V(s,x_2)+V(s,y_2)}{4}\right)^3
\left[\tfrac{\lVert (x_1-y_1)-(x_2-y_2)\rVert}{\sqrt{T}}
+ 
\tfrac{\left(
\lVert x_1-y_1\rVert + \lVert x_2-y_2\rVert \right)
\lVert x_1-x_2\rVert}{2T}
\right].
\label{k05}\end{align}
Next, \eqref{k03}, H\"older's inequality, the fact that $\tfrac{2}{\exponentV}+\tfrac{1}{\exponentX}\leq \tfrac{1}{2}$,  the triangle inequality, 
\eqref{l05}, and \eqref{l03b} show for all $s\in[0,T]$, $t\in [s,T]$, $x,y\in\R^d$ that
\begin{align}
&\left\lVert 
 u(t,X_{s,t}^{x})-
 u(t,X_{s,t}^{y})\right\rVert_2\leq 
\xeqref{k03}
\left\lVert 
4 e^{2 c T}\left(\tfrac{V(t,X_{s,t}^{x})+V(t,X_{s,t}^{y})}{2}\right)^2
\tfrac{\left\lVert X_{s,t}^{x}-X_{s,t}^{y}\right\rVert}{\sqrt{T}}
\right\rVert_{2}\nonumber\\
&\leq 
4 e^{2 c T}
\left\lVert 
\tfrac{
V(t,X_{s,t}^{x})+V(t,X_{s,t}^{y})}{2}\right\rVert_{\exponentV}^2
\tfrac{\left\lVert \left\lVert X_{s,t}^{x}-X_{s,t}^{y}\right\rVert\right\rVert_{\exponentX}}{\sqrt{T}}\leq 
4 e^{2 c T} \xeqref{l05}\left(\tfrac{V(s,x)+V(s,y)}{2}\right)^3
\xeqref{l03b}
\tfrac{\lVert x-y \rVert}{\sqrt{T}}.\label{k08}\end{align}
This, \eqref{l02}, H\"older's inequality,  the triangle inequality, and the fact that
$\tfrac{c}{2}\left(4e^{2cT}\right)^2 2^6= 2^9ce^{4cT}$ imply for all $s\in[0,T]$, $t\in [s,T]$, $x_1,x_2,y_1,y_2\in\R^d$ that
\begin{align}
&
\left\lVert 
\left(
f( u(t,X_{s,t}^{x_1}))-
f( u(t,X_{s,t}^{y_1}))\right)
-\left(f( u(t,X_{s,t}^{x_2}))-f( u(t,X_{s,t}^{y_2}))\right)\right\rVert_1\nonumber\\
&\quad - c
\left\lVert 
\left(
 u(t,X_{s,t}^{x_1})-
 u(t,X_{s,t}^{y_1})\right)
-\left( u(t,X_{s,t}^{x_2})- u(t,X_{s,t}^{y_2})\right)\right\rVert_1\nonumber\\
&\leq \xeqref{l02}
\tfrac{c}{2}
\Bigl[
\left\lVert 
 u(t,X_{s,t}^{x_1})-
 u(t,X_{s,t}^{y_1})\right\rVert_2+\left\lVert  u(t,X_{s,t}^{x_2})- u(t,X_{s,t}^{y_2})\right\rVert_2 
\Bigr] \left\lVert u(t,X_{s,t}^{x_1})-u(t,X_{s,t}^{x_2})\right\rVert_2\nonumber \\
&\leq 
\tfrac{c}{2}\xeqref{k08}
4 e^{2 c T}\left[ \left(\tfrac{V(s,x_1)+V(s,y_1)}{2}\right)^3
\tfrac{\lVert x_1-y_1\rVert}{\sqrt{T}}
+ \left(\tfrac{V(s,x_2)+V(s,y_2)}{2}\right)^3
\tfrac{\lVert x_2-y_2\rVert}{\sqrt{T}}
\right]
\!
4 e^{2 c T} \left(\tfrac{V(s,x_1)+V(s,x_2)}{2}\right)^3
\tfrac{\lVert x_1-x_2\rVert}{\sqrt{T}}\nonumber \\
&\leq 2^{10} ce^{4cT}
\left(\tfrac{V(s,x_1)+V(s,y_1)+V(s,x_2)+V(s,y_2)}{4}\right)^6
\tfrac{\left[
\lVert x_1-y_1\rVert+\lVert x_2-y_2\rVert
\right]\lVert x_1-x_2\rVert}{2T}.\label{k22}
\end{align}
This, \eqref{k04}, 
\eqref{d01}, the disintegration theorem (see, e.g., \cite[Lemma~2.3]{HJKNW2018}),
and the triangle inequality
imply for all $s\in[0,T]$, $t\in [s,T]$, $x_1,x_2,y_1,y_2\in\R^d$ that
\begin{align}
&\left\lVert 
\left(
f( u(t,X_{s,t}^{x_1}))-
f( u(t,X_{s,t}^{y_1}))\right)
-\left(f( u(t,X_{s,t}^{x_2}))-f( u(t,X_{s,t}^{y_2}))\right)\right\rVert_1\nonumber \\
&
\quad - 2^{10} ce^{4cT} 
\left(\tfrac{V(s,x_1)+V(s,y_1)+V(s,x_2)+V(s,y_2)}{4}\right)^6
\tfrac{\left[
\lVert x_1-y_1\rVert+\lVert x_2-y_2\rVert
\right]\lVert x_1-x_2\rVert}{2T}\nonumber \\
&\leq \xeqref{k22} c
\left\lVert \Bigl[
(u(t,\tilde{x}_1)-u(t,\tilde{y}_1))-
(u(t,\tilde{x}_2)-u(t,\tilde{y}_2))\Bigr]
|_{\substack{\tilde{x}_1=X_{s,t}^{x_1},\tilde{y}_1=X_{s,t}^{y_1},\tilde{x}_2=X_{s,t}^{x_2},\tilde{y}_2=X_{s,t}^{y_2}}}
\right\rVert_1\nonumber \\
&\leq\xeqref{k04}  c \Biggl\lVert\Biggl\lVert
(g(X_{t,T}^{\tilde{x}_1})- 
g(X_{t,T}^{\tilde{y}_1}))
- (g(X_{t,T}^{\tilde{x}_2})-
g(X_{t,T}^{\tilde{y}_2}))+
{ \int_{t}^{T}}
\left(
f( u(r,X_{t,r}^{\tilde{x}_1}))-
f( u(r,X_{t,r}^{\tilde{y}_1}))\right)\nonumber \\
&\qquad\qquad\qquad\qquad\qquad  
-\left(f( u(r,X_{t,r}^{\tilde{x}_2}))-f( u(r,X_{t,r}^{\tilde{y}_2}))\right)dr\Biggr\lVert_1
\bigr|_{\substack{\tilde{x}_1=X_{s,t}^{x_1},\tilde{y}_1=X_{s,t}^{y_1},
\tilde{x}_2=X_{s,t}^{x_2},\tilde{y}_2=X_{s,t}^{y_2}}}\Biggr\rVert_1\nonumber \\
&= c \Biggl\lVert
(g(X_{s,T}^{x_1})- 
g(X_{s,T}^{y_1}))
- (g(X_{s,T}^{x_2})-
g(X_{s,T}^{y_2}))\nonumber \\
&\qquad \qquad +
{ \int_{t}^{T}}
\left(
f( u(r,X_{s,r}^{x_1}))-
f( u(r,X_{s,r}^{y_1}))\right)
-\left(f( u(r,X_{s,r}^{x_2}))-f( u(r,X_{s,r}^{y_2}))\right)dr\Biggr\rVert_1\nonumber 
\\
&\leq c
\left\lVert 
(g(X_{s,T}^{x_1})- 
g(X_{s,T}^{y_1}))
- (g(X_{s,T}^{x_2})-
g(X_{s,T}^{y_2}))
\right\rVert_1 \nonumber \\
&  \quad +
\int_{t}^{T}c\left\lVert 
\left(
f( u(r,X_{s,r}^{x_1}))-
f( u(r,X_{s,r}^{y_1}))\right)
-\left(f( u(r,X_{s,r}^{x_2}))-f( u(r,X_{s,r}^{y_2}))\right)\right\rVert_1dr.
\end{align}
This, \eqref{k05},  the fact that
$1\leq V
$, and the fact that 
$ 2^{10} ce^{4cT} +5c\leq 1029 ce^{4cT}  $
  imply for all $s\in[0,T]$, $t\in [s,T]$, $x_1,x_2,y_1,y_2\in\R^d$ that
\begin{align}
&
\left\lVert 
\left(
f( u(t,X_{s,t}^{x_1}))-
f( u(t,X_{s,t}^{y_1}))\right)
-\left(f( u(t,X_{s,t}^{x_2}))-f( u(t,X_{s,t}^{y_2}))\right)\right\rVert_1 
\nonumber \\
&\leq 
1029 ce^{4cT} 
\left(\tfrac{V(s,x_1)+V(s,y_1)+V(s,x_2)+V(s,y_2)}{4}\right)^6
\left[\tfrac{\lVert(x_1-y_1)-(x_2-y_2)\rVert}{\sqrt{T}}
+ 
\tfrac{\left(
\lVert x_1-y_1\rVert+\lVert x_2-y_2\rVert\right)
\lVert x_1-x_2\rVert}{2T}
\right]\nonumber \\
& \quad +
\int_{t}^{T}c\left\lVert 
\left(
f( u(r,X_{s,r}^{x_1}))-
f( u(r,X_{s,r}^{y_1}))\right)
-\left(f( u(r,X_{s,r}^{x_2}))-f( u(r,X_{s,r}^{y_2}))\right)\right\rVert_1dr.
\label{k13}
\end{align}
Moreover,  \eqref{k04}  shows for all
$s\in[0,T]$,
 $x\in\R^d$ that
$\int_{s}^{T}\left\lVert f(u(r,X_{s,r}^x))\right\rVert_1 dr<\infty  $.
This, \eqref{k13},  and the reverse
Gronwall lemma (see, e.g., \cite[Lemma~3.2]{hutzenthaler2019overcoming})
show 
for all $s\in[0,T]$, $t\in[s,T]$, $x_1,x_2,y_1,y_2\in\R^d$ that
\begin{align}
&
\left\lVert 
\left(
f( u(t,X_{s,t}^{x_1}))-
f( u(t,X_{s,t}^{y_1}))\right)
-\left(f( u(t,X_{s,t}^{x_2}))-f( u(t,X_{s,t}^{y_2}))\right)\right\rVert_1\nonumber \\
&\leq 
1029 ce^{5cT} 
\left(\tfrac{V(s,x_1)+V(s,y_1)+V(s,x_2)+V(s,y_2)}{4}\right)^6
\left[\tfrac{\lVert(x_1-y_1)-(x_2-y_2)\rVert}{\sqrt{T}}
+ 
\tfrac{\left(
\lVert x_1-y_1\rVert+\lVert x_2-y_2\rVert\right)
\lVert x_1-x_2\rVert }{2T}
\right].
\end{align}
This, \eqref{k05},  the fact that
$\forall\,x\in [0,\infty)\colon ex\leq e^x$, and the fact that
$1029e^{-1}\leq 379$
show 
for all $t\in[0,T]$,  $x_1,x_2,y_1,y_2\in\R^d$ that
$5+ 1029 cTe^{5cT}\leq 
5+ 1029e^{-1}e^{cT} e^{5cT}\leq 
  384 e^{6cT}$ and
\begin{align}
&
\left\lVert 
(g(X^{x_1}_{t,T})-g(X^{y_1}_{t,T}))-( g(X^{x_2}_{t,T})-g(X^{y_2}_{t,T}))
\right\rVert_1\nonumber \\
&+T\sup_{r\in [t,T]}
\left\lVert 
\left(
f( u(r,X_{t,r}^{x_1}))-
f( u(r,X_{t,r}^{y_1}))\right)
-\left(f( u(r,X_{t,r}^{x_2}))-f( u(r,X_{t,r}^{y_2}))\right)\right\rVert_1\nonumber \\
&\leq 384 e^{6cT}
\left(\tfrac{V(t,x_1)+V(t,y_1)+V(t,x_2)+V(t,y_2)}{4}\right)^6
\left[\tfrac{\lVert(x_1-y_1)-(x_2-y_2)\rVert}{\sqrt{T}}
+ 
\tfrac{\left(
\lVert x_1-y_1\rVert+\lVert x_2-y_2\rVert\right)
\lVert x_1-x_2\rVert}{2T}
\right].\label{k13b}
\end{align}
Next, the triangle inequality, \eqref{l04}, \eqref{l03b}, and \eqref{l05} show that for all $s\in[0,T]$, $t\in[s,T]$, $x,y,\tilde{x},\tilde{y}\in\R^d$
with $V(s,x)\leq V(s,y)$ that
\begin{align}
&
\tfrac{1}{\sqrt{T}}\left\lVert \left\lVert 
(X_{s,t}^x - X_{s,t}^y)   
-(\tilde{x} - \tilde{y} )\right\rVert\right\rVert_{\exponentX}+  
\tfrac{1}{2T}\left[\left\lVert \left\lVert 
X_{s,t}^x - X_{s,t}^y   \right\rVert\right\rVert_{\exponentX}+\left\lVert 
\tilde{x} - \tilde{y} \right\rVert\right] \left\lVert \left\lVert X_{s,t}^x-\tilde{x}\right\rVert\right\rVert_{\exponentX}\nonumber \\
&\leq \tfrac{
\left\lVert \left\lVert 
(X_{s,t}^x - X_{s,t}^y)   
-({x} - {y} )\right\rVert\right\rVert_{\exponentX}
+
\lVert x-y-(\tilde{x}-\tilde{y})\rVert}{\sqrt{T}}
+
\tfrac{\left[\left\lVert \left\lVert 
X_{s,t}^x - X_{s,t}^y  \right\rVert \right\rVert_{\exponentX}+\left\lVert 
\tilde{x} - \tilde{y} \right\rVert\right]\left[ \left\lVert \left\lVert X_{s,t}^x-x\right\rVert\right\rVert_{\exponentX} +\lVert x-\tilde{x}\lVert\right]}{2T}\nonumber \\
&\leq \xeqref{l04}
\tfrac{(V(s,x)+V(s,y))}{2} \tfrac{ \lVert x-y\rVert }{\sqrt{T}}\tfrac{\lvert t-s\rvert^{\nicefrac{1}{2}}}{\sqrt{T}}
+\tfrac{\lVert x-y-(\tilde{x}-\tilde{y})\rVert}{ \sqrt{T} }
+
\tfrac{\xeqref{l03b}\left[\frac{V(s,x)+V(s,y)}{2}
\lVert x-y\rVert+\lVert\tilde{x}-\tilde{y}\rVert
\right]\left[\xeqref{l05}
V(s,x)\lvert t-s\rvert^{\nicefrac{1}{2}}+\lVert x-\tilde{x}\rVert\right]}{2T}\nonumber \\
&\leq 
\tfrac{\lVert x-y-(\tilde{x}-\tilde{y})\rVert}{ \sqrt{T} }+
\tfrac{V(s,x)+V(s,y)}{2}
\tfrac{\lVert x-y\rVert+\lVert\tilde{x}-\tilde{y}\rVert}{\sqrt{T}}
\left[\tfrac{3}{2}
\tfrac{V(s,x)+V(s,y)}{2}\tfrac{\lvert t-s\rvert^{\nicefrac{1}{2}}}{\sqrt{T}}+\tfrac{\lVert x-\tilde{x}\rVert}{\sqrt{T}}\right].\label{k02b}
\end{align}
Moreover, \eqref{k04}, \eqref{d01}, the distributional and independence properties, and the  disintegration theorem (see, e.g., \cite[Lemma~2.4]{HJKNW2018}) show  for all $s\in[0,T]$, $t\in[s,T]$, $x\in\R^d$ that 
\begin{equation}\label{k02}\begin{split}
&\E [ u(t,X_{s,t}^x)  ] 
= \E\!\left[ u(t,\tilde{x})\bigr|_{\tilde{x}=X_{s,t}^x}\right]
\\&=\E\! \left[\E\!\left[
g(X_{t,T}^{\tilde{x}})+
\int_{t}^{T} f( u(r,X_{t,r}^{\tilde{x}}))\,dr\right]\bigr|_{\tilde{x}=X_{s,t}^x}
\right]=
\E\!\left [ g(X_{s,T}^x)  +\int_{t}^{T}f(u(r,X_{s,r}^x))\, dr\right].
\end{split}\end{equation}
This,  \eqref{k04},  Tonelli's theorem,
H\"older's inequality,
the fact that $ \tfrac{6}{\exponentV}+\tfrac{2}{\exponentX} \leq 1$,
the triangle inequality,
\eqref{k13b},  \eqref{l05},
and \eqref{k02b}   show 
 for all $s\in[0,T]$, $t\in[s,T]$, $x,y,\tilde{x},\tilde{y}\in\R^d$ that 
 \begin{align}
&
\left\lvert \E\bigl[\bigl( u(t,X_{s,t}^x) -u(t,X_{s,t}^y)\bigr)- 
\bigl(u(t,\tilde{x})-u(t,\tilde{y}) \bigr)\bigr]\right\rvert \nonumber\\
&\leq \xeqref{k04}
\Biggl\lVert\Biggl[  \left\lVert \left( g(X_{t,T}^\xi))
-g(X_{t,T}^\eta) \right)
-\left(g(X_{t,T}^{\tilde{x}}) 
- g(X_{t,T}^{\tilde{y}})\right)\right\rVert_1\nonumber \\
&\qquad+
T\sup_{r\in [t,T]}
 \left\lVert 
\left[ f( u(r,X_{t,r}^\xi))
-f( u(r,X_{t,r}^\eta)) \right] 
-\left[f( u(r,X_{t,r}^{\tilde{x}}) )
- f(u(r,X_{t,r}^{\tilde{y}}))\right]\right\rVert_1\Biggr]\bigr|_{\substack{\xi= X_{s,t}^x\\\eta= X_{s,t}^y}}\Biggr\rVert_1\nonumber\\
&\leq\xeqref{k13b} 384 e^{6cT}\left\lVert \left\lVert 
\left(\tfrac{V(t,\xi)+V(t,\eta)+V(t,\tilde{x})+V(t,\tilde{y})}{4}\right)^6
\left[\tfrac{\lVert(\xi- \eta)-(\tilde{x}-\tilde{y})\rVert}{\sqrt{T}}
+ 
\tfrac{\left(
\lVert\xi-\eta\rVert+\lVert\tilde{x}-\tilde{y}\rVert\right)
\lVert\xi-\tilde{x}\rVert}{2T}
\right]
\right\rVert_1\Bigr|_{\substack{\xi= X_{s,t}^x\nonumber,\eta= X_{s,t}^y}}\right\rVert_1\nonumber\\
&\leq 
384 e^{6cT}
\left\lVert \tfrac{V(t, X_{s,t}^x)+V(t,X_{s,t}^y)+V(t,\tilde{x})+V(t,\tilde{y})}{4}\right\rVert_{\exponentV}^6\nonumber \\
&\quad \cdot\left(\tfrac{1}{\sqrt{T}}\left\lVert \left\lVert 
(X_{s,t}^x - X_{s,t}^y)   
-(\tilde{x} - \tilde{y} )\right\rVert\right\rVert_{\frac{\exponentX}{2}}+  
\tfrac{1}{2T}\left[\left\lVert \left\lVert 
X_{s,t}^x - X_{s,t}^y   \right\rVert\right\rVert_{\exponentX}+\left\lVert 
\tilde{x} - \tilde{y} \right\rVert\right] \left\lVert \left\lVert X_{s,t}^x-\tilde{x}\right\rVert\right\rVert_{\exponentX}\right)\nonumber \\
&\leq 384 e^{6cT}\xeqref{l05}
\left(\tfrac{V(s,x)+V(s,y)+V(t,\tilde{x})+V(t,\tilde{y})}{4}\right)^6\nonumber \\
&\quad \cdot \xeqref{k02b}
\left(
\tfrac{\lVert x-y-(\tilde{x}-\tilde{y})\rVert}{ \sqrt{T} }+
\tfrac{V(s,x)+V(s,y)}{2}
\tfrac{\lVert x-y\rVert+\lVert\tilde{x}-\tilde{y}\rVert}{\sqrt{T}}
\left[\tfrac{3}{2}
\tfrac{V(s,x)+V(s,y)}{2}\tfrac{\lvert t-s\rvert^{\nicefrac{1}{2}}}{\sqrt{T}}+\tfrac{\lVert x-\tilde{x}\rVert}{\sqrt{T}}\right]\right)\nonumber \\
&\leq 
768 e^{6cT}
\left(\tfrac{V(s,x)+V(s,y)+V(t,\tilde{x})+V(t,\tilde{y})}{4}\right)^7\nonumber \\
&\quad \cdot 
\left(
\tfrac{\lVert x-y-(\tilde{x}-\tilde{y})\rVert}{ \sqrt{T} }+
\tfrac{\lVert x-y\rVert+\lVert\tilde{x}-\tilde{y}\rVert}{\sqrt{T}}
\left[\tfrac{3}{2}
\tfrac{V(s,x)+V(s,y)}{2}\tfrac{\lvert t-s\rvert^{\nicefrac{1}{2}}}{\sqrt{T}}+\tfrac{\lVert x-\tilde{x}\rVert }{\sqrt{T}}\right]\right)
.\label{k10b}
\end{align}
Next, \eqref{k02}, \eqref{l02b}, and \eqref{k08} show for all $s\in[0,T]$, $t\in [s,T]$,
 $x,y\in\R^d$ that
\begin{align}
&
\left\lvert 
\E\bigl[ 
\bigl(u(s,x)-u(s,y)\bigr) -\bigl( u(t,X_{s,t}^x) -u(t,X_{s,t}^y)\bigr)\bigr]\right\rvert 
= \xeqref{k02}\left\lvert \int_{s}^{t}\E\bigl[ f(u(r,X_{s,r}^x))-f(u(r,X_{s,r}^y))\bigr]\,dr\right\rvert \nonumber \\
&
\leq \int_{s}^{t}
 \left\lVert f( u(r,X_{s,r}^x))
-f( u(r,X_{s,r}^y))\right\rVert_1dr\leq\xeqref{l02b} c\lvert t-s \rvert \left[\sup\nolimits_{r\in [s,t]} \left\lVert u(r,X_{s,r}^x)-u(r,X_{s,r}^y)\right\rVert_1\right]\nonumber \\
&\leq\xeqref{k08} cT\tfrac{\lvert t-s \rvert ^{\nicefrac{1}{2}}}{\sqrt{T}}
4 e^{2 c T} \left(\tfrac{V(s,x)+V(s,y)}{2}\right)^3
\tfrac{\lVert x-y\rVert}{\sqrt{T}}= 4cTe^{2cT}
\left(\tfrac{V(s,x)+V(s,y)}{2}\right)^3
\tfrac{\lVert x-y\rVert}{\sqrt{T}}\tfrac{\lvert t-s \rvert ^{\nicefrac{1}{2}}}{\sqrt{T}}.
\label{k10}
\end{align}
Next, \eqref{k03} shows for all $s\in[0,T]$, $t\in [s,T]$,
 $x,y,\tilde{x},\tilde{y}\in\R^d$ that
\begin{align}
&\tfrac{c}{2}\left[
\lvert u(s,x)-u(s,y)\rvert + \lvert u(t,\tilde{x})-u(t,\tilde{y})\rvert
\right]
\lvert u(s,x)-u(t,\tilde{x})\rvert \nonumber \\
&\leq \tfrac{c(4 e^{2 c T})^2}{2}\!\left[
\left(\tfrac{V(s,x)+V(s,y)}{2}\right)^2
\tfrac{\lVert x-y\rVert}{\sqrt{T}}+
\left(\tfrac{V(t,\tilde{x})+V(t,\tilde{y})}{2}\right)^2
\tfrac{\lVert \tilde{x}-\tilde{y}\rVert}{\sqrt{T}}\right]
\left(\tfrac{V(s,x)+V(t,\tilde{x})}{2}\right)^2
\left(\tfrac{V(s,x)\lvert t-s \rvert ^{\nicefrac{1}{2}}}{\sqrt{T}}+
\tfrac{\lVert x-\tilde{x}\rVert}{\sqrt{T}}\right)\nonumber \\
&\leq 2^7 ce^{4cT}\!
\left(\tfrac{V(s,x)+V(s,y)+V(t,\tilde{x})+V(t,\tilde{y})}{4}\right)^4
\tfrac{\lVert x-y\rVert+\lVert\tilde{x}-\tilde{y}\rVert}{\sqrt{T}}
\left(V(s,x)\tfrac{\lvert t-s \rvert ^{\nicefrac{1}{2}}}{\sqrt{T}}+
\tfrac{\lVert x-\tilde{x}\rVert}{\sqrt{T}}\right) .\label{k23}
\end{align}
This,
\eqref{l02}, 
the triangle inequality, \eqref{k10}, \eqref{k10b}, \eqref{k03},  the fact that $1\leq V$, and the fact that
$4c^2T e^{2cT}2^3+ 768 c e^{6cT}+2^7ce^{4cT}\leq 
768 c e^{6cT} + 2^7ce^{4cT}(1+cT)\leq 896 c e^{6cT}
$
show for all $s\in[0,T]$, $t\in [s,T]$,
 $x,y,\tilde{x},\tilde{y}\in\R^d$ that
\begin{align}
&\left\lvert \bigl[
f(
u(s,x))-f(u(s,y))\bigr]-\bigl[f(u(t,\tilde{x}))-f(u(t,\tilde{y}))\bigr]\right\rvert \nonumber\\
&
\leq\xeqref{l02} c \left\lvert 
(u(s,x)-u(s,y))-(u(t,\tilde{x})-u(t,\tilde{y}))\right\rvert 
+\tfrac{c\left[
\lvert u(s,x)-u(s,y)\rvert +\lvert u(t,\tilde{x})-u(t,\tilde{y})\rvert
\right]\lvert u(s,x)-u(t,\tilde{x})\rvert}{2}\nonumber\\
&\leq c \left\lvert \E\bigl[ 
\bigl(u(s,x)-u(s,y)\bigr) -\bigl( u(t,X_{s,t}^x) -u(t,X_{s,t}^y)\bigr)\bigr]\right\rvert \nonumber\\
& 
+c\left\lvert \E\bigl[\bigl( u(t,X_{s,t}^x) -u(t,X_{s,t}^y)\bigr)- 
\bigl(u(t,\tilde{x})-u(t,\tilde{y}) \bigr)\bigr]\right\rvert +\tfrac{c\left[
\lvert u(s,x)-u(s,y)\rvert+\lvert u(t,\tilde{x})-u(t,\tilde{y})\rvert 
\right]\lvert u(s,x)-u(t,\tilde{x})\rvert}{2}
\nonumber\\
&\leq\xeqref{k10} 4c^2Te^{2cT}2^3
\left(\tfrac{V(s,x)+V(s,y)+V(t,\tilde{x})+V(t,\tilde{y})}{4}\right)^3
\tfrac{\lVert x-y\rVert}{\sqrt{T}}\tfrac{\lvert t-s\rvert^{\nicefrac{1}{2}}}{\sqrt{T}}+
\xeqref{k10b}
768 ce^{6cT}
\left(\tfrac{V(s,x)+V(s,y)+V(t,\tilde{x})+V(t,\tilde{y})}{4}\right)^7\nonumber \\
&\quad \cdot 
\left(
\tfrac{\lVert x-y-(\tilde{x}-\tilde{y})\rVert}{ \sqrt{T} }+
\tfrac{\lVert x-y\rVert+\lVert\tilde{x}-\tilde{y}\rVert}{\sqrt{T}}
\left[\tfrac{3}{2}
\tfrac{V(s,x)+V(s,y)}{2}\tfrac{\lvert t-s\rvert^{\nicefrac{1}{2}}}{\sqrt{T}}+\tfrac{\lVert x-\tilde{x}\rVert}{\sqrt{T}}\right]\right)\nonumber \\
&+\xeqref{k23}
2^7 ce^{4cT}\!
\left(\tfrac{V(s,x)+V(s,y)+V(t,\tilde{x})+V(t,\tilde{y})}{4}\right)^4
\tfrac{\lVert x-y\rVert+\lVert \tilde{x}-\tilde{y}\rVert}{\sqrt{T}}
\left(2\tfrac{V(s,x)+V(s,y)}{2}\tfrac{\lvert t-s\rvert^{\nicefrac{1}{2}}}{\sqrt{T}}+
\tfrac{\lVert x-\tilde{x}\rVert}{\sqrt{T}}\right)\nonumber \\
&\leq 896 ce^{6cT}
\left(\tfrac{V(s,x)+V(s,y)+V(t,\tilde{x})+V(t,\tilde{y})}{4}\right)^7\nonumber \\
&\quad \cdot 
\left(
\tfrac{\lVert x-y-(\tilde{x}-\tilde{y})\rVert}{ \sqrt{T} }+
\tfrac{\lVert x-y\rVert+\lVert\tilde{x}-\tilde{y}\rVert}{\sqrt{T}}
\left[2
\tfrac{V(s,x)+V(s,y)}{2}\tfrac{\lvert t-s\rvert^{\nicefrac{1}{2}}}{\sqrt{T}}+\tfrac{\lVert x-\tilde{x}\rVert}{\sqrt{T}}\right]\right).
\end{align}
This proves \eqref{k15}.
The proof of \cref{b02} is thus completed.
\end{proof}

\section{Error estimates for MLP approximations in $L^p$-norms}
\label{sec3}
In this section we prove strong convergence rates of MLP approximations
in $L^p$-norms; see \Cref{a02} below.
We note that we rescale the $L^p$-distance with a suitable Lyapunov-type
function and take then the supremum over time and space; see
the definition \cref{a14b} of our semi-norms.
The assumptions of \Cref{a02} are collected in the following \Cref{a01}
and include global Lipschitz continuity of the nonlinearity $f$,
local Lipschitz continuity of the terminal condition $g$
whose local Lipschitz constant grows at most like the Lyapunov-type function $V$,
and strong regularity estimates \cref{k01} and \cref{k04b} for the forward diffusion.
\Cref{a02} extends the analysis of \cite{hutzenthaler2021strong}
where the foward diffusion is Brownian motion.
Moreover, \Cref{c01} shows that MLP method approximates
semilinear PDEs with a computatinal effort which is of order $2+$ in the
reciprocal accuracy $\nicefrac{1}{\epsilon}$ and at most of polynomial order in the
dimension. \Cref{c01} improves \cite{HJKN20} which derived computational order
$4+$ in the reciprocal accuracy under slightly weaker assumptions on the coefficient
functions.

\begin{setting}\label{a01}
Let $d\in \N$, 
$\exponentV,\exponentX\in [1,\infty)$,
$c,T \in (0,\infty)$, 
$f \in C( \R,\R)$, 
$g\in C(\R^d,\R)$,
$V\in C([0,T]\times\R^d,[1,\infty))$,
let $\lVert\cdot\rVert\colon\R^d\to[0,\infty)$ be a norm,
let  $  \Theta = \bigcup_{ n \in \N } \Z^n$,
let
$(\Omega, \mathcal{F}, \P)$
be a probability space,
for every random variable 
$\mathfrak{X}\colon\Omega\to\R\cup\{-\infty,\infty\}$
 let $\lVert\mathfrak{X}\rVert_{r} \in [0,\infty]$, $r\in[1,\infty)$,
satisfy for all $r\in[1,\infty)$ that $\lVert\mathfrak{X}\rVert_{r}^r= \E[\lvert\mathfrak{X}\rvert^r]$,
let
$\unif^\theta\colon \Omega\to[0,1]$,
$\theta\in \Theta$, be i.i.d.\ random variables which satisfy for all
$t\in [0,1]$
that
$\P(\unif^0 \leq t)= t$,
let
$(X^{x}_{s,t})_{s\in[0,T],t\in[s,T],x\in\R^d}
\colon \{(\mathfrak{s},\mathfrak{t})\in [0,T]^2\colon \mathfrak{s}\leq \mathfrak{t}\} \times\R^d\times \Omega\to\R^d
$ be measurable,
$
(\approximationX^{n,\theta,x}_{s,t})_{s\in[0,T],t\in[s,T],x\in\R^d}
\colon \{(\mathfrak{s},\mathfrak{t})\in [0,T]^2\colon \mathfrak{s}\leq \mathfrak{t}\} \times\R^d\times \Omega\to\R^d$, $\theta\in\Theta$,
$n\in\N$, be measurable,
assume 
for all $n,\tilde{n}\in\N $ that 
$
\bigl(\approximationX^{n,\theta,x}_{s,t}, \approximationX^{\tilde{n},\theta,\tilde{x}}_{\tilde{s},\tilde{t}}\bigr)_{s,\tilde{s}\in[0,T], t\in[s,T],\tilde{t}\in[\tilde{s},T],x,\tilde{x}\in\R^d}$,
$\theta\in\Theta$,
are i.i.d.\ random fields,
assume for all 
$s\in [0,T]$, $t\in[s,T]$,
$x,y\in \R^d$,
$w_1,w_2\in \R$, $n\in\N$, $A\in ( \mathcal{B}(\R^d))^{\otimes \R^d}$,
$B\in (\mathcal{B}(\R^d))^{\otimes ([t,T]\times\R^d)}
$ 
that
$((X^{x}_{s,t},\approximationX^{n,\theta,x}_{s,t}))_{\theta\in\Theta}$ and
$(\unif^\theta)_{\theta \in\Theta}$ are independent  
and that
\begin{equation}
\label{a14}
  \lvert g(x)\rvert \leq  V(T,x), \quad 
\lvert Tf(0)\rvert\leq V(s,x ), 
\end{equation}%
\begin{equation}\label{a14c}
\lvert g(x)-g(y)\rvert \leq 
\tfrac{V(T,x)+V(T,y)}{2}\tfrac{\lVert x-y\rVert}{\sqrt{T}},\quad 
\lvert  f(w_1)-f(w_2)\rvert\leq c \lvert w_1-w_2\rvert,
\end{equation}
\begin{equation}\label{d01c}\begin{split}
\P \!\left(
X_{ t,r}^{X_{ s,t }^{x} } = X_{s,r}^{x}\right)=1,
\quad 
\P \!\left(X_{ s,t }^{(\cdot)} \in A,
X_{ t,(\cdot)}^{(\cdot)} \in B\right)= 
\P \!\left(X_{ s,t }^{(\cdot)} \in A\right)\P \!\left(
X_{ t,(\cdot)}^{(\cdot)} \in B\right)
,\end{split}
\end{equation}
\begin{equation}\label{k01}
\left\lVert \left\lVert 
X_{s,t}^{x}-x\right\rVert
\right\rVert_{\exponentX}\leq V(s,x)\lvert t-s\rvert^{\nicefrac{1}{2}},\quad 
\left\lVert \left\lVert 
X_{s,t}^{x}-X_{s,t}^{y}\right\rVert
\right\rVert_{\exponentX}
\leq 
\tfrac{V(s,x)+V(s,y)}{2}\lVert x-y\rVert,
\end{equation}
\begin{equation}
\left\lVert\left\lVert
\approximationX_{s,t}^{n,0,x}-X_{s,t}^{x}
\right\rVert
\right\rVert_{\exponentX}\leq \tfrac{\sqrt{T}}{\sqrt{n}}V(s,x),\quad\text{and}\quad
 \bigl\lVert V(t,\approximationX_{s,t}^{n,0,x})\bigr\rVert_{\exponentV}
\leq V(s,x),
\label{k04b}
\end{equation}
and
let 
$ 
  {U}_{ n,m}^{\theta } \colon [0, T] \times \R^d \times \Omega \to \R
$,
$n,m\in\Z$, $\theta\in\Theta$, satisfy for all 
$n,m\in \N$, $\theta\in\Theta$, $t\in[0,T]$, $x\in\R^d$ that
$
{U}_{-1,m}^{\theta}(t,x)={U}_{0,m}^{\theta}(t,x)=0$ and
\begin{align}
\label{t27}
&  {U}_{n,m}^{\theta}(t,x)
=  \sum_{\ell=0}^{n-1}\Biggl[\frac{1}{m^{n-\ell}}
    \sum_{i=1}^{m^{n-\ell}}
\biggl(
      g \bigl(\approximationX^{m^\ell,(\theta,\ell,i),x}_{t,T}\bigr)-\1_{\N}(\ell)
g \bigl(\approximationX^{m^{\ell-1},(\theta,\ell,i),x}_{t,T}\bigr)\nonumber \\
 &+(T-t)
     \bigl( f\circ {U}_{\ell,m}^{(\theta,\ell,i)}\bigr)
\left(t+(T-t)\unif^{(\theta,\ell,i)},\approximationX_{t,t+(T-t)\unif^{(\theta,\ell,i)}}^{m^\ell,(\theta,\ell,i),x}\right)\nonumber \\
&-\1_{\N}(\ell)(T-t)
\bigl(f\circ {U}_{\ell-1,m}^{(\theta,\ell,-i)}\bigr)
    \left(t+(T-t)\unif^{(\theta,\ell,i)},\approximationX_{t,t+(T-t)\unif^{(\theta,\ell,i)}}^{m^{\ell-1},(\theta,\ell,i),x}   \right) \biggr)\Biggr].
\end{align}
\end{setting}

\begin{lemma}[Independence and distributional properties]\label{b11}
Assume \cref{a01}. Then
\begin{enumerate}[i)]\itemsep0pt
\item\label{b12} it holds for all 
$n\in\N_0$,
$m\in\N$,
 $\theta\in\Theta$ that
$U_{n,m}^\theta$ is 
measurable,
\item \label{b10b}it holds for all 
$n\in\N_0$,
$m\in\N$,
 $\theta\in\Theta$ that
\begin{equation}\begin{split}
&
\sigma\!\left(\left\{ U_{n,m}^\theta(t,x)\colon t\in[0,T],x\in \R^d\right\}\right)\\
&
\subseteq
\sigma\!\left(\left\{ \unif^{(\theta,\nu)}, \approximationX^{m^\ell,(\theta,\nu),x}_{s,t}\colon 
\nu\in\Theta, \ell\in [0,n-1]\cap\Z,
s\in[0,T], t\in[s,T] ,x\in\R^d\right\}\right),
\end{split}\end{equation}
\item \label{a18}
it holds
for all 
$\theta \in\Theta $, 
$m\in \N $ that
$
\bigl(U_{\ell,m}^{(\theta,\ell,i)}(t,x))\bigr)_{t\in[0,T],x\in\R^d}$,
$ \bigl(
U_{\ell-1,m}^{(\theta,\ell,-i)}\bigr)_{t\in[0,T],x\in\R^d}$,
$ \bigl(\bigl(
\approximationX^{m^\ell,(\theta,\ell,i),x}_{s,t}, \approximationX^{m^{\ell-1},(\theta,\ell,i),\tilde{x}}\bigr)\bigr)_{s,\tilde{s}\in[0,T],t\in [s,T], \tilde{t}\in [\tilde{s},T],x,\tilde{x}\in\R^d}$,
$  \unif^{(\theta,\ell,i)}$, $ i\in\N$, $\ell \in \N_0$,
are independent,
\item \label{b10}  it holds 
for all  $n\in\N_0$,
$m\in\N$ 
  that
$( U_{n,m}^\theta(t,x))_{t\in[0,T],x\in\R^d} $, $\theta\in\Theta$, are identically distributed, and
 \item\label{b10c} it holds 
for all  
$\theta\in\Theta$,
$\ell\in \N_0$, $m\in\N$, $t\in[0,T]$, $x\in\R^d$
  that
\begin{align}\begin{split}
&\Bigl[
(T-t)
     \bigl( f\circ {U}_{\ell,m}^{(\theta,\ell,i)}\bigr)\!
\left(t+(T-t)\unif^{(\theta,\ell,i)},\approximationX_{t,t+(T-t)\unif^{(\theta,\ell,i)}}^{m^\ell,(\theta,\ell,i),x}\right)\\
&\qquad\qquad-\1_{\N}(\ell)(T-t)
\bigl(f\circ {U}_{\ell-1,m}^{(\theta,\ell,-i)}\bigr)\!
    \left(t+(T-t)\unif^{(\theta,\ell,i)},\approximationX_{t,t+(T-t)\unif^{(\theta,\ell,i)}}^{m^{\ell-1},(\theta,\ell,i),x}   \right)\Bigr],\quad i\in\N,
\end{split}
\end{align}
are i.i.d.\ and have the same distribution as
\begin{align}\label{b13}\begin{split}
&\Bigl[
(T-t)
     \bigl( f\circ {U}_{\ell,m}^{0}\bigr)\!
\left(t+(T-t)\unif^{0},\approximationX_{t,t+(T-t)\unif^{0}}^{m^\ell,0,x}\right)\\
&\qquad\qquad-\1_{\N}(\ell)(T-t)
\bigl(f\circ {U}_{\ell-1,m}^{1}\bigr)\!
    \left(t+(T-t)\unif^{0},\approximationX_{t,t+(T-t)\unif^{0}}^{m^{\ell-1},0,x}   \right)\Bigr].
\end{split}
\end{align}
\end{enumerate}
\end{lemma}
\begin{proof}
[Proof of \cref{b11}]
The assumptions on measurability and distributions, basic properties of measurable functions,
and induction prove \eqref{b12} and \eqref{b10b}.
Next, \eqref{b10b} and the assumptions on independence
prove \eqref{a18}. Next,
\eqref{a18},
 the fact that
$\forall\,\theta\in\Theta,m\in\N\colon U_{0,m}^\theta=0$, 
\eqref{t27}, the disintegration theorem (see, e.g., \cite[Lemma 2.2]{HJKNW2018}), 
the assumptions on distributions,
and induction 
show \eqref{b10} and  \eqref{b10c}.
\end{proof}

\begin{proposition}[Error analysis by semi-norms]\label{a02}
Assume \cref{a01}, 
let $\exponentLP\in [2,\infty)$, $\expFirstNorm\in[3,\infty)$, assume that
$ \tfrac{6}{\exponentV}+\tfrac{2}{\exponentX} \leq 1$,
$
  \exponentLP \expFirstNorm\leq \exponentV$ 
 and  
$ \tfrac{2}{\exponentV}+\tfrac{1}{\exponentX}\leq \tfrac{1}{p} $, and
for every random field $H\colon [0,T]\times\R^d \times\Omega\to \R$ 
let 
$\tnorm{H}{1,s}\in[0,\infty]$, $s\in [0,T]$, satisfy for all $s\in [0,T]$ that
\begin{equation}
\label{a14b}
\tnorm{H}{1,s}=  \sup_{t\in[s,T]}
\sup_{x\in\R^d}\frac{\lVert H(t,x)\rVert_{\exponentLP}}{ 
(V(t,x))^{\expFirstNorm}}.
\end{equation}
Then
\begin{enumerate}[(i)]\itemsep0pt
\item \label{a02a}
there exists a unique measurable  $u\colon [0,T]\times\R^d\to\R$ which satisfies for all $t\in[0,T]$, $x\in \R^d$ that
$\E\bigl[\lvert
g(X_{t,T}^{x})\rvert\bigr]+
\int_{t}^{T} \E\bigl[\lvert f( u(r,X_{t,r}^{x}))\rvert
\bigr]\,d r+\sup_{r\in[0,T],\xi\in\R^d}\frac{\lvert u(r,\xi)\rvert}{V(r,\xi)}<\infty$ and
$
u(t,x)=
\E\bigl[
g(X_{t,T}^{x})\bigr]+
\int_{t}^{T} \E\bigl[f( u(r,X_{t,r}^{x}))
\bigr]\,dr,
$
\item \label{a02g} it holds
for all  $n,m\in\N$, $t\in[0,T]$, $x\in\R^d$, $\theta\in\Theta$ that
$ {U}_{n,m}^{\theta}(t,x)
$,
$      g \bigl(\approximationX^{m^{n-1},\theta,x}_{t,T}\bigr)$,
and $
\bigl( f\circ {U}_{n-1,m}^{\theta}\bigr)
\bigl(t+(T-t)\unif^{\theta},\approximationX_{t,t+(T-t)\unif^{\theta}}^{m^{n-1},\theta,x}\bigr)$
are  integrable,
\item  \label{a02f}it holds
for all  $n,m\in\N$, $t\in[0,T]$, $x\in\R^d$ that
\begin{align*}
 & \E \bigl[ {U}_{n,m}^{0}(t,x)\bigr]=
\E\bigl[
      g \bigl(\approximationX^{m^{n-1},0,x}_{t,T}\bigr)\bigr]
+(T-t)
 \E\!\left[\bigl( f\circ {U}_{n-1,m}^{0}\bigr)\!
\left(t+(T-t)\unif^{0},\approximationX_{t,t+(T-t)\unif^{0}}^{m^{n-1},0,x}\right)\right]\\
&\quad\text{and}\quad
u(t,x)= \E\bigl[
      g \bigl(X^{x}_{t,T}\bigr)\bigr]
+(T-t)
 \E\!\left[( f\circ u)
\left(t+(T-t)\unif^{0},X_{t,t+(T-t)\unif^{0}}^{x}\right)\right],
\end{align*}
\item\label{a02d}  it holds for all $n,m\in\N$, $s\in[0,T]$ that
\begin{align}
&\tnorm{{U}_{n,m}^{0}-u}{1,s}
\leq \tfrac{16mne^{3cT}\sqrt{p-1}}{\sqrt{m^n}}
+
\sum_{\ell=0}^{n-1}\left[
\tfrac{4(T-s)^{1-\frac{1}{\exponentLP}} c\sqrt{p-1}}{\sqrt{m^{n-\ell-1}}}
\left[\int_{s}^{T} \tnorm{ {U}_{\ell,m}^{0}- u}{1,\zeta}^{\exponentLP} d\zeta\right]^{1/\exponentLP} 
\right]
,
\end{align}
 and
\item \label{a02e} it holds for all $n,m\in\N$, $s\in[0,T]$ that
 \begin{align}
\tnorm{{U}_{n,m}^{0}-u}{1,s} 
\leq  16 mn (p-1)^{n/2} 
e^{4c(T-s)n}
e^{m^{p/2}/p}m^{-n/2}.
\end{align}
\end{enumerate}
\end{proposition}
\begin{proof}
[Proof of \cref{a02}]\sloppy
Throughout this proof for every 
random variable $\mathfrak{X}\colon \Omega\to \R$ with $\E [\lvert \mathfrak{X}\rvert]<\infty$ let $\var_{\exponentLP}(\mathfrak{X})\in[0,\infty]$ satisfy that
$
\var_{\exponentLP}(\mathfrak{X})=\lVert\mathfrak{X}-\E[\mathfrak{X}]\rVert_{\exponentLP}^2.
$
First, \eqref{k04b} and Fatou's lemma yield for all 
$s\in[0,T]$,
$t\in [s,T]$, $x\in \R^d$ that $(\approximationX_{s,t}^{n,0,x})_{n\in\N}$ converges to $X_{s,t}^{x}$ in probability and
\begin{align}
\left \lVert V(t,X_{s,t}^{x})\right\rVert_{\exponentV}
= 
\left \lVert\operatornamewithlimits{\P\text{-}\lim}_{n\to\infty} V(t,\approximationX_{s,t}^{n,0,x})\right\rVert_{\exponentV}\leq 
 \liminf_{n\to\infty}
\left \lVert V(t,\approximationX_{s,t}^{n,0,x})\right\rVert_{\exponentV}\leq\xeqref{k04b} V(s,x).
\label{t21}
\end{align}
This,
\cite[Lemma~3.1]{hutzenthaler2021overcoming} (applied with 
$L\gets c$,
$p_3\gets\exponentV  $, 
$f\gets ([0,T]\times \R^d\times \R \ni (t,x,w) \mapsto f(w)\in\R)$,
$\phi\gets V$, $\psi \gets V$
in the notation of \cite[Lemma~3.1]{hutzenthaler2021overcoming}), the fact that
$\frac{2}{\exponentV}+\frac{1}{\exponentX}\leq 1$, 
the assumptions on measurability and distributional properties, 
\eqref{a14}, \eqref{a14c}, \eqref{d01c}, \eqref{k01},   \eqref{a14b}, 
and the fact that
$1\leq V\leq V^{\expFirstNorm}$
 prove \eqref{a02a} and show for all
$t\in[0,T]$,
$x,y\in\R^d$ 
that
 \begin{equation}\begin{split}
&\lvert u(t,x)\rvert\leq 2e^{c(T-t)}V(t,x),\quad
\tnorm{u}{1,t}\leq 2e^{c(T-t)},\\
&\quad\text{and}\quad
\lvert u(t,x)-u(t,y)\rvert\leq 4 e^{2 c T}\left(\tfrac{V(t,x)+V(t,y)}{2}\right)^2
\tfrac{\lVert x-y\rVert}{\sqrt{T}}.\label{t04}
\end{split}\end{equation}
Next,  
\eqref{a14},  Jensen' inequality,  the fact that $ \exponentLP\leq \exponentV$, \eqref{k04b}, and the fact that $V\leq V^{\expFirstNorm}$ show for all
$n\in\N$,  $t\in[0,T]$,
 $x\in\R^d$ that
\begin{equation}\begin{split}
&\bigl\lVert
 g \bigl(\approximationX^{n,0,x}_{t,T}\bigr)\bigr\rVert_{\exponentLP}
\leq \xeqref{a14}
\bigl\lVert
V \bigl(T,\approximationX^{n,0,x}_{t,T}\bigr)\bigr\rVert_{\exponentLP}
\leq 
\bigl\lVert
V \bigl(T,\approximationX^{n,0,x}_{t,T}\bigr)\bigr\rVert_{\exponentV}\leq\xeqref{k04b} 
V(t,x)\leq (V(t,x))^{\expFirstNorm}.
\end{split}\label{a16}\end{equation}
In addition, \eqref{a14c},
H\"older's inequality, the fact that $\tfrac{1}{\exponentV}+\tfrac{1}{\exponentX}\leq \tfrac{1}{\exponentLP}$,
\eqref{k04b}, 
\eqref{t21},
and the fact that $V^2\leq V^{\expFirstNorm}$ prove for all 
$t\in[0,T]$,
  $n\in\N$, $x\in\R^d$ that
\begin{align}
&
\bigl\lVert
g\bigl(\approximationX_{t,T}^{n,0,x}\bigr) 
-g\bigl(X_{t,T}^{x}\bigr)\bigr\rVert_{\exponentLP}\nonumber \\
&
\leq\xeqref{a14c} \left\lVert 
\tfrac{V( T,\approximationX_{t, T}^{n,0,x})+
V(X_{t, T}^{x})
}{2}
\tfrac{\left\lVert \approximationX_{t, T}^{n,0,x}-X_{t, T}^{x}\right\rVert}{\sqrt{T}}\right\rVert_{\exponentLP}\leq 
\left\lVert 
\tfrac{V( T,\approximationX_{t, T}^{n,0,x})+
V(X_{t, T}^{x})
}{2}\right\rVert_{\exponentV}\left\lVert 
\tfrac{\left\lVert \approximationX_{t, T}^{n,0,x}-X_{t, T}^{x}\right\rVert}{\sqrt{T}}\right\rVert_{\exponentX}\nonumber \\
&\leq \xeqref{k04b}V(t,x)\cdot\xeqref{t21} \tfrac{V(t,x)}{\sqrt{n}}
 \leq 
\tfrac{  (V(t,x))^{\expFirstNorm}  }{\sqrt{n}}
.\label{a50}  
\end{align}
This and the triangle inequality 
show for all  $t\in[0,T]$,
$\ell,m\in\N$, $x\in\R^d$ that
\begin{align}\label{a30}
 &\left\lVert g \bigl(\approximationX^{m^\ell,0,x}_{t,T}\bigr)-
g \bigl(\approximationX^{m^{\ell-1},0,x}_{t,T}\bigr)\right\rVert_{\exponentLP}
\leq {\sum_{j=\ell-1}^{\ell}}
\left\lVert g \bigl(\approximationX^{m^j,0,x}_{t,T}\bigr)-
g \bigl(X^{x}_{t,T}\bigr)\right\rVert_{\exponentLP}
\leq 
{\sum\limits_{j=\ell-1}^{\ell}}
\tfrac{ (V(t,x))^{\expFirstNorm} }{\sqrt{m^j}}.
\end{align}
Next, the disintegration theorem (see, e.g., \cite[Lemma 2.2]{HJKNW2018}),
the assumptions on measurability and independence,  Jensen's inequality,  the fact that
$\exponentLP \expFirstNorm\leq \exponentV$, and \eqref{k04b}
 prove for all 
$t\in[0,T]$, $\ell\in\N_0$, $m,n\in\N$, $x\in\R^d$,
$\nu\in\Z$,
$H \in \mathrm{span}_{\R}( \{ f\circ {U}_{\ell,m}^{(0,\nu)}, f\circ u\})$
 that
\begin{equation}\begin{split}
&\left\lVert (T-t)H
\bigl(t+(T-t)\unif^{0},\approximationX_{t,t+(T-t)\unif^{0}}^{n,0,x}\bigr)\right\rVert_{\exponentLP}= (T-t)\left\lVert \left\lVert \left\lVert H
\bigl(r,y\bigr)\right\rVert_{\exponentLP}
\bigr|_{ y=\approximationX_{t,r}^{n,0,x} }\right\rVert
\bigr|_{r=t+(T-t)\unif^{0}}
\right\rVert_{\exponentLP} 
\\
&\leq (T-t)
\left\lVert 
\biggl[
\tnorm{H}{1,r}\left\lVert 
\bigl(V\bigl(r,\approximationX_{t,r}^{n,0,x} \bigr)\bigr)^{ \expFirstNorm }\right\rVert_{\exponentLP}\biggr]
\bigr|_{r=t+(T-t)\unif^{0}} \right\rVert_{\exponentLP}\\
&\leq (T-t)
\left\lVert 
\tnorm{H}{1,t+(T-t)\unif^{0}}
\right\rVert_{\exponentLP}\xeqref{k04b}(V(t,x))^{\expFirstNorm} 
.\label{a15}\end{split}
\end{equation}
Moreover, \eqref{a14c}  and \eqref{a14b} show for all $t\in [0,T]$ and all random fields $H,K\colon [0,T]\times\R^d\times\Omega\to\R $ that 
$\tnorm{f\circ H-f\circ K}{1,t}\leq c \tnorm{ H-K}{1,t} $. 
This, \eqref{a15}, and the independence and distributional properties
show for all
$t\in[0,T]$, $\nu,\ell\in\N_0$, $m,n\in\N$, $x\in\R^d$
that
\begin{align}\begin{split}
&
\frac{1}{(V(t,x))^{\expFirstNorm} }
\left\lVert (T-t)\bigl((f\circ {U}_{\ell,m}^{\nu})- (f\circ u)\bigr)
\bigl(t+(T-t)\unif^{0},\approximationX_{t,t+(T-t)\unif^{0}}^{n,0,x}\bigr)\right\rVert_{\exponentLP} 
\\
&\leq \xeqref{a15}(T-t)\Bigl\lVert\tnorm{(f\circ {U}_{\ell,m}^{\nu})-(f\circ u)}{1,t+(T-t)\unif^{0}}\Bigr\rVert_{\exponentLP}
\leq (T-t)c\Bigl\lVert\tnorm{ {U}_{\ell,m}^{\nu}- u}{1,t+(T-t)\unif^{0}}\Bigr\rVert_{\exponentLP}\\
&=(T-t)c\Bigl\lVert\tnorm{ {U}_{\ell,m}^{0}- u}{1,t+(T-t)\unif^{0}}\Bigr\rVert_{\exponentLP}
.\label{a15b}  \end{split}
\end{align}%
Next, 
\eqref{a14c},
\eqref{t04},
H\"older's inequality, the fact that 
$\tfrac{2}{\exponentV}+\tfrac{1}{\exponentX}\leq \tfrac{1}{p} $,
 \eqref{k04b}, 
\eqref{t21},
and
the fact that $ V^3\leq V^{\expFirstNorm}$
 prove for all $t\in [0,T]$,
$r\in[t,T]$, $x,y\in\R^d$, 
$n\in\N$ that 
\begin{align}\begin{split}
&\left\lVert 
(f\circ u)\bigl(r,\approximationX_{t,r}^{n,0,x}\bigr) 
-(f\circ u)\bigl(r,X_{t,r}^{x}\bigr)\right\rVert_{\exponentLP}\leq
\xeqref{a14c}
c
\left\lVert 
u\bigl(r,\approximationX_{t,r}^{n,0,x}\bigr) 
-u\bigl(r,X_{t,r}^{x}\bigr)\right\rVert_{\exponentLP} 
\\
&\leq \xeqref{t04}
4c e^{2 c T}
\left\lVert \left(\tfrac{
V(r,
\approximationX_{t,r}^{n,0,x})+
V(r,X_{t,r}^{x})}{{2}}\right)^2
\tfrac{\left\lVert 
\approximationX_{t,r}^{n,0,x}
-X_{t,r}^{x}\right\rVert}{\sqrt{T}}\right\rVert_{\exponentLP}
\\
&\leq 
4c e^{2 c T}
\left\lVert \tfrac{
V(r,
\approximationX_{t,r}^{n,0,x})+
V(r,X_{t,r}^{x})}{{2}}\right\rVert_{\exponentV}^2
\tfrac{\bigl \lVert\approximationX_{t,r}^{n,0,x}-X_{t,r}^{x}\bigr\rVert_{\exponentX}}{\sqrt{T}}
\\
&\leq  4c e^{2 c T}\xeqref{k04b}\xeqref{t21}
(V(t,x))^2 \cdot \xeqref{k04b}\tfrac{V(t,x)}{\sqrt{n}}\leq   
 4c e^{2 c T} \tfrac{(V(t,x))^{\expFirstNorm}}{\sqrt{n}}.
\end{split}
\label{a15c}  
\end{align}
 This, the triangle inequality, and
 \eqref{a15b} show for all
 $\ell,m\in\N$,
$t\in[0,T]$,
$x\in\R^d$
that 
\begin{align}
&  \left\lVert  (T-t)\left[ \bigl( f\circ {U}_{\ell,m}^{0}\bigr)\!
\left(t+(T-t)\unif^{0},\approximationX_{t,t+(T-t)\unif^{0}}^{m^\ell,0,x}\right) 
-
\bigl(f\circ  {U}_{\ell-1,m}^{1}\bigr)\!
    \left(t+(T-t)\unif^{0},\approximationX_{t,t+(T-t)\unif^{0}}^{m^{\ell-1},0,x}   \right)\right]\right\rVert_{\exponentLP}\nonumber 
\\
&\leq
\biggl[
(T-t)\left\lVert \bigl( (f\circ {U}_{\ell,m}^{0})-(f\circ u)\bigr)\!
\left(t+(T-t)\unif^{0},\approximationX_{t,t+(T-t)\unif^{0}}^{m^\ell,0,x}\right)\right\rVert _{\exponentLP} \nonumber \\
&\qquad+ 
(T-t)
\left\lVert \bigl( (f\circ {U}_{\ell-1,m}^{1})-(f\circ u)\bigr)\!
\left(t+(T-t)\unif^{0},\approximationX_{t,t+(T-t)\unif^{0}}^{m^{\ell-1},0,x}\right)\right\rVert _{\exponentLP}\nonumber 
\\
&\qquad+(T-t)
\left\lVert 
(f\circ u)\left(t+(T-t)\unif^{0},\approximationX_{t,t+(T-t)\unif^{0}}^{m^\ell,0,x}\right) 
-(f\circ u)\left(t+(T-t)\unif^{0},X_{t,t+(T-t)\unif^{0}}^{x}\right)
\right\rVert_{\exponentLP}\biggr]\nonumber \\
&\qquad+(T-t)
\left\lVert 
(f\circ u)\left(t+(T-t)\unif^{0},\approximationX_{t,t+(T-t)\unif^{0}}^{m^{\ell-1},0,x}\right) 
-(f\circ u)\left(t+(T-t)\unif^{0},X_{t,t+(T-t)\unif^{0}}^{x}\right)
\right\rVert_{\exponentLP}\biggr]\nonumber 
\\
&\leq \left[\sum_{j=\ell-1}^{\ell}\left[ \xeqref{a15b}(T-t) c \left\lVert \tnorm{ {U}_{j,m}^{0}- u}{1,t+(T-t)\unif^{0}}\right\rVert_{\exponentLP}
+\xeqref{a15c}
 \tfrac{4  c(T-t) e^{2 c T}}{\sqrt{m^j}}
 \right]\right](V(t,x))^{\expFirstNorm}.\label{a22}
\end{align}
This, 
\eqref{t27}, the triangle inequality,
the fact that $\forall\,m\in \N\colon U_{0,m}^0=0$,
the independence and distributional properties,
\eqref{a16}, \eqref{a30}, \eqref{t04}, and induction prove for all $n\in\N_0$, $m\in\N$, $x\in\R^d$,
$t\in[0,T]$, $\theta\in\Theta$ that
$
\tnorm{U^{\theta}_{n,m}}{1,t}+\bigl\lVert(T-t)(f\circ {U}_{n,m}^{\theta})
\bigl(t+(T-t)\unif^{\theta},\approximationX_{t,t+(T-t)\unif^{\theta}}^{m^n,\theta,x}\bigr)\bigr\rVert_{\exponentLP} <\infty.
$
This, \eqref{a14b}, and \eqref{a16} establish  \eqref{a02g}.

Next, \eqref{t27}, 
linearity,  the independence and distributional properties, and a telescoping sum argument  prove for all  $n,m\in\N$, $t\in[0,T]$, $x\in\R^d$ that
\begin{align}
& \E \bigl[ {U}_{n,m}^{0}(t,x)\bigr]=
  \sum_{\ell=0}^{n-1}\Biggl[\frac{1}{m^{n-\ell}}
    \sum_{i=1}^{m^{n-\ell}}
\biggl(
    \E\!\left[  g \bigl(\approximationX^{m^\ell,(0,\ell,i),x}_{t,T}\bigr)\right]
-\1_{\N}(\ell)\E\!\left[
g \bigl(\approximationX^{m^{\ell-1},(0,\ell,i),x}_{t,T}\bigr)\right]\nonumber \\
&\qquad\qquad+(T-t)\E\!\left[
     \bigl( f\circ {U}_{\ell,m}^{(0,\ell,i)}\bigr)\!
\left(t+(T-t)\unif^{(0,\ell,i)},\approximationX_{t,t+(T-t)\unif^{(0,\ell,i)}}^{m^\ell,(0,\ell,i),x}\right)\right]\nonumber \\
&\qquad\qquad-\1_{\N}(\ell)(T-t)\E\!\left[
\bigl(f\circ {U}_{\ell-1,m}^{(0,\ell,-i)}\bigr)\!
    \left(t+(T-t)\unif^{(0,\ell,i)},\approximationX_{t,t+(T-t)\unif^{(0,\ell,i)}}^{m^{\ell-1},(0,\ell,i),x}   \right) \right]\biggr)\Biggr]\nonumber 
\\
&=
  \sum_{\ell=0}^{n-1}\Biggl[\frac{1}{m^{n-\ell}}
    \sum_{i=1}^{m^{n-\ell}}
\biggl(
    \E\!\left[  g \bigl(\approximationX^{m^\ell,0,x}_{t,T}\bigr)\right]
-\1_{\N}(\ell)\E\!\left[
g \bigl(\approximationX^{m^{\ell-1},0,x}_{t,T}\bigr)\right]\nonumber \\
&\qquad\qquad +(T-t)\E\!\left[
     \bigl( f\circ {U}_{\ell,m}^{0}\bigr)\!
\left(t+(T-t)\unif^{0},\approximationX_{t,t+(T-t)\unif^{0}}^{m^\ell,0,x}\right)\right]\nonumber \\
&\qquad\qquad-\1_{\N}(\ell)(T-t)\E\!\left[
\bigl(f\circ  {U}_{\ell-1,m}^{1}\bigr)\!
    \left(t+(T-t)\unif^{0},\approximationX_{t,t+(T-t)\unif^{0}}^{m^{\ell-1},0,x}   \right) \right]\biggr)\Biggr]\nonumber 
\\
&=\E\bigl[
      g \bigl(\approximationX^{m^{n-1},0,x}_{t,T}\bigr)\bigr]
+(T-t)
 \E\!\left[\bigl( f\circ {U}_{n-1,m}^{0}\bigr)\!
\left(t+(T-t)\unif^{0},\approximationX_{t,t+(T-t)\unif^{0}}^{m^{n-1},0,x}\right)\right]
  .\label{a25} 
\end{align}
Moreover, \eqref{a02a},  the disintegration theorem (see, e.g., \cite[Lemma 2.2]{HJKNW2018}), and the independence and distributional properties show for all $t\in[0,T]$, $x\in\R^d$ that
\begin{equation}
u(t,x)= \E\bigl[
      g \bigl(X^{x}_{t,T}\bigr)\bigr]
+(T-t)
 \E\bigl[( f\circ u)
(t+(T-t)\unif^{0},X_{t,t+(T-t)\unif^{0}}^{x})\bigr].\label{a25c}
\end{equation}
This and \eqref{a25} show \eqref{a02f}.

Next, \eqref{a25} and \eqref{a25c} prove for all $n,m\in\N$, $t\in[0,T]$, $x\in\R^d$
 that
\begin{align}
&\E \bigl[  {U}_{n,m}^{0}(t,x)\bigr]-u(t,x)\nonumber 
=\E\bigl[
      g \bigl(\approximationX^{m^{n-1},0,x}_{t,T}\bigr)-g \bigl(X^{x}_{t,T}\bigr)\bigr]\\
&
+(T-t)
 \E\!\left[\bigl((f\circ {U}_{n-1,m}^{0})-(f\circ u)\bigr)\!
\left(t+(T-t)\unif^{0},\approximationX_{t,t+(T-t)\unif^{0}}^{m^{n-1},0,x}\right)\right]
\nonumber \\
&+(T-t)
 \E\!\left[(f\circ u)\!
\left(t+(T-t)\unif^{0},\approximationX_{t,t+(T-t)\unif^{0}}^{m^{n-1},0,x}\right)
-
( f\circ u)\!
\left(t+(T-t)\unif^{0},X_{t,t+(T-t)\unif^{0}}^{x}\right)\right].\label{a25b}
\end{align}
This, the triangle inequality, Jensen's inequality,
\eqref{a50},
\eqref{a15b}, 
\eqref{a15c}, and  the fact that
$ \forall\,t\in [0,T]\colon 4c(T-t)e^{2c T}+1\leq 4e^{2c T}(1+c(T-t))\leq 4e^{3cT}$
 prove for all  $n, m\in\N$, $t\in[0,T]$, $x\in\R^d$ that
\begin{align}\label{a23}
&\frac{\bigl\lvert\E \bigl[  {U}_{n, m}^{0}(t,x)\bigr]-u(t,x)\bigr\rvert}{(V(t,x))^{\expFirstNorm}}
\leq 
 (T-t)c\left\lVert \tnorm{ {U}_{n-1,m}^{0}- u}{1,t+(T-t)\unif^{0}}\right\rVert_{\exponentLP}
+\tfrac{4e^{3cT}}{\sqrt{m^{n-1}}}
.
\end{align}
Moreover, 
the Marcinkiewicz-Zygmund inequality (see \cite[Theorem~2.1]{Rio09}), the fact that $\exponentLP\in [2,\infty)$, the triangle inequality, and  Jensen's inequality show that for all $n\in\N$ and all i.i.d.\ random variables $\mathfrak{X}_k$, $k\in [1,n]\cap\Z$, with $\E [\lvert\mathfrak{X}_1\rvert]<\infty$ it holds that
$
(\var_{\exponentLP}(\sum_{k=1}^{n}\mathfrak{X}_k))^{1/2}=
\left\lVert \sum_{k=1}^{n}(\mathfrak{X}_k-\E[\mathfrak{X}_k])\right\rVert_{\exponentLP}\leq (\exponentLP-1)^{1/2}
(\sum_{k=1}^{n}\left\lVert \mathfrak{X}_k-\E[\mathfrak{X}_k]\right\rVert_p^2)^{1/2}\leq  2(\exponentLP-1)^{1/2}\lVert \mathfrak{X}_1\rVert_{\exponentLP}/\sqrt{n}
$.
This,
\eqref{t27},
the triangle inequality,  the properties on independence and distributions,  \eqref{a16},
\eqref{a30}, 
\eqref{a22}, and the fact that
$\forall\,t\in[0,T]\colon 4c(T-t)e^{2c T}+1\leq 4e^{2cT}(1+c(T-t))\leq 4e^{3cT}$
 prove
 for all  $n,m\in\N$, $t\in[0,T]$, $x\in\R^d$
that
\begin{align}
&  \tfrac{1}{(V(t,x))^{\expFirstNorm}}
\left \lVert {U}_{n,m}^{0}(t,x) -\E \bigl[{U}_{n,m}^{0}(t,x)\bigr]\right\rVert_{\exponentLP}
=\tfrac{1}{(V(t,x))^{\expFirstNorm}}\left(\var_{\exponentLP}({U}_{n,m}^{0}(t,x))\right)^{\nicefrac{1}{2}}
\nonumber \\
&
\leq\xeqref{t27} \tfrac{1}{(V(t,x))^{\expFirstNorm}} \sum_{\ell=0}^{n-1}\left(\var_{\exponentLP}\! \left[\tfrac{1}{m^{n-\ell}}
    \sum_{i=1}^{m^{n-\ell}}
g \bigl(\approximationX^{m^{\ell},(0,\ell,i),x}_{t,T}\bigr)-\1_{\N}(\ell)
g \bigl(\approximationX^{m^{\ell-1},(0,\ell,i),x}_{t,T}\bigr)\right]\right)^{\nicefrac{1}{2}}
\nonumber \\
&  \quad 
+\tfrac{1}{(V(t,x))^{\expFirstNorm}} \sum_{\ell=0}^{n-1}\Biggl(\var_{\exponentLP} \Biggl[
\tfrac{1}{m^{n-\ell}}
    \sum_{i=1}^{m^{n-\ell}}
\biggl[    
 (T-t)
     \bigl(  f\circ {U}_{\ell,m}^{(0,\ell,i)}\bigr)
\left(t+(T-t)\unif^{(0,\ell,i)},\approximationX_{t,t+(T-t)\unif^{(0,\ell,i)}}^{m^{\ell},(0,\ell,i),x}\right)\nonumber \\
&\qquad\qquad\qquad\qquad-\1_{\N}(\ell)(T-t)
\bigl( f\circ  {U}_{\ell-1,m}^{(0,\ell,-i)}\bigr)
    \left(t+(T-t)\unif^{(0,\ell,i)},\approximationX_{t,t+(T-t)\unif^{(0,\ell,i)}}^{m^{\ell-1},(0,\ell,i),x}   \right) \biggr]\Biggr]
\Biggr)^{\nicefrac{1}{2}}
  \nonumber 
\\
&\leq \tfrac{1}{(V(t,x))^{\expFirstNorm}}
\tfrac{2\sqrt{\exponentLP-1}}{\sqrt{m^n}}
\left\lVert
g \bigl(\approximationX^{1,0,x}_{t,T}\bigr)\right\rVert_{\exponentLP}+\tfrac{1}{(V(t,x))^{\expFirstNorm}}
\sum_{\ell=1}^{n-1}\left[\tfrac{2\sqrt{\exponentLP-1}}{\sqrt{m^{n-\ell}}}
\left\lVert
g \bigl(\approximationX^{m^{\ell},0,x}_{t,T}\bigr)-
g \bigl(\approximationX^{m^{\ell-1},0,x}_{t,T}\bigr)
\right\rVert_{\exponentLP}\right]\nonumber \\
&\quad+\tfrac{1}{(V(t,x))^{\expFirstNorm}}
\sum_{\ell=1}^{n-1}\Biggl[
\tfrac{2\sqrt{\exponentLP-1}}{\sqrt{m^{n-\ell}}}
\Bigl\lVert
 (T-t)
     \bigl(  f\circ {U}_{\ell,m}^{0}\bigr)
\left(t+(T-t)\unif^{0},\approximationX_{t,t+(T-t)\unif^{0}}^{m^{\ell},0,x}\right)\nonumber \\
&\qquad\qquad\qquad\qquad\qquad\qquad \qquad
-(T-t)
\bigl( f\circ  {U}_{\ell-1,m}^{1}\bigr)
    \left(t+(T-t)\unif^{0},\approximationX_{t,t+(T-t)\unif^{0}}^{m^{\ell-1},0,x}   \right) \Bigr\rVert_{\exponentLP}
\Biggr]
  \nonumber  
\\
&\leq\xeqref{a16}
\tfrac{2\sqrt{\exponentLP-1}}{\sqrt{m^n}}
+\sum_{\ell=1}^{n-1}\xeqref{a30}\sum_{j=\ell-1}^{\ell}\tfrac{2\sqrt{p-1}}{\sqrt{m^j}}
\nonumber\\
&\qquad\qquad\qquad+
\sum_{\ell=1}^{n-1}\sum_{j=\ell-1}^{\ell}\left[
\tfrac{2\sqrt{\exponentLP-1}}{\sqrt{m^{n-\ell}}}
\left(
\xeqref{a22} (T-t)c\left\lVert\tnorm{ {U}_{j,m}^{0}- u}{1,t+(T-t)\unif^{0}}\right\rVert_{\exponentLP}+
\tfrac{4c(T-t)e^{2cT}}{\sqrt{m^j}}
\right)
 \right]
\nonumber\\
&\leq 
\tfrac{2\sqrt{\exponentLP-1}}{\sqrt{m^n}}
+
\sum_{\ell=1}^{n-1}\sum_{j=\ell-1}^{\ell}\left[
\tfrac{2\sqrt{\exponentLP-1}}{\sqrt{m^{n-\ell}}}
\left(
 (T-t)c\left\lVert\tnorm{ {U}_{j,m}^{0}- u}{1,t+(T-t)\unif^{0}}\right\rVert_{\exponentLP}
+\tfrac{ 4e^{3cT} }{\sqrt{m^j}} 
 \right)\right].
\label{a28}
\end{align}
In addition,
the fact that $\unif^0$ is uniformly distributed on $[0,1]$ and  the substitution rule imply for all
$s\in[0,T]$, $t\in [0,T]$, and all measurable $h\colon [0,T]\to \R$ that
\begin{align}\label{b04}
&\textstyle(T-t)\left\lVert h(t+(T-t)\unif^0)\right\rVert_{\exponentLP}
= (T-t)^{1-\frac{1}{\exponentLP}}\left[\int_{0}^{1}(T-t)
\lvert h(t+(T-t)\lambda)\rvert^{\exponentLP}\,d\lambda\right]^{\frac{1}{\exponentLP}}\nonumber \\
&= \textstyle
 (T-t)^{1-\frac{1}{\exponentLP}}
\left[\int_{t}^{T}
\lvert h(\zeta)\rvert^{\exponentLP}\,d\zeta\right]^{\frac{1}{\exponentLP}}\leq 
 (T-s)^{1-\frac{1}{\exponentLP}}
\left[\int_{s}^{T}
\lvert h(\zeta)\rvert^{\exponentLP}\,d\zeta\right]^{\frac{1}{\exponentLP}} 
.
\end{align}
This, \eqref{a14b}, the triangle inequality, \eqref{a28},
\eqref{a23}, 
the fact that $\forall\,n,m\in \N$, $a_0,a_1,\ldots,a_{n-1}\in[0,\infty)\colon 
(
\sum_{\ell=1}^{n-1}\sum_{j=\ell-1}^\ell\frac{ a_j}{\sqrt{ m^{n-\ell}}}
)
+a_{n-1}
\leq \sum_{\ell=0}^{n-1}\frac{2a_\ell}{\sqrt{ m^{n-\ell-1}}}
$,
and the fact that
$ \forall\,m,n\in \N, t\in [0,T]\colon \tfrac{2}{\sqrt{m^n}} +(n-1)\frac{4\cdot 4e^{3cT}}{\sqrt{m^{n-1}}}= \frac{2+16e^{3cT}m(n-1)}{\sqrt{m^n}}\leq \frac{16mne^{3cT}}{\sqrt{m^n}}$
 show for all $s\in [0,T]$, $m,n\in\N$ that
\begin{align}
&\tnorm{{U}_{n,m}^{0}-u}{1,s}
=\xeqref{a14b}
\sup_{ t\in[s,T]}\left[
(V(t,x))^{-{\expFirstNorm}}
\left\lVert {U}_{n,m}^{0}(t,x)-u(t,x)\right\rVert_{\exponentLP}\right]\nonumber 
\\
&\leq \sup_{ t\in[s,T]}\left[(V(t,x))^{-{\expFirstNorm}}\left(\left\lVert 
{U}_{n,m}^{0}(t,x)-
\E [{U}_{n,m}^{0}(t,x)]
\right\rVert_{\exponentLP}+
\bigl\lvert\E \bigl[  {U}_{n,m}^{0}(t,x)\bigr]-u(t,x)\bigr\rvert
\right)\right]\nonumber 
\\
&\leq \sup_{t\in [s,T]}\Biggl[\xeqref{a28}
\tfrac{2\sqrt{\exponentLP-1}}{\sqrt{m^n}}+
\sum_{\ell=1}^{n-1}\sum_{j=\ell-1}^{\ell}\left[
\tfrac{2\sqrt{p-1}}{\sqrt{m^{n-\ell}}}\left(
(T-t) c\left\lVert \tnorm{ {U}_{j,m}^{0}- u}{1,t+(T-t)\unif^{0}}\right\rVert_{\exponentLP}
+\tfrac{4e^{3cT}}{\sqrt{m^j}}\right)\right]\nonumber 
\\
&\qquad\qquad\qquad+\xeqref{a23}
 (T-t)c\left\lVert \tnorm{ {U}_{n-1,m}^{0}- u}{1,t+(T-t)\unif^{0}}\right\rVert_{\exponentLP}
+\tfrac{4e^{3c T}}{\sqrt{m^{n-1}}}
\Biggr]\nonumber \\
&\leq 
\sup_{t\in [s,T]}\left[
\tfrac{2\sqrt{\exponentLP-1}}{\sqrt{m^n}}+
\sum_{\ell=0}^{n-1}\left[
\tfrac{4\sqrt{p-1}}{\sqrt{m^{n-\ell-1}}}\left(
(T-t) c\left\lVert \tnorm{ {U}_{\ell,m}^{0}- u}{1,t+(T-t)\unif^{0}}\right\rVert_{\exponentLP} 
+\tfrac{4e^{3cT}}{\sqrt{m^\ell}}\right)\right]\right]\nonumber \\
&\leq 
\tfrac{2\sqrt{\exponentLP-1}}{\sqrt{m^n}}+
\sum_{\ell=0}^{n-1}\left[
\tfrac{4\sqrt{p-1}}{\sqrt{m^{n-\ell-1}}}\left(
(T-s)^{1-\tfrac{1}{\exponentLP}} c\left[\int_{s}^{T} \tnorm{ {U}_{\ell,m}^{0}- u}{1,\zeta}^{\exponentLP} d\zeta\right]^{1/\exponentLP} 
+\tfrac{4e^{3cT}}{\sqrt{m^\ell}}\right)\right]\nonumber \\
&\leq \tfrac{16mne^{3cT}\sqrt{p-1}}{\sqrt{m^n}}
+
\sum_{\ell=0}^{n-1}\left[
\tfrac{4(T-s)^{1-\tfrac{1}{\exponentLP}} c\sqrt{p-1}}{\sqrt{m^{n-\ell-1}}}
\left[\int_{s}^{T} \tnorm{ {U}_{\ell,m}^{0}- u}{1,\zeta}^{\exponentLP} d\zeta\right]^{1/\exponentLP} 
\right]
.
\end{align}
This shows \eqref{a02d}. 

Next, 
\cite[Lemma~3.11]{HJKN20} (applied
for every $s\in[0,T] $, $n,m\in\N $ with
$M\gets m$,
$N\gets n$, $\tau\gets s$, 
$a\gets 16mne^{3c(T-s)}\sqrt{p-1}$,
$b\gets 4(T-s)^{1-\frac{1}{\exponentLP}} c\sqrt{p-1}$,  
$(f_j)_{j\in\N_0}\gets 
([s,T]\ni t\mapsto \tnorm{{U}_{j,m}^{0}-u}{1,t}\in[0,\infty])_{j\in\N_0}
$ 
in the notation of \cite[Lemma~3.11]{HJKN20}), \eqref{a02d}, 
\eqref{t04}, 
the fact that $\forall\,m\in \N\colon U_{0,m}^0=0$,
and
the fact that
$\forall\,s\in[0,T]\colon 16mne^{3c(T-s)}+8c(T-s) e^{c(T-s)}\leq 16 mne^{3c(T-s)} (1+c(T-s))\leq 16 mne^{4c(T-s)} $  prove for all $m,n\in\N$, $s\in[0,T]$ that
\begin{align}
&\tnorm{{U}_{n,m}^{0}-u}{1,s} \leq  
\left(
16mne^{3c(T-s)}\sqrt{p-1}+
4(T-s)^{1-\frac{1}{\exponentLP}} c\sqrt{p-1}\cdot (T-s)^{\frac{1}{p}}\cdot \sup_{t\in [s,T]}\tnorm{u}{1,t}
\right)\nonumber \\&\qquad\qquad\qquad\qquad\cdot 
e^{m^{p/2}/p}m^{-n/2}
 \left(1+4(T-s)^{1-\frac{1}{\exponentLP}} c\sqrt{p-1}\cdot (T-s)^{\frac{1}{p}}\right)^{n-1}\nonumber \\
&\leq \sqrt{p-1}\left(16mne^{3c(T-s)}+8c(T-s)e^{c(T-s)}\right)e^{m^{p/2}/p}m^{-n/2} \left(\sqrt{p-1}(1+4c(T-s))\right)^{n-1}\nonumber \\
&\leq  16 mn(p-1)^{n/2}
e^{4c(T-s)n}
e^{m^{p/2}/p}m^{-n/2}.
\end{align}
This proves \eqref{a02e}. The proof of \cref{a02} is thus completed.
\end{proof}
The following \cref{c01} estimates the computational effort of the
MLP approximations in \cref{t27b}.
Since the exact computational effort is difficult  to define, we instead
assume that the computational effort satisfies the recursive inequality \cref{c02a},
which we now motivate.
We think of the parameter $a_1$ as an upper bound
for the effort to compute $f(w)$ or $g(x)$ for any $w\in\R$, $x\in\R^d$, 
we think of $a_2m^\ell$ as an upper bound for the  effort
to compute one realisation of $(\approximationX_{s,t}^{\ell,m,\theta,x},\approximationX_{s,t}^{\ell-1,m,\theta,x})$,  
and 
we think of
$
\FEU_{\ell,m}$ as an upper bound for the  effort
to compute one realisation $U_{\ell,m}^0(t,x)$.
We note that we approximate the foward diffusion by the Euler approximations
in \cref{c24}.

\begin{corollary}[Analysis of the computational effort]\label{c01}
Let $\lVert \cdot\rVert\colon \bigcup_{k,\ell\in\N}\R^{k\times \ell}\to[0,\infty)$ satisfy for all $k,\ell\in\N$, $s=(s_{ij})_{i\in[1,k]\cap\N,j\in [1,\ell]\cap\N}\in\R^{k\times \ell}$ that
$\lVert s\rVert^2=\sum_{i=1}^{k}\sum_{j=1}^{\ell}\lvert s_{ij}\rvert^2$,
let  $T,\delta \in (0,\infty)$, 
$c,\beta\in [1,\infty)$,
$p\in[2,\infty)$,
$d,a_1,a_2\in \N$,
  $  \Theta = \bigcup_{ n \in \N } \Z^n$,
$f \in C( \R,\R)$, 
$g\in C(\R^d,\R)$,
$u\in C^{1,2}([0,T]\times\R^d,\R)$, 
$\mu=(\mu_{i})_{i\in[1,d]\cap\Z}\in C^2(\R^d,\R^d)$,
$\sigma =(\sigma_{i,j})_{i,j\in[1,d]\cap\Z}\in C^2(\R^d,\R^{d\times d})$, 
assume for all
 $x,y\in\R^d$, 
 $w_1,w_2\in\R$, $t\in[0,T]$ that
\begin{align}
\lVert \mu(0)\rVert+\lVert\sigma(0)\rVert\leq b,\quad 
\max\{
\lVert \mu(x)-\mu(y)\rVert,
\lVert \sigma(x)-\sigma(y)\rVert
\}\leq c\lVert x-y\rVert,\label{c20}
\end{align}
\begin{align}\label{c21}
\lvert Tf(0)\rvert+
\lvert u(t,x)\rvert +
\lvert g(x)\rvert\leq \bigl[b^2+c^2\lVert x\rVert^2\bigr]^{\beta},
\end{align}
\begin{align}
\lvert
g(x)-g(y)\rvert\leq
\max\{b^2,\lVert x\rVert^{2\beta},\lVert y\rVert^{2\beta}\}\lVert x-y\rVert,\quad \lvert f(w_1)-f(w_2)\rvert\leq c\lvert w_1-w_2\rvert,\label{c23}
\end{align}
\begin{align}\begin{split}
&
(\tfrac{\partial }{\partial t}u)(t,x)+
\sum_{i=1}^{d}\left[\mu_{i}(x)(\tfrac{\partial}{\partial x_i}u)(t,x)\right] +
\tfrac{1}{2}\sum_{i,j,k=1}^{d}\left[(\tfrac{\partial^2}{\partial x_i\partial x_j}u)(t,x)
\sigma_{i,k}(t,x)
\sigma_{j,k}(t,x)\right]\\
&=-f(u(t,x)),\quad\text{and}\quad u(T,x)=g(x) ,
\end{split}\label{c25}\end{align}
let $(\Omega,\mathcal{F},\P, (\F_t)_{t\in [0,T]})$ be a filtered probability space which satisfies the usual conditions,
for every random variable 
$\mathfrak{X}\colon\Omega\to\R\cup\{-\infty,\infty\}$
 let $\lVert\mathfrak{X}\rVert_{r} \in [0,\infty]$, $r\in[1,\infty)$,
satisfy for all $r\in[1,\infty)$ that $\lVert\mathfrak{X}\rVert_{r}^r= \E[\lvert\mathfrak{X}\rvert^r]$,
let
$\unif^\theta\colon \Omega\to[0,1]$,
$\theta\in \Theta$, be i.i.d.\ random variables which satisfy for all
$t\in [0,1]$
that
$\P(\unif^0 \leq t)= t$,
let $W^{\theta} \colon [0,T]\to \R^d$,  $\theta\in\Theta$, be independent $(\F_t)_{t\in [0,T]}$-Brownian motions, assume that
$(\unif^{\theta})_{\theta\in\Theta} $ and
$(W^{\theta})_{\theta\in\Theta}$ are independent,
for every $n\in\N$,
$\theta\in\Theta$,
$s\in[0,T]$,
$x\in\R^d$
let
$(\approximationX^{n,\theta,x}_{s,t})_{t\in[s,T]}
\colon [s,T] \times\R^d\times \Omega\to\R^d$ satisfy for all 
$k\in [0,n-1]\cap\Z$,
$t\in\bigl[\max\{s,\frac{kT}{n}\},\max\{s,\frac{(k+1)T}{n}\}\bigr]$
 that
$\approximationX^{n,\theta,x}_{s,s}=x$ and
\begin{align}\label{c24}\begin{split}
\approximationX^{n,\theta,x}_{s,t}&=
\approximationX^{n,\theta,x}_{s,\max\{s,\frac{kT}{n}\}}+ \mu\bigl(
\approximationX^{n,\theta,x}_{s,\max\{s,\frac{kT}{n}\}} \bigr)\bigl(t-\max\{s,\tfrac{kT}{n}\}\bigr)
+
\sigma\bigl(
\approximationX^{n,\theta,x}_{s,\frac{kT}{n}} \bigr)\bigl(W^{\theta}_t-W^{\theta}_{\max\{s,\frac{kT}{n}\}}\bigr),\end{split}
\end{align}
let 
$ 
  {U}_{ n,m}^{\theta } \colon [0, T] \times \R^d \times \Omega \to \R
$,
$n,m\in\Z$, $\theta\in\Theta$, satisfy for all 
$n,m\in \N$, $\theta\in\Theta$, $t\in[0,T]$, $x\in\R^d$ that
$
{U}_{-1,m}^{\theta}(t,x)={U}_{0,m}^{\theta}(t,x)=0$ and
\begin{align}
\label{t27b}
&  {U}_{n,m}^{\theta}(t,x)
=  \sum_{\ell=0}^{n-1}\Biggl[\frac{1}{m^{n-\ell}}
    \sum_{i=1}^{m^{n-\ell}}
\biggl(
      g \bigl(\approximationX^{m^\ell,(\theta,\ell,i),x}_{t,T}\bigr)-\1_{\N}(\ell)
g \bigl(\approximationX^{m^{\ell-1},(\theta,\ell,i),x}_{t,T}\bigr)\nonumber \\
 &+(T-t)
     \bigl( f\circ {U}_{\ell,m}^{(\theta,\ell,i)}\bigr)
\left(t+(T-t)\unif^{(\theta,\ell,i)},\approximationX_{t,t+(T-t)\unif^{(\theta,\ell,i)}}^{m^\ell,(\theta,\ell,i),x}\right)\nonumber \\
&-\1_{\N}(\ell)(T-t)
\bigl(f\circ {U}_{\ell-1,m}^{(\theta,\ell,-i)}\bigr)
    \left(t+(T-t)\unif^{(\theta,\ell,i)},\approximationX_{t,t+(T-t)\unif^{(\theta,\ell,i)}}^{m^{\ell-1},(\theta,\ell,i),x}   \right) \biggr)\Biggr].
\end{align}
let   $
 ( {\FEU}_{ n,m})_{n,m\in\Z}\subseteq  \N_0
$
satisfy for all $ n,m\in\N$ that
\begin{align}
{\FEU}_{0,m}=0\quad\text{and}\quad {\FEU}_{n,m}\leq  \sum_{\ell=0}^{n-1}\left[
m^{n-\ell}\left[
4a_1+a_2 m^\ell+
     {\FEU}_{\ell,m}
+\1_{\N}(\ell){\FEU}_{\ell-1,m}\right]\right],\label{c02a}
\end{align}
and
let
$M\colon \N\to\N$, $C\in  [0,\infty]$  satisfy  that
\begin{equation}
\label{c02b}
\limsup_{n\to\infty}\tfrac{1}{M(n)}=0
,\quad \sup_{n\in\N}\left[\tfrac{M(n+1)}{M(n)}+\tfrac{(M(n))^{p/2}}{n}\right]<\infty,
\end{equation} and
\begin{equation}
 C=
\sup_{n\in\N}\tfrac{
(6  (M(n))^2 )^{2+\delta}
\left[7pe^{4cT} \left(\sup\limits_{k\in \N}\tfrac{M(k+1)}{M(k)}\right) 
\right]^{(n+1)(2+\delta)}
e^{(2+\delta)(M(n))^{p/2}/p}}{
(M(n))^{n\delta/2}}.\label{c05}
\end{equation}
Then
\begin{enumerate}[i)]\itemsep0pt
\item\label{c01d} it holds that $C<\infty$,
\item \label{c01e}
there exists  an up to indistinguishability unique continuous random field
$(X^{x}_{s,t})_{s\in[0,T],t\in[s,T],x\in\R^d}
\colon \{(\mathfrak{s},\mathfrak{t})\in [0,T]^2\colon \mathfrak{s}\leq \mathfrak{t}\} \times\R^d\times \Omega\to\R^d
$  
such that 
for all
$s\in[0,T]$, 
$x\in\R^d$ it holds that
$(X^{x}_{s,t})_{t\in [s,T]}$ is $ (\F_t)_{t\in [s,T]}$-adapted
and such that
for all 
$s\in[0,T]$,
$t\in[s,T]$, $x\in\R^d$ it holds a.s.\ that
\begin{align}
X^{x}_{s,t}= x+\int_{s}^{t} \mu(X^{x}_{s,r})\,dr+\int_{s}^{t} \sigma(X^{x}_{s,r})\,dW^{0}_r,
\end{align}
\item \label{c01b}
 for all $t\in[0,T]$, $x\in \R^d$ it holds that
$\E\bigl[\lvert
g(X_{t,T}^x)\rvert\bigr]+
\int_{t}^{T} \E\bigl[\lvert f( u(r,X_{t,r}^x))\rvert
\bigr]\,d r+\sup_{r\in[0,T],\xi\in\R^d}\frac{\lvert u(r,\xi)\rvert}{V(r,\xi)}<\infty$ and
$
u(t,x)=
\E\bigl[
g(X_{t,T}^x)\bigr]+
\int_{t}^{T} \E\bigl[f( u(r,X_{t,r}^x))
\bigr]\,dr,
$
\item \label{c03} it holds  for all $n,m\in\N$ that
$
{\FEU}_{n,m}\leq \max\{4 a_1,a_2n\}(5m)^n
$,
and
\item \label{c07}
there exists $\mathsf{n}\colon (0,1)\to\N$ such that for all $\epsilon,\delta\in(0,1)$ 
 it holds that
\begin{align}
\sup_{t\in [0,T],x\in[-1,1]^d}
\left\lVert U_{\mathsf{n}(\epsilon),M(\mathsf{n}(\epsilon))}^0(t,x)-u(t,x)\right\rVert_{\exponentLP}\leq \epsilon
\end{align} and
\begin{align}\small\begin{split}
\epsilon^{2+\delta}
{\FEU}_{\mathsf{n}(\epsilon),M(\mathsf{n}(\epsilon))}\leq  (a_1+a_2)C
\left( {5}c e^{c^2\left[
\sqrt{T}+16\beta p
\right]^2T}
\left[
\sqrt{T}+16\beta p
\right]^3
e^{288\beta p c^2T}4^\beta (b^2+c^2d)^\beta\right)^{3(\delta+2)}
.\end{split}
\end{align}
\end{enumerate}
\end{corollary}
\begin{proof}[Proof of \cref{c01}]
 First, \eqref{c05} and \eqref{c02b} 
prove \eqref{c01d}. 
Next,
a standard result on stochastic differential equations with Lipschitz continuous coefficients
(see, e.g., 
\cite[Theorem~4.5.1]{Kunita1990}) 
and \eqref{c20} show \eqref{c01e}.

Throughout the rest of this proof 
let $q,\bar{c}\in \R$ satisfy that $q=8\beta p$ and $\bar{c}= 
16q^2c^2$, let $\varphi\colon \R^d\to[1,\infty) $,
$V\colon [0,T]\times\R^d\to[1,\infty)  $
 satisfy 
for all 
$t\in[0,T]$,
$x\in\R^d$ that
\begin{align}
\varphi(x)= 2^{2q}\Bigl( b^2+ c^2\lVert x \rVert^2\Bigr)^{q}\label{c22}
\end{align}
and
\begin{align}
V(t,x)=\left[{5}c e^{c^2\left[
\sqrt{T}+2q
\right]^2T}
\left[
\sqrt{T}+2q
\right]^3
e^{\frac{1.5 \bar{c} T}{2q}}\right]
e^{\frac{1.5\bar{c}\beta(T-t)}{ q}}
(\varphi(x))^{\frac{\beta}{ q}}.\label{p24}
\end{align}
Observe that \eqref{c21}, \eqref{c22}, and \eqref{p24}
show for all $t\in[0,T]$, $x\in\R^d$ that
\begin{align}
\lvert Tf(0)\rvert+\lvert u(t,x)\rvert+
\lvert g(x)\rvert\leq\xeqref{c21} \bigl[b^2+c^2\lVert x\rVert^2\bigr]^{\beta}\leq\xeqref{c22} (\varphi(x))^{\frac{\beta}{q}}\leq\xeqref{p24} V(t,x).\label{c12}
\end{align}
Next, \eqref{c23}, the fact that $\beta,c\in [1,\infty)$, \eqref{c22}, and \eqref{p24} show for all $x\in\R^d$, $y\in \R^d\setminus\{x\}$ that
\begin{align}\begin{split}
&
\tfrac{\lVert g(x)-g(y)\rVert}{\lVert x-y\rVert}\leq \xeqref{c23}
\max\{b^2,\lVert x\rVert^{2\beta},\lVert y\rVert^{2\beta}\}\\
&\leq \tfrac{2(b^2+c^2\lVert x\rVert^2)^\beta
+2(b^2+c^2\lVert y\rVert^2)^\beta}{2}\leq\xeqref{c22} \tfrac{(\varphi(x))^{\frac{\beta}{q}}+(\varphi(y))^{\frac{\beta}{q}}}{2}\leq \xeqref{p24}
\tfrac{V(T,x)+V(T,y)}{2\sqrt{T}}.
\end{split}\label{c13}
\end{align}
Next, \eqref{c20}, the fact that $\forall\,A,B\in [0,\infty)\colon A+B\leq 2\sqrt{A^2+B^2}$, and \eqref{p24} show 
 for all
$x\in\R^d$  that 
\begin{align}
\lVert\mu(0)\rVert+\lVert\sigma(0)\rVert+c\lVert x\rVert\leq  \xeqref{c20}
b+c\lVert x\rVert
\leq
2\left[b^2+ c^2\lVert x \rVert^2\right]^{\frac{1}{2}}=  (\varphi(x))^{\frac{1}{2q}}.\label{p27}
\end{align}
Moreover, 
the fact that
$\forall\, x\in\R^d\colon 
\varphi(x)= \bigl(4b^2+ 4c^2\lVert x \rVert^2\bigr)^{q}
$ (see \eqref{p24}), the fact that 
$q\geq 3$, 
  \cite[Lemma~3.1]{HN21}
 (applied with
$p\gets q$, $a\gets 4 b^2$, $c\gets 2c$, $V\gets\varphi$
in the notation of 
\cite[Lemma~3.1]{HN21}), and the fact that
$\bar{c}=16q^2c^2$
show for all 
$x,y\in\R^d$ that
$
\bigl\lvert((\totalD \varphi)(x))(y)\bigr\rvert
\leq {4q}{c} (\varphi(x))^{\frac{2q-1}{2q}}\lVert y\rVert
\leq \bar{c} (\varphi(x))^{\frac{2q-1}{2q}}\lVert y\rVert
$ and $
\bigl\lvert ((\totalD^2 \varphi)(x))(y,y)\bigr\rvert
\leq 16q^2c^2(\varphi (x))^{\frac{2q-2}{2q}}\lVert y\rVert^2
=\bar{c}(\varphi (x))^{\frac{2q-2}{2q}}\lVert y\rVert^2
.
$
This, 
\eqref{p27},  \eqref{c20}, 
\eqref{s01},
and
\cite[Theorem~3.2]{HN21}
(applied with
$m\gets d$, $b\gets \infty$,
$p\gets 2q$,  $V\gets \varphi$
 in the notation of \cite[Theorem~3.2]{HN21}), 
the fact that $q\geq 4$,
Jensen's inequality,  the fact that
$1\leq  q/\beta \leq q $, and the fact that
$\forall\,t,s\in [0,T]\colon \lvert t-s\rvert^{\nicefrac{1}{2}}\leq \sqrt{T}$
show that
\begin{enumerate}[(I)]\itemsep0pt
\item it holds 
for all $n\in\N$,
 $s\in[0,T]$, $t\in [s,T]$,
 $x\in\R^d$
  that 
\begin{align}\label{r10}
\E\bigl[ \varphi(\approximationX_{s,t}^{n,0,x})\bigr]\leq e^{1.5\bar{c}\lvert t-s\rvert}\varphi(x)
,
\end{align}
\item  it holds for all
 $n\in\N$,
$s\in[0,T]$, 
$t\in [s,T]$,
$x\in\R^d$ that
\begin{align}\label{r11}
\left\lVert \left\lVert 
\approximationX_{s,t}^{n,0,x}
-X_{s,t}^{x}\right\rVert\right\rVert_{\frac{q}{\beta}}
\leq 
 \sqrt{2}
c\left[
\sqrt{T}+2q
\right]^3
e^{ c^2\left[
\sqrt{T}+2q
\right]^2 T}
(e^{1.5\bar{c} T}\varphi(x))^{\frac{1}{2q} }
 \tfrac{\sqrt{T}}{\sqrt{n}}
,
\end{align}
and
\item 
 it holds for all $n\in\N$,
$ s,\tilde{s}\in [0,T]$,
$t\in[s,T]$,
$\tilde{t}\in[\tilde{s},T]$,
 $x,\tilde{x}\in\R^d$ that
\begin{align}\label{r12}\begin{split}
&
\left \lVert \left \lVert X^{x}_{s,t}
-
X^{\tilde{x}}_{\tilde{s},\tilde{t}}\right\rVert\right\rVert_{\frac{q}{\beta}}\leq \sqrt{2}\lVert x-\tilde{x} \rVert 
e^{c^2\left[
\sqrt{T}+2q
\right]^2 T}\\
&+
{5} e^{c^2\left[
\sqrt{T}+2q
\right]^2T}
\left[
\sqrt{T}+2q
\right]
e^{\frac{1.5 \bar{c} T}{2q}}
\frac{(\varphi(x))^{\frac{1}{2q}}
+(\varphi(\tilde{x}))^{\frac{1}{2q} }}{2}
\left[
\lvert s-\tilde{s}\rvert^{\nicefrac{1}{2}} +
\lvert t-\tilde{t}\rvert^{\nicefrac{1}{2}}\right]
.
\end{split}\end{align}
\end{enumerate}
This, the fact that $\forall\,s\in[0,T],x\in\R^d\colon\P(X_{s,s}^x=x)=1$, \eqref{c22},
and \eqref{p24} show for all $s\in[0,T]$, $t\in[s,T]$, $x,y\in \R^d$ that
\begin{align}\begin{split}
&
\left \lVert\left \lVert X^{x}_{s,t}-x\right\rVert\right\rVert_{\frac{q}{\beta}}\leq 
\left \lVert \left \lVert X^{x}_{s,t}
-
X^{x}_{{s},{s}}\right\rVert\right\rVert_{\frac{q}{\beta}}\\
&
\leq 
{5} e^{c^2\left[
\sqrt{T}+2q
\right]^2T}
\left[
\sqrt{T}+2q
\right]
e^{\frac{1.5 \bar{c} T}{2q}}
(\varphi(x))^{\frac{1}{2q}}
\lvert t-s\rvert^{\nicefrac{1}{2}}\leq\xeqref{c22}\xeqref{p24} V (s,x)\lvert t-s\rvert^{\nicefrac{1}{2}}
,
\end{split}\label{m12}\end{align}
\begin{align}\begin{split}
&
\left \lVert X^{x}_{s,t}
-
X^{y}_{{s},{t}}\right\rVert_{\frac{q}{\beta}}\leq \sqrt{2}\lVert x-y \rVert 
e^{c^2\left[
\sqrt{T}+2q
\right]^2 T}\leq\xeqref{c22}\xeqref{p24}  \tfrac{V(s,x)+V(s,y)}{2}\lVert x-y\rVert,
\end{split}\label{r12d}\end{align}
and
\begin{align}
\left\lVert\left\lVert
\approximationX_{s,t}^{n,0,x}-X_{s,t}^{x}
\right\rVert
\right\rVert_{\frac{q}{\beta}}\leq V(s,x)\tfrac{\sqrt{T}}{\sqrt{n}}.\label{r13}
\end{align}
Next, 
\eqref{r10} shows
for all $n\in\N$, $ x\in\R^d$, $s\in[0,T]$, $t\in[s,T]$ that
\begin{align}\begin{split}
&
\left
\lVert
e^{\frac{1.5\bar{c}\beta(T-t)}{ q}}
(\varphi (\approximationX_{s,t}^{ n,0,x}))^{\frac{\beta}{q}}\right\rVert_{\frac{q}{\beta}}=e^{\frac{1.5\bar{c}\beta(T-t)}{ q}} \bigl(\E\bigl[ \varphi (\approximationX_{s,t}^{ n,0,x})\bigr]\bigr)^{\frac{\beta}{q}}\\
&\leq e^{\frac{1.5\bar{c}\beta(T-t)}{ q}} e^{\frac{1.5\beta\bar{c}\lvert t-s\rvert}{q}}(\varphi (x))^{\frac{\beta}{q}}=e^{\frac{1.5\bar{c}\beta(T-s)}{ q}} (\varphi (x))^{\frac{\beta}{q}}.
\end{split}\end{align}
This,  \eqref{c22}, and \eqref{p24} show
for all $n\in\N$, $ x\in\R^d$, $s\in[0,T]$, $t\in[s,T]$ that
$
\bigl\lVert V(t,\approximationX_{s,t}^{n,0,x})\bigr\rVert_{\frac{q}{\beta}}
\leq V(s,x).
$
This, \eqref{r11},  continuity of $V$, and Fatou's lemma show for all
$n\in\N$, $ x\in\R^d$, $s\in[0,T]$, $t\in[s,T]$ that
$
X_{s,t}^{x}=
\operatornamewithlimits{\P\text{-}\lim}_{n\to\infty}
\approximationX_{s,t}^{n,0,x}
$
and
$
\bigl\lVert V\bigl(t,X_{s,t}^{x}\bigr)\bigr\rVert_{\frac{q}{\beta}}
=
\left\lVert
\operatornamewithlimits{\P\text{-}\lim}_{n\to\infty} V\bigl(t,\approximationX_{s,t}^{n,0,x}\bigr)\right\rVert_{\frac{q}{\beta}}
\leq \liminf_{n\to\infty}
\bigl\lVert V\bigl(t,\approximationX_{s,t}^{n,0,x}\bigr)\bigr\rVert_{\frac{q}{\beta}}\leq V(s,x).$ This, \eqref{c12}, and \eqref{c23} show for all $t\in[0,T] $, $x\in\R^d$ that
$\E\bigl[\lvert
g(X_{t,T}^{x})\rvert\bigr]+
\int_{t}^{T} \E\bigl[\lvert f( u(r,X_{t,r}^{x}))\rvert
\bigr]\,d r<\infty$. This, \eqref{c25}, \eqref{c01e}, and the Feynman-Kac formula imply~\eqref{c01b}.

Now,
\cref{a02} (applied with $\exponentV\gets \frac{q}{\beta}$, $\exponentX\gets \frac{q}{\beta}$, $\expFirstNorm\gets 3$), 
the fact that $\frac{6}{q/\beta}+\frac{2}{q/\beta}= \frac{8\beta}{q}= 
\frac{8\beta}{8\beta \exponentLP}
\leq \frac{1}{\exponentLP}$,
the fact that
$3\exponentLP \leq \frac{8p\beta}{\beta}= \frac{q}{\beta}$, the fact that $\frac{2}{q/\beta}+\frac{1}{q/\beta}=
\frac{3\beta}{q}= \frac{3\beta}{8\beta p}\leq 
\frac{1}{\exponentLP} $
(recall that
$q=3\beta p$), \eqref{c12}, \eqref{c13}, \eqref{c23},
 the Markov property,
\eqref{m12}, \eqref{r12d}, \eqref{r13},  the assumptions of \cref{c01},
 the fact that 
$1\leq (p-1)^{\nicefrac{1}{2}}e^{4cT}\leq pe^{4cT}$, and
 the fact that $\forall\, n\in\N\colon 4n\leq 2^{n+1}$ imply
for all $n,m\in\N$, $t\in[0,T]$, $x\in\R^d$  that
\begin{align}\label{c02}
&\tfrac{\left\lVert U_{n,m}^0(t,x)-u(t,x)\right\rVert_{\exponentLP}}{(V(t,x))^{3}}\leq 
\tfrac{ 16 mn(p-1)^{n/2}
e^{4cTn}
e^{m^{p/2}/p}}{m^{n/2}}\leq 
\tfrac{4 m \left(2 p e^{4cT}\right)^{n+1} e^{m^{p/2}/p}}{m^{n/2}}.
\end{align}
Next,
\eqref{c02a} shows for all $ n,m\in\N$ that
\begin{align}\begin{split}
{\FEU}_{n,m}&\leq 
 \sum_{\ell=0}^{n-1}\left[
m^{n-\ell}\left[
4a_1+a_2 m^\ell+
     {\FEU}_{\ell,m}
+\1_{\N}(\ell){\FEU}_{\ell-1,m}\right]\right] \\
&= 
a_2 nm^n+
 \sum_{\ell=0}^{n-1} \left[
m^{n-\ell}\left[4a_1+
     {\FEU}_{\ell,m}
+\1_{\N}(\ell){\FEU}_{\ell-1,m}\right]\right] \\
&\leq 
\max\{4 a_1,a_2n\}m^n+
 \sum_{\ell=0}^{n-1} \left[
m^{n-\ell}\left[\max\{4 a_1,a_2n\}+1+
     {\FEU}_{\ell,m}
+\1_{\N}(\ell){\FEU}_{\ell-1,m}\right]\right]
. 
\end{split}\end{align}
This and \cite[Lemma~3.6]{HJKNW2018} (applied for every $n\in \N$ with $d\gets \max\{4 a_1,a_2n\}$, $ (\mathrm{RV}_{i,M})_{i,M\in\Z}\gets 
({\FEU}_{\min\{i,n\},M})_{i,M\in\Z}
$ in the notation of \cite[Lemma~3.6]{HJKNW2018}) show for all $n,m\in\N$ that
$
{\FEU}_{n,m}\leq \max\{4 a_1,a_2n\}(5m)^n.
$ This shows \eqref{c03}.

Throughout the rest of this proof
let
$(\mathsf{n}(\epsilon))_{ \epsilon\in(0,1)}\subseteq [0,\infty]$
 satisfy for all   $\epsilon\in (0,1)$ that
\begin{align}
\mathsf{n}(\epsilon)=\inf\left(\left \{ n\in \N \colon
\sup_{k\in [n,\infty)\cap\N}
 \sup_{ t\in[0,T],x\in[-1,1]^d } \left\lVert U_{k,M(k)}^0(t,x)-u(t,x)\right\rVert_{\exponentLP}\leq \epsilon \right\}\cup\{\infty\}\right).
\label{c02c}
\end{align}
Observe that \eqref{c02},  \eqref{c02b}, and continuity of $V$ prove for all 
 $\epsilon\in(0,1)$ that
\begin{align}\begin{split}
&
\limsup_{n\to\infty}\sup_{\substack{t\in[0,T], x\in[-1,1]^d}}
\left\lVert U_{n,M(n)}^0(t,x)-u(t,x)\right\rVert_{\exponentLP}\\
&\leq \xeqref{c02}
\limsup_{n\to\infty}\sup_{\substack{t\in [0,T],  x\in[-1,1]^d}}
\tfrac{4 (M(n)) \left(2 p e^{4cT}\right)^{n+1} e^{(M(n))^{p/2}/p}(V(t,x))^{3}}{(M(n))^{n/2}}=\xeqref{c02b}0.
\end{split}\end{align} This and \eqref{c02c} show that $ \mathsf{n}(\epsilon)\in\N 
$. Next, the definitions of $q, \bar{c}$,
 \eqref{c22}, and \eqref{p24} show that
$e^{\frac{1.5 \bar{c} T}{2q}}
e^{\frac{1.5\bar{c}\beta T}{ q}}=
e^{\frac{4.5\bar{c}T}{2q}}=e^{\frac{4.5\cdot 16q^2c^2T}{2q}}=e^{36qc^2T}= 
e^{36\cdot 8\beta p c^2T}= 
e^{288\beta p c^2T}
$,
$\sup_{x\in [-1,1]^d}
\varphi(x)= 2^{2q}\Bigl( b^2+ c^2\lVert x \rVert^2\Bigr)^{q}= 4^q(b+c^2d)^q
$, 
$\sup_{x\in [-1,1]^d}(\varphi(x))^\beta= 4^\beta(b+c^2d)^\beta$,
and
\begin{align}\begin{split}
&\sup_{t\in[0,T],x\in [-1,1]^d}
V(t,x)= {5}c e^{c^2\left[
\sqrt{T}+2q
\right]^2T}
\left[
\sqrt{T}+2q
\right]^3
e^{\frac{1.5 \bar{c} T}{2q}}
e^{\frac{1.5\bar{c}\beta T}{ q}}
(\varphi(x))^{\frac{\beta}{ q}}\\
&
=  {5}c e^{c^2\left[
\sqrt{T}+16\beta p
\right]^2T}
\left[
\sqrt{T}+16\beta p
\right]^3
e^{288\beta p c^2T}4^\beta (b^2+c^2d)^\beta.\end{split}\label{c26}
\end{align}
Moreover, \eqref{c03}, \eqref{c02b}, and the fact that
$\forall\,n\in\N\colon n+1\leq 2^n$ prove for all $n\in\N$ that
\begin{align}\begin{split}
{\FEU}_{n+1,M(n+1)}&\leq \max\{4 a_1,a_2(n+1)\}(5M(n+1))^{n+1} \\
&
\leq 
4(a_1+a_2)
2^n\left[5\left(\sup\limits_{k\in \N}\tfrac{M(k+1)}{M(k)}\right)  M(n)\right]^{n+1}
.\label{c03e}\end{split}
\end{align}
This, 
\eqref{c02},
 the fact that
$4\sqrt{2}\leq 6$,
 the fact that $2\sqrt{10}\leq 7 $, and \eqref{c05}
show for all  $n\in\N$,  $t\in[0,T]$, $x\in\R^d$ that
\begin{align}
&\left[
\tfrac{\bigl\lVert 
{U}_{n,M(n)}^{0}(t,x)-u(t,x)\bigr\rVert_2}{(V(t,x))^{3}}
\right] ^{2+\delta}\max\{1,{\FEU}_{n+1,M(n+1)}\}\nonumber \\
&\leq \left[\tfrac{4m \left(2pe^{4cT}\right)^{n+1} e^{m^{\exponentLP/2}/p}}{m^{n/2}}\Bigr|_{m=M(n)}\right]^{2+\delta}
2(a_1+a_2)
2^{n+1}\left[5 \left(\sup\limits_{k\in \N}\tfrac{M(k+1)}{M(k)}\right)  M(n)\right]^{n+1}\nonumber \\
&
\leq (a_1+a_2) \!\left[
\tfrac{
(4m)^{2+\delta}
\left(2p e^{4cT}\right)^{n+1)(2+\delta)}e^{(2+\delta)m^{p/2}/p}
}{
 m^{n(1+\tfrac{\delta}{2})}}2
\left[10 \left(\sup\limits_{k\in \N}\tfrac{M(k+1)}{M(k)}\right)  m\right]^{n+1}\right]\Bigr|_{m=M(n)}
\nonumber  \\
&\leq (a_1+a_2)\tfrac{
(6  m^2 )^{2+\delta}
\left[7p
e^{4cT} \left(\sup\limits_{k\in \N}\tfrac{M(k+1)}{M(k)}\right)  \right]^{(n+1)(2+\delta)}
e^{(2+\delta)m^{p/2}/p}}{
m^{n\delta/2}}\Bigr|_{m=M(n)}\leq (a_1+a_2)C.
\label{c04}
\end{align}
This and \eqref{c03}
 show for all $\epsilon\in(0,1)$ that 
 in the case $\mathsf{n}(\epsilon)=1 $ it holds that
$\epsilon^{2+\delta}
{\FEU}_{\mathsf{n}(\epsilon),M(\mathsf{n}(\epsilon))} \leq 
{\FEU}_{1,1}\leq 20(a_1+a_2)M(1)
$ and
in the case $\mathsf{n}(\epsilon)\in \N\cap [2,\infty)$ it holds that
\begin{align}\begin{split}
&
\epsilon^{2+\delta}
{\FEU}_{\mathsf{n}(\epsilon),M(\mathsf{n}(\epsilon))}\leq 
\sup_{t\in[0,T],x\in[-1,1]^d}\left[
\left\lVert 
{U}_{n,M(n)}^{0}(t,x)-u(t,x)\right\rVert_2 ^{2+\delta}{\FEU}_{n+1,M(n+1)}\right]\bigr|_{n=\mathsf{n}(\epsilon)-1} \\
&\leq  (a_1+a_2)C\sup_{t\in[0,T], x\in[-1,1]^d}(V(t,x))^{3(\delta+2)}
.
\end{split}\end{align}
This, the fact that $ 20M(1)\leq C $, and the fact that $1\leq V$ prove 
for all $\epsilon\in(0,1)$, $x\in\R^d$ that 
$
\epsilon^{2+\delta}
{\FEU}_{\mathsf{n}(\epsilon),M(\mathsf{n}(\epsilon))}\leq  (a_1+a_2)C\sup_{t\in[0,T], x\in[-1,1]^d}(V(t,x))^{\expFirstNorm(\delta+2)}
$. This, \eqref{c02c},  the fact that $\forall\, \epsilon\in (0,1)\colon \mathsf{n}(\epsilon)\in\N$, and \eqref{c26} 
imply \eqref{c07}.
The proof of \cref{c01} is thus completed.
\end{proof}

\section{Error estimates for MLP approximations in temporal-spatial H\"older-norms}
\label{sec4}
In this section we prove strong convergence rates of MLP approximations
in temporal-spatial H\"older norms; see \Cref{p05} below.
Our H\"older-norms are defined in \cref{h19} below.
It turned out that it is advantageous to weight differences of time points
and differences of space points with different monomials of the Lyapunov-type function;
see the denominator of \cref{h19}. With this choice of H\"older-norm
we succeded to derive the closed recusion \cref{h37}.

\begin{setting}\label{b01}

Assume \cref{s12}, let $\Theta=\cup_{n\in\N}\Z^n$,
let
$\unif^\theta\colon \Omega\to[0,1]$,
$\theta\in \Theta$, be i.i.d.\ random variables which satisfy for all
 $t\in [0,1]$
that
$\P(\unif^0 \leq t)= t$,
let
$(\approximationX^{n,\theta,x}_{s,t})_{s\in[0,T],t\in[s,T],x\in\R^d}
\colon \{(\mathfrak{s},\mathfrak{t})\in [0,T]^2\colon \mathfrak{s}\leq \mathfrak{t}\} \times\R^d\times \Omega\to\R^d$, 
$\theta\in\Theta$,
$n\in\N$, be measurable,
assume  that 
$
\bigl(\approximationX^{n,\theta,x}_{s,t}, \approximationX^{\tilde{n},\theta,\tilde{x}}_{\tilde{s},\tilde{t}}\bigr)_{s,\tilde{s}\in[0,T], t\in[s,T],\tilde{t}\in[\tilde{s},T],n,\tilde{n}\in\N,x,\tilde{x}\in\R^d}$,
$\theta\in\Theta$,
are i.i.d.\ random fields,
assume  for all 
$n\in\N$,
$s\in [0,T]$, $\tilde{s},t\in[s,T]$, 
$\tilde{t}\in [\tilde{s},T]$, 
$x,\tilde{x}\in \R^d$
%
that
$((X^{x}_{s,t},\approximationX^{n,\theta,x}_{s,t}))_{\theta\in\Theta}$ and
$(\unif^\theta)_{\theta \in\Theta}$ are independent 
and that
\begin{equation}\begin{split}
&
\left\lVert 
\left\lVert 
\approximationX_{s,t}^{n,0,x}-
\approximationX_{\tilde{s},\tilde{t}}^{n,0,\tilde{x}}\right\rVert\right\rVert_{\exponentX}\leq
\tfrac{V(s,x)+V(\tilde{s},\tilde{x})}{2}
\tfrac{\lvert s-\tilde{s}\rvert^{\nicefrac{1}{2}}+\lvert t-\tilde{t}\rvert^{\nicefrac{1}{2}}}{2}
+
e^{cT}
\lVert x-\tilde{x}\rVert,\label{h21}
\end{split}\end{equation}
\begin{equation}
\begin{split}\label{h21b}
&\left\lVert \left\lVert 
 \bigl( \approximationX^{n,0,x}_{s,t} -
X^{x}_{s,t}\bigr)
-   \bigl(\approximationX^{n,0,\tilde{x}}_{\tilde{s},\tilde{t}} - X^{\tilde{x}}_{\tilde{s},\tilde{t}}\bigr)\right\rVert\right\rVert_{\exponentX}
\leq 
\tfrac{V(s,x)+V(\tilde{s},\tilde{x})}{2}\left[\tfrac{\lvert s-\tilde{s}\rvert^{\nicefrac{1}{2}} 
+\lvert t-\tilde{t}\rvert^{\nicefrac{1}{2}}}{2}+
 \lVert x-\tilde{x}\rVert\right]\tfrac{1}{\sqrt{n}},
\end{split}
\end{equation}
\begin{equation}\label{s10}
\P(X^x_{s,s}=x)=
\P\!\left(\approximationX^{n,0,x}_{s,s}=x\right)=1,
\quad\text{and}\quad
e^{cT}\leq 
\bigl\lVert V\bigl(t,\approximationX_{s,t}^{n,0,x}\bigr)\bigr\rVert_{\exponentV}
\leq V(s,x),
\end{equation}
and
let 
$ 
  {U}_{ n,m}^{\theta } \colon [0, T] \times \R^d \times \Omega \to \R
$,
$n,m\in\Z$, $\theta\in\Theta$, satisfy for all 
$n,m\in \N$, $\theta\in\Theta$, $t\in[0,T]$, $x\in\R^d$ that
$
{U}_{-1,m}^{\theta}(t,x)={U}_{0,m}^{\theta}(t,x)=0$ and
\begin{align}
\label{t27c}
  {U}_{n,m}^{\theta}(t,x)
=  \sum_{\ell=0}^{n-1}\frac{1}{m^{n-\ell}}
    &\sum_{i=1}^{m^{n-\ell}}
\Biggl[
      g \bigl(\approximationX^{m^\ell,(\theta,\ell,i),x}_{t,T}\bigr)-\1_{\N}(\ell)
g \bigl(\approximationX^{m^{\ell-1},(\theta,\ell,i),x}_{t,T}\bigr)\nonumber \\
 &+(T-t)
     \bigl( f\circ {U}_{\ell,m}^{(\theta,\ell,i)}\bigr)
\left(t+(T-t)\unif^{(\theta,\ell,i)},\approximationX_{t,t+(T-t)\unif^{(\theta,\ell,i)}}^{m^\ell,(\theta,\ell,i),x}\right)\nonumber \\
&-\1_{\N}(\ell)(T-t)
\bigl(f\circ {U}_{\ell-1,m}^{(\theta,\ell,-i)}\bigr)
    \left(t+(T-t)\unif^{(\theta,\ell,i)},\approximationX_{t,t+(T-t)\unif^{(\theta,\ell,i)}}^{m^{\ell-1},(\theta,\ell,i),x}   \right) \Biggr].
\end{align}
\end{setting}

\begin{theorem}\label{p05}Assume~\cref{b01},
let $\exponentLP\in[2,\infty)$,
$\expFirstNorm\in[3,\infty)$, $\expSecondNorm\in [9,\infty)$
satisfy that 
$\frac{\expSecondNorm+1}{\exponentV}+\frac{2}{\exponentX}\leq \frac{1}{\exponentLP} $
and
$\expFirstNorm+2\leq \expSecondNorm$,
and
for every $s\in[0,T]$  and  
every
random field $H\colon [0,T]\times\R^d \times\Omega\to \R$ 
let 
$\tnorm{H}{s},\tnorm{H}{1,s},\tnorm{H}{2,s}\in[0,\infty]$  
$ $
 satisfy that
\begin{align}\begin{split}
\tnorm{H}{s}&=\max \left\{\tnorm{H}{1,s},\tnorm{H}{2,s}\right\},\quad 
\tnorm{H}{1,s}=  \sup_{t\in[s,T]}
\sup_{x\in\R^d}\frac{\lVert H(t,x)\rVert_{\exponentLP}}{ 
(V(t,x))^{\expFirstNorm}}, \quad\text{and} 
\\
\tnorm{H}{2,s}&=\sup_{
\substack{
t_1,t_2\in[s,T],
x_1,x_2\in\R^d\colon \\(t_1,x_1)\neq (t_2,x_2)
} } \frac{\left\lVert H(t_1,x_1)-H(t_2,x_2)\right\rVert_{\exponentLP}}{
\tfrac{(V(t_1,x_1))^{\expSecondNorm}+(V(t_2,x_2))^{\expSecondNorm}}{2}
\left[
\tfrac{V(t_1,x_1)+V(t_2,x_2)}{2}
\tfrac{ \lvert t_1-t_2\rvert^{\nicefrac{1}{2}}}{\sqrt{T}}
+\tfrac{\lVert x_1-x_2\rVert}{\sqrt{T}}
\right]}.
\label{h19}\end{split}
\end{align}
Then
\begin{enumerate}[(i)]\itemsep0pt
\item \label{p01}there exists a unique measurable  $u\colon [0,T]\times\R^d\to\R$ which satisfies for all $t\in[0,T]$, $x\in \R^d$ that
$\E\bigl[|
g(X_{t,T}^{x})|\bigr]+
\int_{t}^{T} \E\bigl[|f( u(r,X_{t,r}^{x}))|
\bigr]\,d r+\sup_{r\in[0,T],\xi\in\R^d}\frac{|u(r,\xi)|}{V(r,\xi)}<\infty$ and
$
u(t,x)=
\E\bigl[
g(X_{t,T}^{x})\bigr]+
\int_{t}^{T} \E\bigl[f( u(r,X_{t,r}^{x}))
\bigr]\,dr
$, and
\item \label{p01a}
 it holds for all $s\in[0,T]$ that
$\tnorm{u}{1,s}\leq 2e^{cT}$ and $\tnorm{u}{2,s}\leq 8 e^{2cT}$,
\item\label{p04} 
it holds
for all $n,m\in\N$, 
$t\in[0,T]$ that
\begin{align}\label{h19a}
\tnorm{U^0_{n,m}-u}{t}\leq  1776 m n e^{6cT}
 e^{m^{p/2}/p}m^{-n/2}37^n e^{4 n cT}
(\exponentLP-1)^{\frac{n}{2}}.
\end{align}
\end{enumerate}
\end{theorem}

\begin{proof}[Proof of \cref{p05}]
Throughout this proof  for every random variable $\mathfrak{X}\colon \Omega\to\R$ with $\E[|\mathfrak{X}|]<\infty$ let 
$\var_{\exponentLP}(\mathfrak{X})\in[0,\infty]$ satisfy that
$\var_{\exponentLP}(\mathfrak{X})= \E [|\mathfrak{X}-\E [\mathfrak{X}]|^{\exponentLP}]$. 
First,
\eqref{l01}, \eqref{l02},  
\eqref{h21b}, and \eqref{s10}
show for all 
$s\in [0,T]$, $t\in [s,T]$,
$x,y\in\R^d$, $v,w\in\R$, $n\in\N$ that
\begin{equation}\begin{split}
&
\lvert g(x)-g(y)\rvert=
\tfrac{\left\lvert (g(x)-g(y)) -(g(x)-g(x))
\right\rvert}{2}
+
\tfrac{\left\lvert (g(y)-g(x)) -(g(y)-g(y))
\right\rvert}{2}\\
&\leq\xeqref{l01}\tfrac{1}{2} \tfrac{3V(T,x)+V(T,y)}{4}
\tfrac{\lVert x-y\rVert}{\sqrt{T}}
+ \tfrac{1}{2}
\tfrac{3V(T,y)+V(T,x)}{4}
\tfrac{\lVert y-x\rVert}{\sqrt{T}}= \tfrac{V(T,x)+V(T,y)}{2}\tfrac{\lVert x-y\rVert}{\sqrt{T}},
\end{split}\label{s01c}\end{equation}
\begin{equation}
\lvert f(v)-f(w)\rvert =\xeqref{l02}
\lvert (f(v)-f(w))-(f(v)-f(v))\rvert\leq c \lvert v-w\rvert,\label{s02b}
\end{equation}
and
\begin{equation}
\begin{split}
&
\left\lVert \left\lVert 
 \approximationX^{n,0,x}_{s,t} -
X^{x}_{s,t}\right\rVert\right\rVert_{\exponentX}
=
\left\lVert \left\lVert 
 \bigl( \approximationX^{n,0,x}_{s,t} -
X^{x}_{s,t}\bigr)
-   \bigl(\approximationX^{n,0,x}_{s,s} - X^{x}_{s,s}\bigr)\right\rVert\right\rVert_{\exponentX}\\
&
\leq \xeqref{h21b}
V(s,x)\tfrac{\lvert t-s\rvert^{\nicefrac{1}{2}} 
}{2}
\tfrac{1}{\sqrt{n}}\leq \tfrac{\sqrt{T}V(s,x)}{2\sqrt{n}}.
\end{split}\label{s11}
\end{equation}
%
This,  
the triangle inequality,
\eqref{h21}, 
and the fact that $e^{cT}\leq V$ (see \eqref{s10})
prove for all $s_1,s_2\in [0,T]$, 
$t_1\in [s_1,T]$,
$t_2\in [s_2,T]$, $x_1,x_2\in \R^d$ that
\begin{align}\begin{split}
&\left\lVert \left\lVert 
X_{s_1,t_1}^{x_1}-
X_{s_2,t_2}^{x_2}\right\rVert\right\rVert_{\exponentX}\\
&
\leq
 \limsup_{ n\to\infty}\left[
\left\lVert \left\lVert 
X_{s_1,t_1}^{x_1}-
\approximationX_{s_1,t_1}^{n,0,x_1}
\right\rVert
\right\rVert_{\exponentX}
+\left\lVert \left\lVert 
\approximationX_{s_1,t_1}^{n,0,x_1}-
\approximationX_{s_2,t_2}^{n,0,x_2}\right\rVert\right\rVert_{\exponentX}
+\left\lVert\left\lVert \approximationX_{s_2,t_2}^{n,0,x_2}-
X_{s_2,t_2}^{x_2}\right\rVert\right\rVert_{\exponentX}\right] \\
&\leq \limsup_{n\to\infty}\left[\xeqref{s11}
\tfrac{\sqrt{T}V(s_1,x_1)}{2\sqrt{n}}+
\xeqref{h21}
\tfrac{V(s_1,x_1)+V(s_2,x_2)}{2}
\tfrac{\lvert s_1-s_2\rvert^{\nicefrac{1}{2}}+\lvert t_1-t_2\rvert^{\nicefrac{1}{2}}}{2}+ e^{cT}
\lVert x_1-x_2\rVert +\xeqref{s11}\tfrac{\sqrt{T}V(s_2,x_2)}{2\sqrt{n}}\right]\\
&\leq \xeqref{s10}
\tfrac{V(s_1,x_1)+V(s_2,x_2)}{2}
\left[
\tfrac{\lvert s_1-s_2\rvert^{\nicefrac{1}{2}}+\lvert t_1-t_2\rvert^{\nicefrac{1}{2}}}{2}
+\lVert x_1-x_2\rVert
\right]
.\label{h21c}\end{split}
\end{align}
This and \eqref{s10} show
for all 
$s\in[0,T]$,
$t\in [s,T]$, $x\in \R^d$ that 
\begin{align}
\left\lVert \left\lVert 
X_{s,t}^{x}-
x\right\rVert\right\rVert_{\exponentX}=
\left\lVert \left\lVert 
X_{s,t}^{x}-
X_{s,s}^{x}\right\rVert\right\rVert_{\exponentX}
\leq V(s,x)\lvert t-s\rvert^{\nicefrac{1}{2}}.\label{h21d}
\end{align}
Next, \eqref{s11},  continuity of $V$,
 Fatou's lemma, 
 and \eqref{s10} yield for all 
$s\in[0,T]$,
$t\in [s,T]$, $x\in \R^d$ that 
$ \operatornamewithlimits{\P\text{-}\lim_{n\to\infty}}\approximationX_{s,t}^{n,0,x}=\xeqref{s11}X_{s,t}^{x}$
 and
\begin{align}
\left \lVert V(t,X_{s,t}^{x})\right\rVert_{\exponentV}
= 
\left \lVert\operatornamewithlimits{\P\text{-}\lim}_{n\to\infty} V(t,\approximationX_{s,t}^{n,0,x})\right\rVert_{\exponentV}\leq 
 \liminf_{n\to\infty}
\left \lVert V(t,\approximationX_{s,t}^{n,0,x})\right\rVert_{\exponentV}\leq V(s,x).
\label{t21b}
\end{align}
This, \eqref{h21d}, the assumptions of \cref{p05},
and
\cref{b02}
 prove that
\begin{enumerate}[(a)]\itemsep0pt
 \item \label{p01b}there exists a unique measurable  $u\colon [0,T]\times\R^d\to\R$ which satisfies for all $t\in[0,T]$, $x\in \R^d$ that
$\E\bigl[|
g(X_{t,T}^{x})|\bigr]+
\int_{t}^{T} \E\bigl[|f( u(r,X_{t,r}^{x}))|
\bigr]\,d r+\sup_{r\in[0,T],\xi\in\R^d}\frac{|u(r,\xi)|}{V(r,\xi)}<\infty$ and
$
u(t,x)=
\E\bigl[
g(X_{t,T}^{x})\bigr]+
\int_{t}^{T} \E\bigl[f( u(r,X_{t,r}^{x}))
\bigr]\,dr
$, 
\item it holds for all $s\in[0,T]$,
$t\in[s,T]$, $x,y\in\R^d$  that
\begin{align}\label{p02b}\begin{split}
&\lvert u(s,x)-u(t,y)\rvert\leq 4 e^{2 c T}
\left(\tfrac{V(s,x)+V(t,y)}{2}\right)^2
\tfrac{V(s,x)\lvert t-s\rvert ^{\nicefrac{1}{2}}+
\lVert x-y\rVert}{\sqrt{T}}\quad\text{and}\quad\\&
\lvert u(t,x)\rvert\leq 2e^{c(T-t)}V(t,x)
,\end{split}
\end{align}
and
\item it holds for all $s\in[0,T]$,
$t\in[s,T]$, $x,y,\tilde{x},\tilde{y}\in\R^d$ that
\begin{align}\label{p03}
&\left\lvert \bigl[
f(
u(s,x))-f(u(s,y))\bigr]-\bigl[f(u(t,\tilde{x}))-f(u(t,\tilde{y}))\bigr]
\right\rvert\leq 896 ce^{6cT}
\left(\tfrac{V(s,x)+V(s,y)+V(t,\tilde{x})+V(t,\tilde{y})}{4}\right)^7 \nonumber \\
&\quad \cdot 
\left(
\tfrac{\lVert x-y-(\tilde{x}-\tilde{y})\rVert}{ \sqrt{T} }+
\tfrac{\lVert x-y\rVert+\lVert\tilde{x}-\tilde{y}\rVert}{\sqrt{T}}
\left[2
\tfrac{V(s,x)+V(s,y)}{2}\tfrac{\lvert t-s\rvert^{\nicefrac{1}{2}}}{\sqrt{T}}+\tfrac{\lVert x-\tilde{x}\rVert}{\sqrt{T}}\right]\right).
\end{align}
\end{enumerate}
This implies \eqref{p01}. 

Next, 
\eqref{p02b} show for all $t\in [0,T]$, $x\in\R^d$ that
$\lvert u(t,x)\rvert\leq 2e^{cT}V(t,x)\leq 2e^{cT}(V(t,x))^{\expFirstNorm}$. This and \eqref{h19} imply for all $t\in[0,T]$ that $\tnorm{u}{1,t}\leq 2e^{cT}$.
Next, \eqref{p02b},   Jensen's inequality,
and the fact that
$V^2\leq V^{\expSecondNorm}$
 imply for all 
$s\in[0,T]$,
$t_1\in[s,T]$, $t_2\in[t_1,T]$,
$x_1,x_2\in\R^d$ that 
\begin{align}\begin{split}
&
\lvert u(t_1,x_1)-u(t_2,x_2)\rvert\leq \xeqref{p02b}
4e^{2cT} \left(\tfrac{V(t_1,x_1)+V(t_2,x_2)}{2}\right)^2\left[
V(t_1,x_1)\tfrac{\lvert t_1-t_2\rvert^{\nicefrac{1}{2}}}{\sqrt{T}}+\tfrac{\lVert x_1-x_2\rVert}{\sqrt{T}}\right]\\
&\leq 
8e^{2cT} \left(\tfrac{V(t_1,x_1)+V(t_2,x_2)}{2}\right)^2\left[
\tfrac{V(t_1,x_1)+V(t_2,x_2)}{2}\tfrac{\lvert t_1-t_2\rvert^{\nicefrac{1}{2}}}{\sqrt{T}}+\tfrac{\lVert x_1-x_2\rVert}{\sqrt{T}}\right]\\
&\leq 8e^{2cT} \tfrac{(V(t_1,x_1))^{\expSecondNorm}+(V(t_2,x_2))^{\expSecondNorm}}{2}\left[
\tfrac{V(t_1,x_1)+V(t_2,x_2)}{2}\tfrac{\lvert t_1-t_2\rvert^{\nicefrac{1}{2}}}{\sqrt{T}}+\tfrac{\lVert x_1-x_2\rVert}{\sqrt{T}}\right].\end{split}\label{p02}
\end{align}
This 
 and \eqref{h19} show that $\sup_{t\in[0,T]}\tnorm{u}{2,t}\leq 8e^{2cT}$. This 
the fact that $\sup_{t\in[0,T]}\tnorm{u}{1,t}\leq 2e^{cT}$
imply 
that $\sup_{t\in[0,T]}\tnorm{u}{t}\leq 8e^{2cT}$. This shows
\eqref{p01a}.

Next, \eqref{s02b}, the fact that
$\forall\,s,t\in[0,T]\colon \lvert s-t\rvert\leq \lvert s-t\rvert^{\nicefrac{1}{2}}\sqrt{T}$,
\eqref{p02} and   prove 
for all $s\in[0,T]$,
$t\in[s,T]$,  $x,y\in\R^d$ that
\begin{align}
&
\left\lvert (t-s)
\left[f(
u(s,x))-f(u(s,y))\right]\right\rvert \leq
\xeqref{s02b}
 c\lvert t-s\rvert^{\nicefrac{1}{2}}\sqrt{T}\lvert u(s,x)-u(s,y)\rvert\nonumber \\
&\leq 
cT\tfrac{\lvert t-s\rvert^{\nicefrac{1}{2}}}{\sqrt{T}}\xeqref{p02}
4e^{2cT}
\left( \tfrac{V(s,x)+V(s,y)}{2}\right)^2\tfrac{\lVert x-y\rVert}{\sqrt{T}}
= 4 c T e^{2cT}
\left( \tfrac{V(s,x)+V(s,y)}{2}\right)^2\tfrac{\lVert x-y\rVert}{\sqrt{T}}
\tfrac{|t-s|^{\nicefrac{1}{2}}}{\sqrt{T}}.\label{p06}
\end{align}
This, the triangle inequality,  \eqref{p03}, and the fact that
$1\leq V$ show
for all $s\in[0,T]$,
$t\in[s,T]$,  $x,y\in\R^d$ that
\begin{align}\begin{split}
&\left\lvert (T-s)\bigl[
f(
u(s,x))-f(u(s,y))\bigr]-(T-t)\bigl[f(u(t,\tilde{x}))-f(u(t,\tilde{y}))\bigr]\right\rvert \\
&\leq 
\left\lvert (T-t)\bigl[\bigl(
f(
u(s,x))-f(u(s,y))\bigr)-\bigl(f(u(t,\tilde{x}))-f(u(t,\tilde{y}))\bigr)\bigr]\right\rvert \\&\quad +\left\lvert (t-s)
f(
u(s,x))-f(u(s,y))\right\rvert \\
&\leq T\cdot \xeqref{p03}
896 c e^{6cT}
\left(\tfrac{V(s,x)+V(s,y)+V(t,\tilde{x})+V(t,\tilde{y})}{4}\right)^7 \\
&\quad \cdot 
\left(
\tfrac{\lVert x-y-(\tilde{x}-\tilde{y})\rVert}{ \sqrt{T} }+
\tfrac{\lVert x-y\rVert+\lVert\tilde{x}-\tilde{y}\rVert}{\sqrt{T}}
\left[2
\tfrac{V(s,x)+V(s,y)}{2}\tfrac{\lvert t-s\rvert^{\nicefrac{1}{2}}}{\sqrt{T}}+\tfrac{\lVert x-\tilde{x}\rVert}{\sqrt{T}}\right]\right)\\
&\quad + \xeqref{p06}
 4 c T e^{2cT}\left( 
\tfrac{V(s,x)+V(s,y)}{2}\right)^2\tfrac{\|x-y\|}{\sqrt{T}}\tfrac{|t-s|^{\nicefrac{1}{2}}}{\sqrt{T}}\\
&\leq 900 cT e^{6cT}
\left(\tfrac{V(s,x)+V(s,y)+V(t,\tilde{x})+V(t,\tilde{y})}{4}\right)^7 \\
&\quad \cdot 
\left(
\tfrac{\lVert x-y-(\tilde{x}-\tilde{y})\rVert}{ \sqrt{T} }+
\tfrac{\lVert x-y\rVert+\lVert\tilde{x}-\tilde{y}\rVert}{\sqrt{T}}
\left[2
\tfrac{V(s,x)+V(s,y)}{2}\tfrac{\lvert t-s\rvert^{\nicefrac{1}{2}}}{\sqrt{T}}+\tfrac{\lVert x-\tilde{x}\rVert}{\sqrt{T}}\right]\right).
\end{split}\label{t15}\end{align}
This, 
the fact that
$\forall\,x\in [0,\infty)\colon x\leq e^{-1}e^{x} $,
 the triangle inequality,
H\"older's inequality, the fact that
$
\max\{
\tfrac{7}{\exponentV}+\tfrac{2}{\exponentX}, \tfrac{8}{\exponentV}+\tfrac{1}{\exponentX}\}
\leq
\tfrac{8}{\exponentV}+\tfrac{2}{\exponentX}
\leq \tfrac{1}{\exponentLP}
$, 
\eqref{s10}, 
\eqref{t21b}, 
\eqref{h21b}, 
\eqref{s11}, 
\eqref{h21c}, 
the fact that $1\leq  V$,  the fact that
$9\leq \expSecondNorm$, and Jensen's inequality
show that for all $t_1\in[0,T] $, 
$t_2,r_1\in [t_1,T]$,
$r_2\in [t_2,T]\cap[r_1,T]$,
$x_1,x_2\in\R^d$,
$n\in\N$
with $ \lvert  r_1-r_2\rvert\leq \lvert t_1-t_2\rvert$
it holds
that $
900 cT e^{6cT}\cdot 6\leq 5400 e^{-1} e^{cT }e^{6cT} \leq 1987 e^{7cT}
$ and
\begin{align}
\begin{split}
&\biggl\lVert (T-t_1)\left[
(f\circ u)\bigl(r_1,\approximationX_{t_1,r_1}^{n,0,x_1}\bigr) 
-(f\circ u)\bigl(r_1,X_{t_1,r_1}^{x_1}\bigr)\right]\nonumber \\
&\qquad\qquad\qquad\qquad 
-(T-t_2)
\left[
(f\circ u)\bigl(r_2,\approximationX_{t_2,r_2}^{n,0,x_2}\bigr) 
-(f\circ u)\bigl(r_2,X_{t_2,r_2}^{x_2}\bigr)\right]\biggr\rVert_{\exponentLP} \nonumber \\\end{split}\nonumber\\
&\leq\xeqref{t15} 900 cT e^{6cT}\Biggl\lVert 
\left(
\tfrac{\sum_{i=1}^{2}\left[V\left(r_i,\approximationX_{t_i,r_i}^{n,0,x_i}\right)+V\left(r_i,X_{t_i,r_i}^{x_i}\right) \right]}{4}\right)^7 \Biggl[
\tfrac{\left\lVert 
\left(
\approximationX_{t_1,r_1}^{n,0,x_1}-X_{t_1,r_1}^{x_1}\right)-
\left(
\approximationX_{t_2,r_2}^{n,0,x_2}-X_{t_2,r_2}^{x_2}\right)\right\rVert}{\sqrt{T}}\nonumber \\
&
 +
\tfrac{\sum_{i=1}^{2}\left\lVert 
\approximationX_{t_i,r_i}^{n,0,x_i}-X_{t_i,r_i}^{x_i}\right\rVert}{\sqrt{T}} \Biggl(2\tfrac{\left[V\left(r_1,\approximationX_{t_1,r_1}^{n,0,x_1}\right)+V\left(r_1,X_{t_1,r_1}^{x_1}\right)\right] }{2} \tfrac{\lvert r_1-r_2\rvert^{\nicefrac{1}{2}}}{\sqrt{T}} +\tfrac{\left\lVert X_{t_1,r_1}^{x_1}-X_{t_2,r_2}^{x_2}\right\rVert}{\sqrt{T}}\Biggr)\Biggr]\Biggr\rVert_{\exponentLP}\nonumber \\
&\leq 900 cT e^{6cT}\Biggl\lVert 
\left(
\tfrac{\sum_{i=1}^{2}\left[V\left(r_i,\approximationX_{t_i,r_i}^{n,0,x_i}\right)+V\left(r_i,X_{t_i,r_i}^{x_i}\right)\right] }{4}\right)^7\Biggr\rVert_{ \frac{\exponentV}{7}}\Biggl[
\tfrac{\left\lVert \left\lVert 
\left(
\approximationX_{t_1,r_1}^{n,0,x_1}-X_{t_1,r_1}^{x_1}\right)-
\left(
\approximationX_{t_2,r_2}^{n,0,x_2}-X_{t_2,r_2}^{x_2}\right)\right\rVert\right\rVert_{\exponentX}}{\sqrt{T}}\nonumber \\
& +
\tfrac{\sum_{i=1}^{2}\left\lVert \left\lVert 
\approximationX_{t_i,r_i}^{n,0,x_i}-X_{t_i,r_i}^{x_i}\right\rVert\right\rVert_{\exponentX}}{\sqrt{T}} \Biggl(
2\left\lVert \tfrac{\left[V\left(r_1,\approximationX_{t_1,r_1}^{n,0,x_1}\right)+V\left(r_1,X_{t_1,r_1}^{x_1}\right)\right] }{2}\right\rVert_{\exponentV} \tfrac{\lvert t_1-t_2\rvert^{\nicefrac{1}{2}}}{\sqrt{T}} +\tfrac{\left\lVert \left\lVert X_{t_1,r_1}^{x_1}-X_{t_2,r_2}^{x_2}\right\rVert\right\rVert_{\exponentX}}{\sqrt{T}}\Biggr)\Biggr]\nonumber \\
&\leq
900 cT e^{6cT}
\xeqref{s10}
\xeqref{t21b}
\left(\tfrac{V(t_1,x_1)+V(t_2,x_2)}{2}\right)^7
\Biggl[\xeqref{h21b}
\tfrac{V(t_1,x_1)+V(t_2,x_2)}{2}
\left[
\tfrac{\lvert t_1-t_2\rvert^{\nicefrac{1}{2}} 
+\lvert r_1-r_2\rvert^{\nicefrac{1}{2}}}{2\sqrt{T}}+
\tfrac{ \lVert x_1-x_2\rVert}{\sqrt{T}}
\right]\tfrac{1}{\sqrt{n}}
\nonumber \\
&+\xeqref{s11}\tfrac{V(t_1,x_1)+V(t_2,x_2)}{2\sqrt{n}}\biggl(
2\xeqref{s10}
\xeqref{t21b}
V(t_1,x_1)\tfrac{\lvert t_1-t_2\rvert^{\nicefrac{1}{2}}}{\sqrt{T}}\
+\xeqref{h21c}\tfrac{V(t_1,x_1)+V(t_2,x_2)}{2}\left[
\tfrac{\lvert t_1-t_2\rvert^{\nicefrac{1}{2}}+\lvert r_1-r_2\rvert^{\nicefrac{1}{2}}}{2\sqrt{T}}+\tfrac{\lVert x_1-x_2\rVert}{\sqrt{T}}\right]
\biggr)\Biggr]\nonumber \\
&\leq
900 cT e^{6cT}
\left(\tfrac{V(t_1,x_1)+V(t_2,x_2)}{2}\right)^8
\Biggl[
\left[
\tfrac{\lvert t_1-t_2\rvert^{\nicefrac{1}{2}} 
}{\sqrt{T}}+
\tfrac{ \lVert x_1-x_2\rVert}{\sqrt{T}}
\right]\tfrac{1}{\sqrt{n}}
\nonumber \\
&+\tfrac{1}{\sqrt{n}}\biggl(
4
\tfrac{V(t_1,x_1)+V(t_2,x_2)}{2}\tfrac{\lvert t_1-t_2\rvert^{\nicefrac{1}{2}}}{\sqrt{T}}\
+\tfrac{V(t_1,x_1)+V(t_2,x_2)}{2}\left[
\tfrac{\lvert t_1-t_2\rvert^{\nicefrac{1}{2}} }{\sqrt{T}}+\tfrac{\lVert x_1-x_2\rVert}{\sqrt{T}}\right]
\biggr)\Biggr]\nonumber \\
&\leq 
900 cT e^{6cT}\tfrac{1}{\sqrt{n}}\left(\tfrac{V(t_1,x_1)+V(t_2,x_2)}{2}\right)^8
\left(6
\tfrac{V(t_1,x_1)+V(t_2,x_2)}{2}
\tfrac{\lvert t_1-t_2\rvert^{\nicefrac{1}{2}}}{\sqrt{T}}
+2\tfrac{V(t_1,x_1)+V(t_2,x_2)}{2}\tfrac{\lVert x_1-x_2\rVert}{\sqrt{T}}
\right)\nonumber \\
&\leq  1987 e^{7cT}
\tfrac{1}{\sqrt{n}}\tfrac{(V(t_1,x_1))^{\expSecondNorm}+(V(t_2,x_2))^{\expSecondNorm}}{2}
\left(
\tfrac{V(t_1,x_1)+V(t_2,x_2)}{2}
\tfrac{\lvert t_1-t_2\rvert^{\nicefrac{1}{2}}}{\sqrt{T}}
+\tfrac{\lVert x_1-x_2\rVert}{\sqrt{T}}
\right)
.\label{h33}
\end{align}
Next, \eqref{l02}, the triangle inequality, \eqref{h19}, \eqref{p02},
H\"older's inequality,
the fact that $\expFirstNorm+2\leq \expSecondNorm$, and the fact that
$1\leq  V$ 
 prove for all
$s\in[0,T] $, $t_1\in[s,T]$, 
$t_2\in [t_1,T]$,
$x_1,x_2\in\R^d$,   $U\in \{U_{n,m}^\theta \colon n,m\in\Z,\theta\in\Theta\}$ 
that 
$
\tnorm{(f\circ U)-(f\circ u)}{1,s}\leq c\tnorm{U-u}{1,s}
$ and
\begin{align}
&
\left\lVert 
(f(u(t_1,x_1))-f(U(t_1,x_1)))-( f(u(t_2,x_2))-f(U(t_2,x_2)))
\right\rVert_2\nonumber \\
&\quad
- c\left\lVert (u(t_1,x_1)-U(t_1,x_1))-(u(t_2,x_2)-U(t_2,x_2))\right\rVert_2 \nonumber \\
&\leq \xeqref{l02}  c\tfrac{\left\lVert u(t_1,x_1)-U(t_1,x_1)\right\rVert_2+\left\lVert u(t_2,x_2)-U(t_2,x_2)\right\rVert_2}{2}
\left\lvert u(t_1,x_1)-u(t_2,x_2)\right\rvert \nonumber \\
&\leq \xeqref{h19} c
\tfrac{[ ( V(t_1,x_1))^{\expFirstNorm}+(V(t_2,x_2))^{\expFirstNorm}]
\tnorm{U-u}{1,s}
}{2}\xeqref{p02}
8e^{2cT} \left(\tfrac{V(t_1,x_1)+V(t_2,x_2)}{2}\right)^2\left[
\tfrac{V(t_1,x_1)+V(t_2,x_2)}{2}\tfrac{\lvert t_1-t_2\rvert^{\nicefrac{1}{2}}}{\sqrt{T}}+\tfrac{\lVert x_1-x_2\rVert}{\sqrt{T}}\right]\nonumber \\
&\leq 8 c e^{2 c T} \tfrac{ (V(t_1,x_1))^{\expSecondNorm}+(V(t_2,x_2))^{\expSecondNorm}}{2}
\left[
\tfrac{V(t_1,x_1)+V(t_2,x_2)}{2}\tfrac{\lvert t_1-t_2\rvert^{\nicefrac{1}{2}}}{\sqrt{T}}+\tfrac{\lVert x_1-x_2\rVert}{\sqrt{T}}\right]\tnorm{u-U}{1,s}.\label{h33b}
\end{align}
Moreover, the fact that
$\forall\,x\in[0,\infty)\colon x\leq e^{-1}e^x $ shows that
$
 c+8c e^{2cT}\leq 9c e^{2cT}= \tfrac{1}{T} 9cT \cdot e^{2cT}\leq 
\tfrac{1}{T} 9 e^{-1} e^{cT}\cdot e^{2cT}\leq \tfrac{3 e^{3cT}}{T}.
$
This, \eqref{h33b}, 
 \eqref{h19}, and symmetry  prove for all
$s\in[0,T] $, 
 $U\in \{U_{n,m}^\theta \colon n,m\in\Z,\theta\in\Theta\}$ 
that 
$
\tnorm{(f\circ u)-(f\circ U)}{2,s}\leq 
c\tnorm{u-U}{2,s}+
8c e^{2cT}\tnorm{u-U}{1,s}
\leq \tfrac{3 e^{3cT}}{T}\tnorm{u-U}{s}
$
and  
\begin{align}\begin{split}
\tnorm{(f\circ U)-(f\circ u)}{s}
\leq \tfrac{3 e^{3cT}}{T}\tnorm{U-u}{s}
.\label{t01}
\end{split}\end{align}
Next, the triangle inequality, \eqref{h19}, and the fact that $1\leq V^{\expSecondNorm}$ show that for all
$s\in [0,T]$,
$t_1,t_2\in[s,T]$,
$r_1\in [t_1,T]$,
$r_2\in [t_2,T]$, 
$r\in [0,\min\{r_1,r_2\}]$,
$\xi_1,\xi_2\in\R^d$, random fields $H\colon [0,T]\times\R^d\to\R$
with $\lvert r_1-r_2\rvert\leq \lvert t_1-t_2\rvert$  and $V(r_1,\xi_1)\leq V(r_2,\xi_2)$
it holds that
\begin{equation}\textstyle
\lvert t_2-t_1\rvert\leq 
T\left(\frac{T-s}{T}\right)^{\frac{1}{2}} \frac{\lvert t_2-t_1\rvert^{\nicefrac{1}{2}}}{\sqrt{T}}
\leq T\left(\frac{T-s}{T}\right)^{\frac{1}{\exponentLP}} \frac{\lvert t_2-t_1\rvert^{\nicefrac{1}{2}}}{\sqrt{T}}
= T^{\frac{\exponentLP-1}{\exponentLP}}(T-s)^{\frac{1}{\exponentLP}} \frac{\lvert t_2-t_1\rvert^{\nicefrac{1}{2}}}{\sqrt{T}}\label{t22}
\end{equation}
and
\begin{align}
&
\left\lVert (T-t_1)H(r_1,\xi_1)-
(T-t_2)H(r_2,\xi_2)\right\rVert_{\exponentLP}\nonumber 
\\
&\leq (T-t_2) \left\lVert  H(r_1,\xi_1)-H(r_2,\xi_2)\right\rVert_{\exponentLP}
+\lVert (t_2-t_1) H(r_1,\xi_1)\rVert_{\exponentLP}\nonumber \\
&\leq T^{\frac{\exponentLP-1}{\exponentLP}}(T-s)^{\frac{1}{\exponentLP}}\xeqref{h19} 
\tfrac{(V(r_1,\xi_1))^{\expSecondNorm}+(V(r_2,\xi_2))^{\expSecondNorm}}{2}
\left[
\tfrac{V(r_1,\xi_1)+V(r_2,\xi_2)}{2}
\tfrac{\ \lvert r_1-r_2\rvert^{\nicefrac{1}{2}}}{\sqrt{T}}+\tfrac{\lVert \xi_1-\xi_2\rVert}{\sqrt{T}}
\right]\tnorm{H}{2,r}\nonumber \\
&\quad +\xeqref{t22} T^{\frac{\exponentLP-1}{\exponentLP}}(T-s)^{\frac{1}{\exponentLP}} \tfrac{\lvert t_2-t_1\rvert^{\nicefrac{1}{2}}}{\sqrt{T}}\xeqref{h19} V(r_1,\xi_1)
\tnorm{H}{1,r}\nonumber \\
&\leq   T^{\frac{\exponentLP-1}{\exponentLP}}(T-s)^{\frac{1}{\exponentLP}}\tnorm{H}{r}
\tfrac{(V(r_1,\xi_1))^{\expSecondNorm}+(V(r_2,\xi_2))^{\expSecondNorm}}{2}
\left[\tfrac{3V(r_1,\xi_1)+V(r_2,\xi_2)}{2}
\tfrac{ \lvert t_1-t_2\rvert^{\nicefrac{1}{2}}}{\sqrt{T}}+\tfrac{\lVert \xi_1-\xi_2\rVert}{\sqrt{T}}
\right]\nonumber
\\
&\leq   T^{\frac{\exponentLP-1}{\exponentLP}}(T-s)^{\frac{1}{\exponentLP}}\tnorm{H}{r}
\tfrac{(V(r_1,\xi_1))^{\expSecondNorm}+(V(r_2,\xi_2))^{\expSecondNorm}}{2}
\left[2
\tfrac{V(r_1,\xi_1)+V(r_2,\xi_2)}{2}
\tfrac{ \lvert t_1-t_2\rvert^{\nicefrac{1}{2}}}{\sqrt{T}}+\tfrac{\lVert \xi_1-\xi_2\rVert}{\sqrt{T}}
\right]
.\label{m11}
\end{align} 
This, symmetry,
the disintegration theorem (see, e.g., \cite[Lemma 2.2]{HJKNW2018}), the assumption on measurability and independence,
the triangle inequality, H\"older's inequality, the fact that
$\max\{\tfrac{\expSecondNorm}{\exponentV}+ \tfrac{1}{\exponentV} ,
\tfrac{\expSecondNorm}{\exponentV}+ \tfrac{1}{\exponentX} 
\}
\leq \frac{1}{\exponentLP}$,
\eqref{s10}, 
and \eqref{h21}
 show that for all
$s\in [0,T]$,
$t_1,t_2\in[s,T]$,
$x_1,x_2\in\R^d$,
$r_1\in [t_1,T]$,
$r_2\in [t_2,T]$,
$r\in [0,\min\{r_1,r_2\}]$,
 $m,n\in \N$, $\ell,\nu\in\Z $, 
$H\in \mathrm{span}_{\R}(\{f\circ U_{\ell,m}^{\nu},f\circ u\}) $ 
with  $|r_1-r_2|\leq |t_1-t_2|$ it holds 
that
\begin{align}
&
\left\lVert (T-t_1)
H
\bigl(r_1,\approximationX_{t_1,r_1}^{n,0,x_1}\bigr)
-(T-t_2)H
\bigl(r_2,\approximationX_{t_2,r_2}^{n,0,x_2}\bigr)
\right\rVert_{\exponentLP}\nonumber \\
&= \left\lVert \left\lVert (T-t_1)
H
\bigl(r_1,\xi_1\bigr)
-(T-t_2)H
\bigl(r_2,\xi_2\bigr)
\right\rVert_{\exponentLP}\big|_{\substack{\xi_1= \approximationX_{t_1,r_1}^{n,0,x_1}
,\xi_2= \approximationX_{t_2,r_2}^{n,0,x_2}
 }}
\right\rVert_{\exponentLP}\nonumber 
\\
&\leq\xeqref{m11} T^{\frac{\exponentLP-1}{\exponentLP}}(T-s)^{\frac{1}{\exponentLP}}\left\lVert \left[
 \tnorm{H}{r}
 \tfrac{(V(r_1,\xi_1))^{\expSecondNorm}+(V(r_2,\xi_2))^{\expSecondNorm}}{2}
\left(2 \tfrac{V(r_1,\xi_1)+V(r_2,\xi_2)}{2}
\tfrac{\left\lvert t_1-t_2\right\rvert ^{\nicefrac{1}{2}}}{\sqrt{T}}+\tfrac{\|\xi_1-\xi_2\|}{\sqrt{T}}\right)\right]
\biggr|_{\substack{\xi_1= \approximationX_{t_1,r_1}^{n,0,x_1}
\\\xi_2= \approximationX_{t_2,r_2}^{n,0,x_2}
 }}
\right\rVert_{\exponentLP}  \nonumber
\\
&\leq  T^{\frac{\exponentLP-1}{\exponentLP}}(T-s)^{\frac{1}{\exponentLP}}
 \tnorm{H}{r}
\left(\tfrac{1}{2}{\textstyle\sum\limits_{i=1}^{2}}\left\lVert \bigl(V\bigl(r_i, \approximationX_{t_i,r_i}^{n,0,x_i}\bigr)\bigr)^{\expSecondNorm}\right\rVert_{\frac{\exponentV}{\expSecondNorm}}\right)\cdot\nonumber 
\\
&\qquad\qquad\qquad\qquad  \cdot
\left[2\left(\tfrac{1}{2}{\textstyle\sum\limits_{i=1}^{2}}
\left\lVert V\bigl(r_i, \approximationX_{t_i,r_i}^{n,0,x_i}\bigr)\right\rVert_{\exponentV}\right)
\tfrac{|t_1-t_2|^{\nicefrac{1}{2}}}{\sqrt{T}}+ \tfrac{
\left\lVert \left\lVert \approximationX_{t_1,r_1}^{n,0,x_1}-
 \approximationX_{t_2,r_2}^{n,0,x_2}
\right\rVert\right\rVert_{\exponentX}}{\sqrt{T}}\right]
 \nonumber 
\\
&\leq  T^{\frac{\exponentLP-1}{\exponentLP}}(T-s)^{\frac{1}{\exponentLP}}
 \tnorm{H}{r}\xeqref{s10}
\tfrac{(V(t_1, x_1))^{\expSecondNorm}+(V(t_2, x_2))^{\expSecondNorm}}{2}\nonumber\\ 
&\qquad\cdot 
\biggl[2\xeqref{s10} \tfrac{V(t_1, x_1)+V(t_2,x_2)}{2}
\tfrac{\lvert t_1-t_2\rvert^{\nicefrac{1}{2}}}{\sqrt{T}}
+ \xeqref{h21}\tfrac{V(t_1, x_1)+V(t_2,x_2)}{2}
\tfrac{\lvert t_1-t_2\rvert^{\nicefrac{1}{2}}
+\lvert r_1- r_2\rvert^{\nicefrac{1}{2}}}{2\sqrt{T}}
+ \tfrac{e^{cT} \lVert x_1-x_2\rVert}{\sqrt{T}}\biggr]
\nonumber 
\\
&\leq 3e^{cT}
 T^{\frac{\exponentLP-1}{\exponentLP}}
\tnorm{(T-s)^{\frac{1}{\exponentLP}}H}{r}\tfrac{(V(t_1,x_1))^{\expSecondNorm}+(V(t_2,x_2))^{\expSecondNorm}}{2}
\left[
\tfrac{V(t_1,x_1)+V(t_2,x_2)}{2}
\tfrac{|t_1-t_2|^{\nicefrac{1}{2}}}{\sqrt{T}}+\tfrac{\lVert x_1-x_2\rVert}{\sqrt{T}}\right].\label{h29}
\end{align}
Next,  a telescoping sum argument  shows for all 
$\ell,m\in\N$, $t_1,t_2\in[0,T] $, 
$r_1\in[t_1,T]$,
$r_2\in[t_2,T]$,
$x_1,x_2\in\R^d$ 
that
\begin{align}
&  (T-t_1)\Bigl[   \bigl( f\circ {U}_{\ell,m}^{0}\bigr)\!
\left(r_1,\approximationX_{t_1,r_1}^{m^{\ell},0,x_1}\right)
-
\bigl(f\circ {U}_{\ell-1,m}^{1}\bigr)\!
    \left(r_1,\approximationX_{t_1,r_1}^{m^{\ell-1},0,x_1}   \right)\Bigr]
\nonumber \\
&\qquad -
(T-t_2)\Bigl[   \bigl( f\circ {U}_{\ell,m}^{0}\bigr)\!
\left(r_2,\approximationX_{t_2,r_2}^{m^{\ell},0,x_2}\right)
-
\bigl(f\circ {U}_{\ell-1,m}^{1}\bigr) \!
    \left(r_2,\approximationX_{t_2,r_2}^{m^{\ell-1},0,x_2}   \right)\Bigr]
\nonumber \\
&= \biggl\{(T-t_1)\bigl(\bigl(f \circ{U}_{\ell,m}^{0}\bigr)-(f\circ u)\bigr)\!
\left(r_1,\approximationX_{t_1,r_1}^{m^{\ell},0,x_1}\right)\nonumber \\
& \qquad\qquad\qquad\qquad\qquad 
-
(T-t_2)\bigl(\bigl(f \circ{U}_{\ell,m}^{0}\bigr)-(f\circ u)\bigr)\!
\left(r_2,\approximationX_{t_2,r_2}^{m^{\ell},0,x_2}\right)\biggr\}\xeqref{h29}
\nonumber \\
&\quad
- \biggl\{(T-t_1)\bigl(\bigl(f\circ {U}_{\ell-1,m}^{1}\bigr)-(f\circ u)\bigr)\!
    \left(r_1,\approximationX_{t_1,r_1}^{m^{\ell-1},0,x_1}   \right) 
\nonumber \\
& \qquad\qquad\qquad\qquad\qquad -
(T-t_2)\bigl(\bigl(f\circ {U}_{\ell-1,m}^{1}\bigr)-(f\circ u)\bigr)\!
    \left(r_2,\approximationX_{t_2,r_2}^{m^{\ell-1},0,x_2}   \right) \biggr\}\xeqref{h29}\nonumber 
\\
&\quad
+\biggl\{
(T-t_1)
\Bigl[
(f\circ u)\!\left(r_1,\approximationX_{t_1,r_1}^{m^{\ell},0,x_1}\right) 
-(f\circ u)\!\left(r_1,X_{t_1,r_1}^{x_1}\right)\Bigr]\nonumber \\
& \qquad\qquad\qquad\qquad\qquad 
-
(T-t_2)
\Bigl[
(f\circ u)\!\left(r_2,\approximationX_{t_2,r_2}^{m^{\ell},0,x_2}\right) 
-(f\circ u)\!\left(r_2,X_{t_2,r_2}^{x_2}\right)\Bigr]\biggr\}\xeqref{h33}
 \nonumber \\
&\quad -\biggl\{(T-t_1)\Bigl[
(f\circ u)\left(r_1,\approximationX_{t_1,r_1}^{m^{\ell-1},0,x_1}   \right)
-(f\circ u)\left(r_1,X_{t_1,r_1}^{x_1}\right)
\Bigr]\nonumber \\
& \qquad\qquad\qquad\qquad\qquad 
-(T-t_2)\Bigl[
(f\circ u)\!\left(r_2,\approximationX_{t_2,r_2}^{m^{\ell-1},0,x_2}   \right)
-(f\circ u)\!\left(r_2,X_{t_2,r_2}^{x_2}\right)
\Bigr]\biggr\}\xeqref{h33}.\label{h29b}
\end{align}
This, the triangle inequality,  
\eqref{h29}, 
\eqref{h33}, 
\eqref{t01}, 
and the fact that
$3e^{cT}
 T^{\frac{\exponentLP-1}{\exponentLP}}\tfrac{3 e^{3cT}}{T}
=9 e^{4cT}T^{-\nicefrac{1}{p}}
$
prove that for all 
$\ell,m\in\N$, $t_1,t_2\in[0,T] $, 
$r_1\in[t_1,T]$,
$r_2\in[t_2,T]$, $r\in [0, \min \{r_1,r_2\}]$,
$x_1,x_2\in\R^d$ 
with  $\lvert r_1-r_2\rvert\leq \lvert t_1-t_2\rvert$  it holds 
that 
\begin{align}
&\tfrac{\footnotesize\begin{aligned}
&\bigg\lVert 
(T-t_1)\Bigl[   \bigl( f\circ {U}_{\ell,m}^{0}\bigr)\!
\left(r_1,\approximationX_{t_1,r_1}^{m^{\ell},0,x_1}\right) 
-
\bigl(f\circ {U}_{\ell-1,m}^{1}\bigr)\!
    \left(r_1,\approximationX_{t_1,r_1}^{m^{\ell-1},0,x_1}   \right)\Bigr]
\nonumber \\[-6pt]
&\quad  -
(T-t_2)\Bigl[   \bigl( f\circ {U}_{\ell,m}^{0}\bigr)\!
\left(r_2,\approximationX_{t_2,r_2}^{m^{\ell},0,x_2}\right)
-
\bigl(f\circ {U}_{\ell-1,m}^{1}\bigr)\!
    \left(r_2,\approximationX_{t_2,r_2}^{m^{\ell-1},0,x_2}   \right)\Bigr]\biggr\rVert_{\exponentLP}\nonumber \end{aligned}}{
\tfrac{(V(t_1,x_1))^{\expSecondNorm}+(V(t_2,x_2))^{\expSecondNorm}}{2}
\left[
\tfrac{V(t_1,x_1)+V(t_2,x_2)}{2}
\tfrac{\lvert t_1-t_2\rvert^{\nicefrac{1}{2}}}{\sqrt{T}}+\tfrac{\lVert x_1-x_2\rVert}{\sqrt{T}}\right]}
\nonumber \\
&\leq \sum_{\nu=0}^{1}\left[\xeqref{h29}
3e^{cT}
 T^{\frac{\exponentLP-1}{\exponentLP}}
\tnorm{(T-s)^{1/\exponentLP}\left[(f\circ U_{\ell-\nu,m}^{\nu})-(f\circ u)\right]}{r}+\xeqref{h33}
\tfrac{1987 e^{7cT}}{\sqrt{m^{\ell-\nu}}}
\right]\nonumber 
\\
&\leq 
 \sum_{\nu=0}^{1}\left[
3e^{cT}
 T^{\frac{\exponentLP-1}{\exponentLP}}\xeqref{t01}\tfrac{3 e^{3cT}}{T}
\tnorm{(T-s)^{{1}/{\exponentLP}}\left[U_{\ell-\nu,m}^{\nu}-u\right]}{r}
+ 
\tfrac{1987 e^{7cT}}{\sqrt{m^{\ell-\nu}}}\right]\nonumber \\
&= 
 \sum_{\nu=0}^{1}\left[9 e^{4cT}T^{-\nicefrac{1}{p}}
\tnorm{(T-s)^{{1}/{\exponentLP}}\left[U_{\ell-\nu,m}^{\nu}-u\right]}{r}
+ 
\tfrac{1987 e^{7cT}}{\sqrt{m^{\ell-\nu}}}\right]
.\label{h35}
\end{align}
Next, \eqref{l01}, 
the triangle inequality, H\"older's inequality, 
the fact that
$\tfrac{1}{\exponentV}+\tfrac{2}{\exponentX}\leq \tfrac{1}{\exponentLP}$,
\eqref{s10}, \eqref{t21b}, 
\eqref{h21b}, 
\eqref{s11}, 
\eqref{h21c}, 
 the fact that
$1\leq  V$, the fact that $3\leq \expSecondNorm$, and Jensen's inequality
show for all   $n\in\N$,
$t_1,t_2\in[0,T]$, 
$x_1,x_2\in\R^d$ that
\begin{align}
&  \left\lVert \left[g \bigl(\approximationX^{n,0,x_1}_{t_1,T}\bigr)-
g \bigl(X^{x_1}_{t_1,T}\bigr)\right]
-\left[g \bigl(\approximationX^{n,0,x_2}_{t_2,T}\bigr)-
g \bigl(X^{x_2}_{t_2,T}\bigr)\right]\right\rVert_{\exponentLP}\nonumber \\
&\leq  \xeqref{l01} \Biggl\lVert 
 \tfrac{
V( T,\approximationX^{n,0,x_1}_{t_1,T}) 
+V(T,X^{x_1}_{t_1,T})
+V(T, \approximationX^{n,0,x_2}_{t_2,T}) +V(T,X^{x_2}_{t_2,T})}{4}
\Biggl[ 
  \tfrac{\left\lVert 
 \bigl(\approximationX^{n,0,x_1}_{t_1,T} -
X^{x_1}_{t_1,T}\bigr)
-   \bigl(\approximationX^{n,0,x_2}_{t_2,T} - X^{x_2}_{t_2,T}\bigr)\right\rVert}{\sqrt{T}}\nonumber \\
&\qquad\qquad\qquad\qquad
+\tfrac{\left(
\left\lVert 
\approximationX^{n,0,x_1}_{t_1,T} -
X^{x_1}_{t_1,T}\right\rVert
+\left\lVert 
\approximationX^{n,0,x_2}_{t_2,T} -
X^{x_2}_{t_2,T}\right\rVert
\right)
\left\lVert X^{x_1}_{t_1,T}-X^{x_2}_{t_2,T}\right\rVert}{2T}\Biggr] \Biggr\rVert_{\exponentLP}
 \nonumber \\
&\leq   \Biggl\lVert 
 \tfrac{
V( T,\approximationX^{n,0,x_1}_{t_1,T}) 
+V(T,X^{x_1}_{t_1,T})
+V(T, \approximationX^{n,0,x_2}_{t_2,T}) +V(T,X^{x_2}_{t_2,T})}{4}\Biggr\rVert_{\exponentV}
\Biggl[ 
  \tfrac{\left\lVert \left\lVert 
 \bigl( \approximationX^{n,0,x_1}_{t_1,T} -
X^{x_1}_{t_1,T}\bigr)
-   \bigl(\approximationX^{n,0,x_2}_{t_2,T} - X^{x_2}_{t_2,T}\bigr)\right\rVert\right\rVert_{\exponentX}}{\sqrt{T}}\nonumber \\
&\qquad\qquad\qquad\qquad
+\tfrac{\left(
\left\lVert 
\left\lVert 
\approximationX^{n,0,x_1}_{t_1,T} -
X^{x_1}_{t_1,T}\right\rVert\right\rVert_{\exponentX}
+
\left\lVert \left\lVert 
\approximationX^{n,0,x_2}_{t_2,T} -
X^{x_2}_{t_2,T}\right\rVert\right\rVert_{\exponentX}
\right)
\left\lVert
\left\lVert X^{x_1}_{t_1,T}-X^{x_2}_{t_2,T}
\right\rVert
\right\rVert_{\exponentX}}{2T}\Biggr] 
 \nonumber \\
&\leq\xeqref{s10}\xeqref{t21b} \tfrac{V(t_1,x_1)+V(t_2,x_2)}{2}\Biggl[
\xeqref{h21b}\tfrac{V(t_1,x_1)+V(t_2,x_2)}{2}
\left[
\tfrac{\lvert t_1-t_2\rvert^{\nicefrac{1}{2}}}{2\sqrt{T}} +\tfrac{ \lVert x_1-x_2\rVert}{\sqrt{T}}
\right]\tfrac{1}{\sqrt{n}}\nonumber \\
&\qquad\qquad\qquad \qquad\qquad+\tfrac{1}{2T}
\xeqref{s11}
\tfrac{\sqrt{T}(V(t_1,x_1)+V(t_2,x_2))}{2\sqrt{n}}\xeqref{h21c}
\tfrac{V(t_1,x_1)+V(t_2,x_2)}{2}
\left[
\tfrac{\lvert t_1-t_2\rvert^{\nicefrac{1}{2}}}{2}
+\lVert x_1-x_2\rVert
\right]\Biggr]\nonumber \\
&\leq\left( \tfrac{V(t_1,x_1)+V(t_2,x_2)}{2}\right)^2\Biggl[
\left[
\tfrac{\lvert t_1-t_2\rvert^{\nicefrac{1}{2}}}{2\sqrt{T}} +\tfrac{ \lVert x_1-x_2\rVert}{\sqrt{T}}
\right]\tfrac{1}{\sqrt{n}}+\tfrac{1}{2\sqrt{n}}
\tfrac{(V(t_1,x_1)+V(t_2,x_2))}{2}
\left[
\tfrac{\lvert t_1-t_2\rvert^{\nicefrac{1}{2}}}{2\sqrt{T}}
+\tfrac{\lVert x_1-x_2\rVert}{\sqrt{T}}
\right]\Biggr]\nonumber \\
&\leq 
\tfrac{(V(t_1,x_1))^{\expSecondNorm}+V(t_2,x_2))^{\expSecondNorm}}{2}
\left[\tfrac{V(t_1,x_1)+V(t_2,x_2)}{2}
\tfrac{\lvert t_1-t_2\rvert^{\nicefrac{1}{2}}}{2\sqrt{T}}+
\tfrac{\lVert x_1-x_2\rVert}{\sqrt{T}}
\right]\tfrac{3}{2\sqrt{n}}
.\label{h32b}
\end{align}
This and the triangle inequality show for all   
$\ell,m\in\N$,
$t_1,t_2\in[0,T]$, 
$x_1,x_2\in\R^d$ that
\begin{align}
&\left\lVert 
\left(
      g \bigl(\approximationX^{m^\ell,(0,\ell,i),x_1}_{t_1,T}\bigr)-
g \bigl(\approximationX^{m^{\ell-1},0,x_1}_{t_1,T}\bigr)\right)\nonumber - \left(g \bigl(\approximationX^{m^\ell,0,x_2}_{t_2,T}\bigr)-
g \bigl(\approximationX^{m^{\ell-1},0,x_2}_{t_2,T}\bigr)\right)\right\rVert_{\exponentLP}\nonumber \\
&\leq {\sum\limits_{j=\ell-1}^{\ell}}\left\lVert \left[g \bigl(\approximationX^{m^j,0,x_1}_{t_1,T}\bigr)-
g \bigl(X^{x_1}_{t_1,T}\bigr)\right]
-\left[g \bigl(\approximationX^{m^j,0,x_2}_{t_2,T}\bigr)-
g \bigl(X^{x_2}_{t_2,T}\bigr)\right]\right\rVert_{\exponentLP}\nonumber \\
&\leq 
\tfrac{(V(t_1,x_1))^{\expSecondNorm}+(V(t_2,x_2))^{\expSecondNorm} }{2}
\left[
\tfrac{V(t_1,x_1)+V(t_2,x_2)}{2}
\tfrac{\lvert t_1-t_2\rvert^{\nicefrac{1}{2}}}{2\sqrt{T}} +
\tfrac{\lVert x_1-x_2\rVert}{\sqrt{T}}\right]{\sum\limits_{j=\ell-1}^{\ell}}\tfrac{2}{\sqrt{m^j}}.\label{h32}
\end{align}
Next, \eqref{s01c}, H\"older's inequality, 
the fact that $\frac{1}{\exponentV}+\frac{1}{\exponentX}\leq \frac{1}{\exponentLP}$,
\eqref{s10}, \eqref{h21},   the fact that $2\leq \expSecondNorm$, and the fact that
$e^{cT}\leq V$ show for all $t_1\in[0,T]$, 
$t_2\in [t_1,T]$,
$x_1,x_2\in\R^d$ that
\begin{align}
&
\left\lVert 
 g \bigl(\approximationX^{1,0,x_1}_{t_1,T}\bigr) 
- g \bigl(\approximationX^{1,0,x_2}_{t_2,T}\bigr)\right\rVert_{\exponentLP}
\leq \xeqref{s01c}
\left\lVert
\tfrac{V(\approximationX^{1,0,x_1}_{t_1,T})+
V( \approximationX^{1,0,x_2}_{t_2,T})}{2}
\tfrac{ \left\lVert \approximationX^{1,0,x_1}_{t_1,T}- \approximationX^{1,0,x_2}_{t_2,T}\right\rVert
}{\sqrt{T}}
\right\rVert_{\exponentLP}
\nonumber \\
&\leq 
\left\lVert
\tfrac{V(\approximationX^{1,0,x_1}_{t_1,T})+
V( \approximationX^{1,0,x_2}_{t_2,T})}{2}\right\rVert_{\exponentV}
\left\lVert 
\tfrac{ \left\lVert \approximationX^{1,0,x_1}_{t_1,T}- \approximationX^{1,0,x_2}_{t_2,T}\right\rVert
}{\sqrt{T}}
\right\rVert_{\exponentX}\nonumber \\
&\leq \xeqref{s10}
\tfrac{V(t_1,x_1)+V(t_2,x_2)}{2}
\xeqref{h21}
\left[ 
\tfrac{V(t_1,x_1)+V(t_2,x_2)}{2}
\tfrac{\lvert t_1-t_2\rvert^{\nicefrac{1}{2}}}{2}+e^{cT}\lVert x_1-x_2\rVert\right]\tfrac{1}{\sqrt{T}}
\nonumber 
\\
&\leq 
\tfrac{(V(t_1,x_1))^{\expSecondNorm}+(V(t_2,x_2))^{\expSecondNorm}}{2}
\left[
\tfrac{V(t_1,x_1)+V(t_2,x_2)}{2}
\tfrac{\lvert t_1-t_2\rvert^{\nicefrac{1}{2}}}{2\sqrt{T}}+\tfrac{\|x_1-x_2\|}{\sqrt{T}}\right]
.\label{h34}
\end{align}
Next,  \cref{a02},
\eqref{l06}\xeqref{a14},
\eqref{s02b}\xeqref{a14c}, 
\eqref{s01c}\xeqref{a14c},
\eqref{d01}\xeqref{d01c},
\eqref{h21d}\xeqref{k01},
\eqref{h21c}\xeqref{k01}, 
\eqref{s11}\xeqref{k04b}, 
\eqref{s10}\xeqref{k04b},  and 
the assumptions of \cref{p05}
prove 
for all  $n,m\in\N$, $t\in[0,T]$, $x\in\R^d$  that
$ {U}_{n,m}^{0}(t,x)
$,
$      g \bigl(\approximationX^{m^{n-1},0,x}_{t,T}\bigr)$,
and $
\bigl( f\circ {U}_{n-1,m}^{0}\bigr)
\bigl(t+(T-t)\unif^{0},\approximationX_{t,t+(T-t)\unif^{0}}^{m^{n-1},0,x}\bigr)$
are  integrable,
\begin{equation}
 \E \bigl[ {U}_{n,m}^{0}(t,x)\bigr]=
\E\bigl[
      g \bigl(\approximationX^{m^{n-1},0,x}_{t,T}\bigr)\bigr]
+(T-t)
 \E \bigl[( f\circ {U}_{n-1,m}^{0})
\bigl(t+(T-t)\unif^{0},\approximationX_{t,t+(T-t)\unif^{0}}^{m^{n-1},0,x}\bigr)\bigr],\label{n03}
\end{equation} 
\begin{equation}\label{n02}
 u(t,x)=
\E\bigl[
      g \bigl(X^{x}_{t,T}\bigr)\bigr]
+(T-t)
 \E \bigl[( f\circ u)
\bigl(t+(T-t)\unif^{0},X_{t,t+(T-t)\unif^{0}}^{x}\bigr)\bigr].
\end{equation}
and
\begin{equation}\label{n01}
\tnorm{{U}_{n,m}^{0}-u}{1,t}
\leq \tfrac{16mne^{3cT}\sqrt{p-1}}{\sqrt{m^n}}
+
\sum_{\ell=0}^{n-1}\left[
\tfrac{4(T-t)^{1-\frac{1}{\exponentLP}} c\sqrt{p-1}}{\sqrt{m^{n-\ell-1}}}
\left[\int_{t}^{T} \tnorm{ {U}_{\ell,m}^{0}- u}{1,\zeta}^{\exponentLP} d\zeta\right]^{1/\exponentLP} 
\right]
.
\end{equation}
This, the disintegration theorem (see, e.g., \cite[Lemma 2.2]{HJKNW2018}), and the independence assumptions
show for all 
$n,m\in\N$, $t_1,t_2\in[0,T]$, $x_1,x_2\in \R^d$ that 
\begin{align}
&\left\lvert \E \!\left[  \left({U}_{n,m}^{0}(t_1,x_1)-u(t_1,x_1)\right) -  
\left({U}_{n,m}^{0}(t_2,x_2)-u(t_2,x_2)\right) 
 \right]\right\rvert \nonumber  \\
& 
=\E\!\left[
     \left( g \bigl(\approximationX^{m^{n-1},0,x_1}_{t_1,T}\bigr)-g \bigl(X^{x_1}_{t_1,T}\bigr)\right)
-
  \left( g \bigl(\approximationX^{m^{n-1},0,x_2}_{t_2,T}\bigr)-g \bigl(X^{x_2}_{t_2,T}\bigr)\right)
\right]\xeqref{h32b}\nonumber \\
& 
+\E\Biggl[ \E \Biggl[\biggl(\Bigl[(T-t_1)
\bigl(f\circ {U}_{n-1,m}^{0}-f\circ u\bigr)
\left(r_1,\approximationX_{t_1,r_1}^{m^{n-1},0,x_1}\right)\nonumber \\
&\qquad-
(T-t_2)
\bigl(f\circ {U}_{n-1,m}^{0}-f\circ u\bigr)
\left(r_2,\approximationX_{t_2,r_2}^{m^{n-1},0,x_2}\right)
\Bigr]\xeqref{h29}
\nonumber \\
&\qquad +
 \Bigl[
(T-t_1)\Bigl(
(f\circ u)
\bigl(r_1,\approximationX_{t_1,r_1}^{m^{n-1},0,x_1}\bigr)
-
( f\circ u)
\bigl(r_1,X_{t_1,r_1}^{x_1}\bigr)\Bigr)\nonumber \\
&\qquad-
(T-t_2)\Bigl(
(f\circ u)
\bigl(r_2,\approximationX_{t_2,r_2}^{m^{n-1},0,x_2}\bigr)
-
( f\circ u)
\bigl(r_2,X_{t_2,r_2}^{x_2}\bigr)\Bigr)
\Bigr]\xeqref{h33}\biggr)\bigr|_{\substack{r_1=t_1+(T-t_1)\lambda\\r_2=t_2+(T-t_2)\lambda}}\Biggr]\biggr|_{\lambda=\unif^0}\Biggr].
\end{align}%
This, 
the triangle inequality,  Jensen's inequality,
\eqref{h32b}, \eqref{h29}, 
the fact that $\forall\,\lambda\in [0,1],t_1,t_2\in [0,T]\colon \lvert
(t_1+(T-t_1)\lambda)-
(t_2+(T-t_2)\lambda)
\rvert\leq \lvert t_1-t_2\rvert$,
\eqref{h33},  \eqref{t01},  the fact that
$2+1987 e^{7cT}\leq 1989 e^{7cT}$, 
and the fact that
$3e^{cT}
 T^{\frac{\exponentLP-1}{\exponentLP}}\tfrac{3 e^{3cT}}{T}
=9 e^{4cT}T^{-\nicefrac{1}{p}}
$
 show for all
 $n,m\in\N$, 
$s\in[0,T]$,
$t_1\in[s,T]$, 
$t_2\in[t_1,T]$,
$x_1,x_2\in\R^d$ that
\begin{align}
&
\frac{
\left\lvert 
\E \!\left[  \left({U}_{n,m}^{0}(t_1,x_1)-u(t_1,x_1)\right) -  
\left({U}_{n,m}^{0}(t_2,x_2)-u(t_2,x_2)\right) 
 \right]\right\rvert }{\tfrac{(V(t_1,x_1))^{\expSecondNorm}+(V(t_2,x_2))^{\expSecondNorm}}{2}
\left[
\tfrac{V(t_1,x_1)+V(t_2,x_2)}{2}
\tfrac{\lvert t_1-t_2\rvert^{\nicefrac{1}{2}}}{\sqrt{T}}+\tfrac{\lVert x_1-x_2\rVert}{\sqrt{T}}\right]}\nonumber \\
&\leq \xeqref{h32b}\tfrac{2}{\sqrt{m^{n-1}}}+
\xeqref{h29}
3e^{cT}
 T^{\frac{\exponentLP-1}{\exponentLP}}\left\lVert (T-s)^{\frac{1}{\exponentLP}} \tnorm{(f\circ U^0_{n-1,m})-(f\circ u)}{2,s+(T-s)\unif^0}\right\rVert_{\exponentLP}+\xeqref{h33}
\tfrac{1987 e^{7cT}}{\sqrt{m^{n-1}}}\nonumber 
\\
&\leq 
3e^{cT}
 T^{\frac{\exponentLP-1}{\exponentLP}}\xeqref{t01}\tfrac{3 e^{3cT}}{T}
\left\lVert 
(T-s)^{\frac{1}{\exponentLP}}
\tnorm{U_{n-1,m}^{0}-u}{s+(T-s)\unif^0}\right\rVert_{\exponentLP}
+ 
\tfrac{1989 e^{7cT}}{\sqrt{m^{n-1}}}\nonumber 
\\
&= 
9 e^{4cT}T^{-\nicefrac{1}{p}}
\left\lVert 
(T-s)^{\frac{1}{\exponentLP}}
\tnorm{U_{n-1,m}^{0}-u}{s+(T-s)\unif^0}\right\rVert_{\exponentLP}
+ 
\tfrac{1989 e^{7cT}}{\sqrt{m^{n-1}}}.\label{m02}
\end{align}
Next, the triangle inequality, and the independence and distributional properties
 show for all $n,m\in\N$, $t_1,t_2\in[0,T]$, $x_1,x_2\in\R^d$ that
\begin{align}
&  \left(\var_p \left({U}_{n,m}^{0}(t_1,x_1)
-{U}_{n,m}^{0}(t_2,x_2)\right)\right)^{1/\exponentLP}\nonumber \\
&
\leq   \sum_{\ell=0}^{n-1}\Biggl(
\var_{\exponentLP} \Biggl[
\tfrac{1}{m^{n-\ell}}
    \sum_{i=1}^{m^{n-\ell}}
\Bigl[
\left(
      g \bigl(\approximationX^{m^{\ell},(0,\ell,i),x_1}_{t_1,T}\bigr)-\1_{\N}(\ell)
g \bigl(\approximationX^{m^{\ell-1},(0,\ell,i),x_1}_{t_1,T}\bigr)\right)\nonumber \\&\qquad\qquad\qquad\qquad\qquad
- \left(g \bigl(\approximationX^{m^{\ell},(0,\ell,i),x_2}_{t_2,T}\bigr)-\1_{\N}(\ell)
g \bigl(\approximationX^{m^{\ell-1},(0,\ell,i),x_2}_{t_2,T}\bigr)\right)\Bigr]\Biggr]\Biggr)^{1/\exponentLP}
\nonumber \\
 &+ \sum_{\ell=0}^{n-1}
\Biggl(\var_{\exponentLP} \Biggl[
\tfrac{1}{m^{n-\ell}}
    \sum_{i=1}^{m^{n-\ell}}
\Bigl[
(T-t_1)
     \bigl( f\circ {U}_{\ell,m}^{(0,\ell,i)}\bigr)
\left(t_1+(T-t_1)\unif^{(0,\ell,i)},\approximationX_{t_1,t_1+(T-t_1)\unif^{(0,\ell,i)}}^{m^{\ell},(0,\ell,i),x_1}\right)\nonumber \\
&\qquad\quad 
-\1_{\N}(\ell)(T-t_1)
\bigl(f\circ {U}_{\ell-1,m}^{(0,\ell,-i)}\bigr)\!
    \left(t_1+(T-t_1)\unif^{(0,\ell,i)},\approximationX_{t_1,t_1+(T-t_1)\unif^{(0,\ell,i)}}^{m^{\ell-1},(0,\ell,i),x_1}   \right) \Bigr]\nonumber \\
&\qquad\quad -\Bigl[
(T-t_2)
     \bigl( f\circ {U}_{\ell,m}^{(0,\ell,i)}\bigr)\!
\left(t_2+(T-t_2)\unif^{(0,\ell,i)},\approximationX_{t_2,t_2+(T-t_2)\unif^{(0,\ell,i)}}^{m^{\ell},(0,\ell,i),x_2}\right)\nonumber \\
&\qquad\quad
-\1_{\N}(\ell)(T-t_2)
\bigl(f\circ {U}_{\ell-1,m}^{(0,\ell,-i)}\bigr)\!
    \left(t_2+(T-t_2)\unif^{(0,\ell,i)},\approximationX_{t_2,t_2+(T-t_2)\unif^{(0,\ell,i)}}^{m^{\ell-1},(0,\ell,i),x_2}   \right) \Bigr]\Biggr]\Biggr)^{1/\exponentLP}.\label{a03}
\end{align}
In addition, the triangle inequality,
the Marcinkiewicz-Zygmund inequality (see \cite[Theorem~2.1]{Rio09}), the fact that $\exponentLP\in [2,\infty)$,  and  Jensen's inequality show that for all $n\in\N$ and all i.i.d.\ 
real-valued, integrable random variables $\mathfrak{X}_1,\mathfrak{X}_2,\ldots,\mathfrak{X}_n$ it holds that
$
(\var_{\exponentLP}(\tfrac{1}{n}\sum_{k=1}^{n}\mathfrak{X}_k))^{1/\exponentLP}=
\left\lVert \tfrac{1}{n}\sum_{k=1}^{n}(\mathfrak{X}_k-\E[\mathfrak{X}_k])\right\rVert_{\exponentLP}\leq \sqrt{\exponentLP-1}\|\mathfrak{X}_1-\E[\mathfrak{X}_1]\|_p
\leq 2\sqrt{\exponentLP-1}\|\mathfrak{X}_1\|_p.
$
This, \eqref{a03},
the fact that $\forall\,m\in\N\colon U_{0,m}^0=0 $, \eqref{h34},
\eqref{h32}, 
  the independence and distributional properties,
the disintegration theorem (see, e.g., \cite[Lemma 2.2]{HJKNW2018}),
\eqref{h35}, 
and the fact that
$2+1987 e^{7cT}\leq 1989 e^{7cT}$
 show for all $n,m\in\N$, 
$s\in[0,T]$, 
$t_1,t_2\in [s,T]$,
$x_1,x_2\in\R^d$ that
\begin{align}
&\left(\var_p \left({U}_{n,m}^{0}(t_1,x_1)
-{U}_{n,m}^{0}(t_2,x_2)\right)\right)^{1/\exponentLP}\leq\xeqref{a03}
\tfrac{2\sqrt{\exponentLP-1}}{ \sqrt{m^n}}
\left\lVert 
 g \bigl(\approximationX^{1,0,x_1}_{t_1,T}\bigr) 
- g \bigl(\approximationX^{1,0,x_2}_{t_2,T}\bigr)\right\rVert_{\exponentLP} \xeqref{h34}  \nonumber\\
& +
\sum_{\ell=1}^{n-1}\Biggl[\tfrac{2\sqrt{\exponentLP-1}}{\sqrt{m^{n-\ell}}}
\left\lVert 
\left(
      g \bigl(\approximationX^{m^{\ell},0,x_1}_{t_1,T}\bigr)-
g \bigl(\approximationX^{m^{\ell-1},0,x_1}_{t_1,T}\bigr)\right)  - \left(g \bigl(\approximationX^{m^{\ell},0,x_2}_{t_2,T}\bigr)-
g \bigl(\approximationX^{m^{\ell-1},0,x_2}_{t_2,T}\bigr)\right)\right\rVert_{\exponentLP}\xeqref{h32} \nonumber \\
&\qquad\quad +\tfrac{2\sqrt{\exponentLP-1}}{\sqrt{m^{n-\ell}}}
\bigg\|\biggl\|\biggl[
(T-t_1)\Bigl[   \bigl( f\circ {U}_{\ell,m}^{0}\bigr)\!
\left(r_1,\approximationX_{t_1,r_1}^{m^{\ell},0,x_1}\right)
-
\bigl(f\circ {U}_{\ell-1,m}^{1}\bigr)\!
    \left(r_1,\approximationX_{t_1,r_1}^{m^{\ell-1},0,x_1}   \right)\Bigr]
\nonumber  \\
& \qquad\qquad\qquad\qquad\qquad-
(T-t_2)\Bigl[   \bigl( f\circ {U}_{\ell,m}^{0}\bigr)\!
\left(r_2,\approximationX_{t_2,r_2}^{m^{\ell},0,x_2}\right) \nonumber \\
&\qquad\qquad\qquad\qquad\qquad\qquad
-
\bigl(f\circ {U}_{\ell-1,m}^{1}\bigr)\!
    \left(r_2,\approximationX_{t_2,r_2}^{m^{\ell-1},0,x_2}   \right)\Bigr]\biggr]\bigr|_{\substack{r_1=t_1+(T-t_1)\lambda\\r_2=t_2+(T-t_2)\lambda}}\biggr\|_{\exponentLP}\biggr|_{\lambda=\unif^0}
\biggr\|_{\exponentLP}\xeqref{h35}\Biggr]  \nonumber
\\
&\leq 
\tfrac{(V(t_1,x_1))^{\expSecondNorm}+(V(t_2,x_2))^{\expSecondNorm}}{2}
\left[
\tfrac{V(t_1,x_1)+V(t_2,x_2)}{2}
\tfrac{\lvert t_1-t_2\rvert^{\nicefrac{1}{2}}}{\sqrt{T}}+\tfrac{\|x_1-x_2\|}{\sqrt{T}}\right]\Biggl(
\tfrac{2\sqrt{\exponentLP-1}}{\sqrt{m^n}}  \nonumber\\
&\quad 
+\sum_{\ell=1}^{n-1}\tfrac{2\sqrt{\exponentLP-1}}{\sqrt{m^{n-\ell}}}\sum_{j=\ell-1}^{\ell}\left[\tfrac{2}{\sqrt{m^j}}+
9 e^{4cT}T^{-\nicefrac{1}{p}}
\left\lVert 
(T-s)^{\frac{1}{\exponentLP}}
\tnorm{U_{j,m}^{0}-u}{s+(T-s)\unif^0}\right\rVert_{\exponentLP}
+ 
\tfrac{1987 e^{7cT}}{\sqrt{m^{j}}}\right]\Biggr)\nonumber
\\
\begin{split}
&\leq
\tfrac{(V(t_1,x_1))^{\expSecondNorm}+(V(t_2,x_2))^{\expSecondNorm}}{2}
\left[
\tfrac{V(t_1,x_1)+V(t_2,x_2)}{2}
\tfrac{|t_1-t_2|^{\nicefrac{1}{2}}}{\sqrt{T}}+\tfrac{\|x_1-x_2\|}{\sqrt{T}}\right]\Biggl(
\tfrac{2\sqrt{\exponentLP-1}}{\sqrt{m^n}}  \\
&\quad  
+\sum_{\ell=1}^{n-1}\tfrac{2\sqrt{\exponentLP-1}}{\sqrt{m^{n-\ell}}}\sum_{j=\ell-1}^{\ell}\left[
9 e^{4cT}T^{-\nicefrac{1}{p}}
\left\lVert 
(T-s)^{\frac{1}{\exponentLP}}
\tnorm{U_{j,m}^{0}-u}{s+(T-s)\unif^0}\right\rVert_{\exponentLP}
+ 
\tfrac{1989 e^{7cT}}{\sqrt{m^{j}}}\right]\Biggr).
\label{h36}\end{split}
\end{align}
This, the triangle inequality, 
\eqref{m02},  and
 the fact that $\forall\,n,m\in \N$, $a_0,a_1,\ldots,a_{n-1}\in[0,\infty)\colon 
a_{n-1}+\sum_{\ell=1}^{n-1}(\frac{1}{\sqrt{ m^{n-\ell}}}
\sum_{j=\ell-1}^\ell a_j)
\leq \sum_{\ell=0}^{n-1}\frac{2a_\ell}{\sqrt{ m^{n-\ell-1}}}
$
prove for all $n,m\in\N$, 
$s\in[0,T]$,
$t_1,t_2\in[s,T]$, 
$x_1,x_2\in\R^d$ that
\begin{align}
&\left\lVert  \left({U}_{n,m}^{0}(t_1,x_1)-u(t_1,x_1)\right) -  
\left({U}_{n,m}^{0}(t_2,x_2)-u(t_2,x_2)\right)\right\rVert_{\exponentLP}\nonumber \\
&\leq 
 \left(\var_p \left({U}_{n,m}^{0}(t_1,x_1)
-{U}_{n,m}^{0}(t_2,x_2)\right)\right)^{1/\exponentLP}\nonumber \\
&\qquad\qquad+
\left\lvert \E \!\left[  \left({U}_{n,m}^{0}(t_1,x_1)-u(t_1,x_1)\right) -  
\left({U}_{n,m}^{0}(t_2,x_2)-u(t_2,x_2)\right) 
 \right]\right\rvert \nonumber \\
&\leq \left(
 \tfrac{2\sqrt{\exponentLP-1}}{\sqrt{m^n}}
+\sum_{\ell=0}^{n-1}\left[\tfrac{4\sqrt{\exponentLP-1}}{\sqrt{m^{n-\ell-1}}}\left[
9 e^{4cT}T^{-\nicefrac{1}{p}}
\left\lVert 
(T-s)^{\frac{1}{\exponentLP}}
\tnorm{U_{\ell,m}^{0}-u}{s+(T-s)\unif^0}\right\rVert_{\exponentLP}
+ 
\tfrac{1989 e^{7cT}}{\sqrt{m^{\ell}}}\right]\right]\right)\nonumber  \\
& \qquad\qquad\cdot
\tfrac{(V(t_1,x_1))^{\expSecondNorm}+(V(t_2,x_2))^{\expSecondNorm}}{2}
\left[\tfrac{V(t_1,x_1)+V(t_2,x_2)}{2}
\tfrac{\lvert t_1-t_2\rvert^{\nicefrac{1}{2}}}{\sqrt{T}}+\tfrac{\lVert x_1-x_2\rVert}{\sqrt{T}}\right].
\end{align}%
This,
\eqref{h19}, 
the fact that
for all $s\in[0,T]$, measurable $h\colon [s,T]\to \R$ it holds that
$
\lVert(T-s)^{\frac{1}{\exponentLP}}h(s+(T-s)\unif^0)\rVert_{\exponentLP}= [\int_{0}^{1}(T-s)|h(s+(T-s)\lambda)|^\exponentLP\,d\lambda]^{\frac{1}{\exponentLP}}= 
[\int_{s}^{T}|h(\zeta)|^\exponentLP\,d\zeta]^{\frac{1}{\exponentLP}}
$, and the fact that $\forall\,m,n\in \N\colon 2\sqrt{\exponentLP-1}+4n\sqrt{\exponentLP-1}\sqrt{m} 1989 e^{7cT}\leq 7958\sqrt{\exponentLP-1}mn e^{7cT}$
show for all  $n,m\in\N$, 
$s\in[0,T]$ that
\begin{align}
&
\tnorm{U_{n,m}^0-u}{2,s}
\leq  \tfrac{7958\sqrt{\exponentLP-1}mn e^{7cT}}{\sqrt{m^n}}
+\sum_{\ell=0}^{n-1}\tfrac{4\sqrt{\exponentLP-1}}{\sqrt{m^{n-\ell-1}}}
9 e^{4cT}T^{-\nicefrac{1}{p}}
\left[{\textstyle\int_{s}^{T}}
\tnorm{U_{\ell,m}^{0}-u}{\zeta}^{\exponentLP}d\zeta\right]^{\frac{1}{\exponentLP}}
.
 \end{align}
This, \eqref{n01}, and \eqref{h19} prove for all $n,m\in\N$, 
$s\in[0,T]$
 that
\begin{align}\label{h37}
&
\tnorm{U_{n,m}^0-u}{s}
\leq  \tfrac{7958\sqrt{\exponentLP-1}mn e^{7cT}}{\sqrt{m^n}}
+\sum_{\ell=0}^{n-1}\tfrac{4\sqrt{\exponentLP-1}}{\sqrt{m^{n-\ell-1}}}
9 e^{4cT}T^{-\nicefrac{1}{p}}
\left[{\int_{s}^{T}}
\tnorm{U_{\ell,m}^{0}-u}{\zeta}^{\exponentLP}d\zeta\right]^{\frac{1}{\exponentLP}}
.
 \end{align}
This, \eqref{p02}, and
\cite[Lemma 3.11]{HJKN20} (applied for every $s\in[0,T] $
with
$M\gets m $,
$N\gets n$,
$\tau \gets s $,
$a\gets 7958\sqrt{\exponentLP-1} m n  e^{7cT}$,
$b\gets 4\sqrt{\exponentLP-1}\cdot 9  e^{4cT} T^{-{1}/{\exponentLP}} $,
$(f_\ell)_{\ell\in\N_0}\gets (
[s,T]\ni t\mapsto 
\tnorm{U^0_{\ell,m}-u}{t}\in [0,\infty])_{\ell\in \N_0}$ in the notation of
\cite[Lemma 3.11]{HJKN20}) show for all $n,m\in\N$, 
$s\in[0,T]$
 that
\begin{align}
&\tnorm{U^0_{n,m}-u}{s}\leq 
\left[
 7958\sqrt{\exponentLP-1} m n  e^{7cT}+
4\sqrt{\exponentLP-1}\cdot 9  e^{4cT} T^{-{1}/{\exponentLP}}  (T-s)^{\nicefrac{1}{p}}\right]\nonumber \\
&\quad \cdot
\left(\sup_{t\in[s,T]} \tnorm{U^0_{0,m}-u}{t}\right)
 e^{m^{p/2}/p}m^{-n/2}
\left[1+4\sqrt{\exponentLP-1}\cdot 9 e^{4cT} T^{-{1}/{\exponentLP}}  (T-s)^{\nicefrac{1}{p}}\right]^{n-1}\nonumber \\
&\leq 
\left[
7958\sqrt{\exponentLP-1} m n  e^{7cT}+
36\sqrt{\exponentLP-1}\cdot   e^{4cT}\right]
8e^{2cT}
 e^{m^{p/2}/p}m^{-n/2}
\left[1+36\sqrt{\exponentLP-1}\cdot   e^{4cT}\right]^{n-1}\nonumber \\
&\leq \left[
7958 m n  e^{7cT}+
36   e^{4cT}\right](\exponentLP-1)^{\frac{1}{2}} 8e^{2cT}
 e^{m^{p/2}/p}m^{-n/2}\left(1+36 e^{4cT}\right)^{n-1}
(\exponentLP-1)^{\frac{n-1}{2}}\nonumber \\
&\leq 222 m n  e^{4cT}
\left[ 1+
36   e^{4cT}\right] 8 e^{2cT}
 e^{m^{p/2}/p}m^{-n/2}\left(1+36e^{4cT}\right)^{n-1}
(\exponentLP-1)^{\frac{n}{2}} \nonumber \\
&=1776 m n e^{6cT}
 e^{m^{p/2}/p}m^{-n/2}\left(1+ 36 e^{4cT}\right)^{n}
(\exponentLP-1)^{\frac{n}{2}}\nonumber \\
&\leq 1776 m n e^{6cT}
 e^{m^{p/2}/p}m^{-n/2}37^n e^{4 n cT}
(\exponentLP-1)^{\frac{n}{2}}.
\end{align}
This shows \eqref{p04}. The proof of \cref{p05} is thus completed.
\end{proof}

\section{Error estimates for multgrid approximations of FBSDEs}
\label{sec5}
In \cref{s01} below we prove that our multigrid approximations \cref{c10a}
have a computational effort which grows essentially quadratically
in the reciprocal accuracy $\nicefrac{1}{\epsilon}$
and polynomially in the dimension.
We note that we measure the error by $L^p$-norms of path distances.
First,
we apply in \Cref{m01} below the H\"older-norm estimate \cref{h19a}
for
MLP approximations of semilinear PDEs
 to obtain  the error estimate \cref{k05a} for our multigrid approximations
 of FBSDEs.
\begin{theorem}[Error estimate for FBSDEs]\label{m01}Assume \cref{b01},
let $\delta\in (0,\infty)$,
$\exponentLP\in[2,\infty)$,
$\expFirstNorm\in[3,\infty)$, $\expSecondNorm\in [9,\infty)$
satisfy that 
$\frac{\expSecondNorm+1}{\exponentV}+\frac{2}{\exponentX}\leq \frac{1}{\exponentLP} $
and
$\expFirstNorm+2\leq \expSecondNorm$, let
 $(\thetaBar_n)_{n\in\N_0}\subseteq\Theta$,
let $\lfloor \cdot \rfloor_m \colon \R \to \R$, $ m \in \N $, and 
$\lceil \cdot \rceil_m \colon  \R \to \R$, $ m \in \N $, 
satisfy for all $m \in \N$, $t \in [0,T]$ that
$\lfloor t \rfloor_m = \max( ([0,t]\backslash \{T\}) \cap \{ 0, \frac{ T }{ m }, \frac{ 2T }{ m }, \ldots \} )$
and 
$\lceil t \rceil_m = \min(((t,\infty) \cup \{T\})\cap  \{ 0, \frac{ T }{ m }, \frac{ 2T }{ m }, \ldots \} )$,
let $\Yappr^{n,m}=(\Yappr^{n,m}_t)_{t\in[0,T]}
\colon [0,T]\times\Omega\to\R $,
 $n,m\in\N$, 
satisfy  for all $n,m\in\N$, $t\in[0,T]$  that
\begin{align}\label{d04d} 
\Yappr^{n,m}_{t}
&= \sum_{\ell=0}^{n-1}\Biggl[
\left[ \tfrac{ \lceil t \rceil_{m^{\ell+1}} - t }{ ( T / m^{ \ell + 1 } ) } \right]
U^{\thetaBar_\ell}_{n-\ell,m}(\lfloor t \rfloor_{m^{\ell+1}}, \approximationX^{m^n,0,0}_{0,\lfloor t \rfloor_{m^{\ell+1}}})+
\left[\tfrac{ t-\lfloor t \rfloor_{m^{\ell+1}} }{ ( T / m^{ \ell + 1 } ) }\right]
U^{\thetaBar_\ell}_{n-\ell,m}(\lceil t \rceil_{m^{\ell+1}}, \approximationX^{m^n,0,0}_{0,\lceil t \rceil_{m^{\ell+1}}})
\nonumber
\\
&-\1_{\N}(\ell)\biggl(
\left[ \tfrac{ \lceil t \rceil_{m^{\ell}} - t }{ ( T / m^{ \ell  } ) } \right]U^{\thetaBar_\ell}_{n-\ell,m}(\lfloor t \rfloor_{m^{\ell}}, \approximationX^{m^n,0,0}_{0,\lfloor t \rfloor_{m^{\ell}}})+
\left[ \tfrac{ t-\lfloor t \rfloor_{m^{\ell}} }{ ( T / m^{ \ell  } ) } \right]
U^{\thetaBar_\ell}_{n-\ell,m}(\lceil t \rceil_{m^{\ell}}, \approximationX^{m^n,0,0}_{0,\lceil t \rceil_{m^{\ell}}})
\biggr)\Biggr],
\end{align}
Then 
\begin{enumerate}[i)]
 \item \label{q01}there exists a unique measurable  $u\colon [0,T]\times\R^d\to\R$ which satisfies for all $t\in[0,T]$, $x\in \R^d$ that
$\E\bigl[|
g(X_{t,T}^{x})|\bigr]+
\int_{t}^{T} \E\bigl[|f( u(r,X_{t,r}^{x}))|
\bigr]\,d r+\sup_{r\in[0,T],\xi\in\R^d}\frac{|u(r,\xi)|}{V(r,\xi)}
<\infty$ and
$
u(t,x)=
\E\bigl[
g(X_{t,T}^{x})\bigr]+
\int_{t}^{T} \E\bigl[f( u(r,X_{t,r}^{x}))
\bigr]\,dr
$, 
\item \label{q02}it holds
for all $t\in[0,T]$, $s\in[0,T]\setminus\{t\}$  that
$
\left\lVert
u(t,X_{0,t}^0 )\right\rVert_{p}\leq 
2e^{2cT}
(V(0,0))^{\expFirstNorm}
$ and $
\tfrac{\sqrt{T}\left\lVert
u(t,X_{0,t}^0 )
-u(s,X_{0,s}^0 )
\right\rVert_{p}}{\lvert t-s\rvert^{\nicefrac{1}{2}}}
\leq 
12  e^{2cT}
 (V(0,0))^{\expSecondNorm+1}$, and
\item \label{k14}it holds for all $n,m\in\N$, $t\in [0,T]$  that
\begin{align}\label{k05a}
\left\lVert \Yappr^{n,m}_{t}-  u(t,X^0_{0,t}) \right\rVert_{\exponentLP}\leq 959 m  e^{6cT}
 e^{m^{p/2}/p}m^{-n/2}
74^{n} e^{4 n cT}
(\exponentLP-1)^{\frac{n}{2}} (V(0,0))^{\expSecondNorm+1}.
\end{align} 

\end{enumerate}
\end{theorem}

\begin{proof}[Proof of \cref{m01}]
Throughout  this proof
for every $s\in[0,T]$  and  
every
random field $H\colon [0,T]\times\R^d \times\Omega\to \R$ 
let 
$\tnorm{H}{s},\tnorm{H}{1,s},\tnorm{H}{2,s}\in[0,\infty]$  
$ $
 satisfy that
\begin{gather}\begin{split}
\tnorm{H}{s}&=\max \left\{\tnorm{H}{1,s},\tnorm{H}{2,s}\right\},\quad 
\tnorm{H}{1,s}=  \sup_{t\in[s,T]}
\sup_{x\in\R^d}\frac{\lVert H(t,x)\rVert_{\exponentLP}}{ 
(V(t,x))^{\expFirstNorm}}, \quad\text{and} 
\\
\tnorm{H}{2,s}&=\sup_{
\substack{
t_1,t_2\in[s,T],
x_1,x_2\in\R^d\colon \\(t_1,x_1)\neq (t_2,x_2)
} } \frac{\left\lVert H(t_1,x_1)-H(t_2,x_2)\right\rVert_{\exponentLP}}{
\tfrac{(V(t_1,x_1))^{\expSecondNorm}+(V(t_2,x_2))^{\expSecondNorm}}{2}
\left[
\tfrac{V(t_1,x_1)+V(t_2,x_2)}{2}
\tfrac{ \lvert t_1-t_2\rvert^{\nicefrac{1}{2}}}{\sqrt{T}}
+\tfrac{\lVert x_1-x_2\rVert}{\sqrt{T}}
\right]}.
\end{split}\label{k06}
\end{gather}
Then
\cref{p05}, the assumptions of \cref{m01}, the fact that
$\forall\, m\in \N\colon U^0_{0,m}=0$, and the fact that $\forall\,n\in \N_0\colon 2n\leq 2^n$
 show that
\begin{enumerate}[a)]\itemsep0pt
\item there exists a unique measurable  $u\colon [0,T]\times\R^d\to\R$ which satisfies for all $t\in[0,T]$, $x\in \R^d$ that
$\E\bigl[|
g(X_{t,T}^{x})|\bigr]+
\int_{t}^{T} \E\bigl[|f( u(r,X_{t,r}^{x}))|
\bigr]\,d r+\sup_{r\in[0,T],\xi\in\R^d}\frac{|u(r,\xi)|}{V(r,\xi)}<\infty$ and
$
u(t,x)=
\E\bigl[
g(X_{t,T}^{x})\bigr]+
\int_{t}^{T} \E\bigl[f( u(r,X_{t,r}^{x}))
\bigr]\,dr
$,
 \item for all $s\in[0,T]$ it holds  that 
\begin{align}
\tnorm{u}{1,s}\leq 2e^{cT},\quad \tnorm{u}{2,s}\leq 8 e^{2cT},\quad\text{and}\quad
\tnorm{u}{s}\leq 8 e^{2cT},\label{k07}
\end{align}
and
\item 
for all $n\in\N_0$, $m\in\N$ , 
$t\in[0,T]$ it holds that
\begin{align}\label{k09}
\tnorm{U^0_{n,m}-u}{t}\leq  888 m  e^{6cT}
 e^{m^{p/2}/p}m^{-n/2}74^n e^{4 n cT}
(\exponentLP-1)^{\frac{n}{2}} .
\end{align}
\end{enumerate} 
This shows \eqref{q01}.

Next, \eqref{s10} and \eqref{h21b}
show for all $ t\in [0,T]$, $n\in\N$ that
\begin{equation}
\begin{split}
&
\left\lVert \left\lVert 
 \approximationX^{n,0,0}_{0,t} -
X^{0}_{0,t}\right\rVert\right\rVert_{\exponentX}
=\xeqref{s10}
\left\lVert \left\lVert 
 \bigl( \approximationX^{n,0,0}_{0,t} -
X^{0}_{0,t}\bigr)
-   \bigl(\approximationX^{n,0,0}_{0,0} - X^{0}_{0,0}\bigr)\right\rVert\right\rVert_{\exponentX}
\leq\xeqref{h21b}
\tfrac{\sqrt{T}V(0,0)
}{2\sqrt{n}}.
\end{split}\label{m10}
\end{equation}
This, the triangle inequality, and \eqref{h21} show for all $s,t\in[0,T]$ that
\begin{align}
&\left\lVert \left\lVert 
X_{0,t}^{0}-
X_{0,s}^{0}\right\rVert\right\rVert_{\exponentX}
\leq
 \limsup_{ n\to\infty}\left[
\left\lVert \left\lVert 
X_{0,t}^{0}-
\approximationX_{0,t}^{n,0,0}
\right\rVert
\right\rVert_{\exponentX}
+\left\lVert \left\lVert 
\approximationX_{0,t}^{n,0,0}-
\approximationX_{0,s}^{n,0,0}\right\rVert\right\rVert_{\exponentX}
+\left\lVert\left\lVert \approximationX_{0,s}^{n,0,0}-
X_{0,s}^{0}\right\rVert\right\rVert_{\exponentX}\right] \nonumber\\
&\leq\xeqref{h21} V(0,0)
\tfrac{\lvert t-s\rvert^{\nicefrac{1}{2}}}{2}.\label{t11}
\end{align}
In addition, \eqref{m10}, continuity of $V$, and \eqref{s10} show for all $t\in[0,T]$ that
\begin{align}\label{t10}
\bigl\lVert V\bigl(t,X_{0,t}^{0}\bigr)\bigr\rVert_{\exponentV}= 
\lim_{n\to\infty}
\bigl\lVert V\bigl(t,\approximationX_{0,t}^{n,0,0}\bigr)\bigr\rVert_{\exponentV}
\leq V(0,0).
\end{align}
This, \eqref{k06}, and \eqref{k07}
show for all $t\in[0,T]$ that
\begin{align}
\left\lVert
u(t,X_{0,t}^0 )\right\rVert_{p}\leq\xeqref{k06}  \tnorm{u}{1,t}\left\lVert
(V(t,X_{0,t}^0 ))^{\expFirstNorm}\right\rVert_{\exponentLP}\leq \xeqref{k07}
2e^{2cT}
\xeqref{t10}
(V(0,0))^{\expFirstNorm}.\label{t12}
\end{align}
Furthermore, \eqref{k06}, the triangle inequality, H\"older's inequality, the fact that
$\max\{\frac{\expSecondNorm}{\exponentV}+\frac{1}{\exponentV}, \frac{\expSecondNorm}{\exponentV}+\frac{1}{\exponentX}\}\leq \frac{1}{\exponentLP} $,
\eqref{k07}, \eqref{t10}, and \eqref{t11} show for all $t,s\in[0,T]$ that
\begin{align}
&
\left\lVert
u(t,X_{0,t}^0 )
-u(s,X_{0,s}^0 )
\right\rVert_{p}\leq \xeqref{k06} \tnorm{u}{t}\left\lVert
\tfrac{(V(t,X_{0,t}^0 ))^{\expSecondNorm}
+(V(s,X_{0,s}^0 ))^{\expSecondNorm}}{2}
\left[\tfrac{V(s,X_{0,s}^0)+V(t,X_{0,t}^0)}{2}
\tfrac{\lvert t-s\rvert^{\nicefrac{1}{2}}}{\sqrt{T}}
+\tfrac{\lVert X_{0,t}^0 -X_{0,s}^0\rVert}{\sqrt{T}}
\right]
\right\rVert_{p}\nonumber\\
&\leq 
\tnorm{u}{t}\left\lVert
\tfrac{(V(t,X_{0,t}^0 ))^{\expSecondNorm}
+(V(s,X_{0,s}^0 ))^{\expSecondNorm}}{2}
\right\rVert_{\frac{\exponentV}{\expSecondNorm}}\left[
\tfrac{\left\lVert V(s,X_{0,s}^0 )\right\rVert_{\exponentV}+
\left\lVert V(t,X_{0,t}^0 )\right\rVert_{\exponentV}
}{2}
\tfrac{\lvert t-s\rvert^{\nicefrac{1}{2}}}{\sqrt{T}}
+\left\lVert\tfrac{\lVert X_{0,t}^0 -X_{0,s}^0\rVert}{\sqrt{T}}\right\rVert_{\exponentX}
\right]\nonumber\\
&\leq \xeqref{k07}
8 e^{2cT}
 \xeqref{t10}(V(0,0))^{\expSecondNorm}\left[
V(0,0)
\tfrac{\lvert t-s\rvert^{\nicefrac{1}{2}}}{\sqrt{T}}+
V(0,0)\xeqref{t11}
\tfrac{\lvert t-s\rvert^{\nicefrac{1}{2}}}{2\sqrt{T}}
\right]=12  e^{2cT}
 (V(0,0))^{\expSecondNorm+1}\tfrac{\lvert t-s\rvert^{\nicefrac{1}{2}}}{\sqrt{T}}.\label{t20}
\end{align}
This  shows for all $t\in[0,T]$, $s\in[0,T]\setminus\{t\}$ that
\begin{align}\begin{split}
&
\tfrac{\sqrt{T}\left\lVert
u(t,X_{0,t}^0 )
-u(s,X_{0,s}^0 )
\right\rVert_{p}}{\lvert t-s\rvert^{\nicefrac{1}{2}}}
\leq 
\tfrac{\sqrt{T}}{\lvert t-s\rvert^{\nicefrac{1}{2}}}
\xeqref{t20}
12  e^{2cT}
 (V(0,0))^{\expSecondNorm+1}\tfrac{\lvert t-s\rvert^{\nicefrac{1}{2}}}{\sqrt{T}}= 
12  e^{2cT}
 (V(0,0))^{\expSecondNorm+1}.\label{t14}\end{split}
\end{align}
This and \eqref{t12} prove \eqref{q02}.

Moreover, \eqref{k06}, \eqref{k07}, H\"older's inequality, 
the fact that
$\frac{\expSecondNorm}{\exponentV} + \frac{1}{\exponentX}\leq \frac{1}{\exponentLP}$,
the triangle inequality,  \eqref{s10}, \eqref{t10}, and \eqref{m10} show for all $t\in[0,T]$, $m,n\in \N$ that
\begin{align}\begin{split}
&\left\lVert
u(t,\approximationX^{m^n,0,0}_{0,t})-
u(t,X^0_{0,t})\right\rVert_{\exponentLP}\leq
\xeqref{k06}
 \tnorm{u}{t}
 \left\lVert 
\tfrac{ \left(V(t,\approximationX^{m^n,0,0}_{0,t})\right)^{\expSecondNorm}
+\left(V(t,X^0_{0,t})  \right)^{\expSecondNorm}}{2}
\tfrac{\left\lVert \approximationX^{m^n,0,0}_{0,t}-X^0_{0,t}\right\rVert}{\sqrt{T}}
\right\rVert_{\exponentLP}\\
&
\leq \xeqref{k07}8 e^{2cT}
\left\lVert 
\tfrac{ \left(V(t,\approximationX^{m^n,0,0}_{0,t})\right)^{\expSecondNorm}
+
\left(V(t,X^0_{0,t})  \right)^{\expSecondNorm}}{2}
\right\rVert_{\frac{\exponentV}{\expSecondNorm}}
\left\lVert 
\tfrac{\left\lVert \approximationX^{m^n,0,0}_{0,t}-X^0_{0,t}\right\rVert}{\sqrt{T}}
\right\rVert_{\exponentX}\\
&
\leq 8 e^{2cT}
\xeqref{s10}\xeqref{t10}
(V(0,0))^{\expSecondNorm}\xeqref{m10}\tfrac{V(0,0)}{2\sqrt{m^n}}
= 4e^{2cT}m^{-n/2}
(V(0,0))^{\expSecondNorm+1}.
\end{split}\label{k12}\end{align}
Next, \eqref{k06},  the triangle inequality, H\"older's inequality, 
the fact that
$\frac{\expSecondNorm}{\exponentV}
+\frac{1}{\exponentV}
\leq \frac{1}{\exponentLP}$, the fact that
$\frac{\expSecondNorm}{\exponentV}+\frac{1}{\exponentX}\leq \frac{1}{\exponentLP}$, \eqref{k09}, \eqref{s10}, and \eqref{h21}
show for all $n\in \N_0$, $\ell\in \{0,1,\ldots,n\}$, $s\in [0,T]$,
$t\in[s,T]$
 that
\begin{align}
&
\left\lVert (U^{\thetaBar_{n-\ell}}_{n-\ell,m}-u)(t,\approximationX^{m^n,0,0}_{0,t})-(U^{\thetaBar_{n-\ell}}_{n-\ell,m}-u)(s,\approximationX^{m^n,0,0}_{0,s})\right\rVert_{\exponentLP}\nonumber\\
&\leq \xeqref{k06}
\tnorm{U^{\thetaBar_{n-\ell}}_{n-\ell,m}-u}{s}\nonumber\\
&\qquad
\left\lVert
\tfrac{\left(V(t,\approximationX^{m^n,0,0}_{0,t})\right)^{\expSecondNorm}
+
\left(V(s,\approximationX^{m^n,0,0}_{0,s})\right)^{\expSecondNorm}}{2}
\left[\tfrac{V(s,\approximationX^{m^n,0,0}_{0,s})
+V(t,\approximationX^{m^n,0,0}_{0,t})}{2}
\tfrac{
\lvert t-s\rvert^{\nicefrac{1}{2}}}{\sqrt{T}}
+
\tfrac{\left\lVert
\approximationX^{m^n,0,0}_{0,t}-
\approximationX^{m^n,0,0}_{0,s}\right\rVert}{\sqrt{T}}
\right]
\right\rVert_{\exponentLP}\nonumber \\
&\leq \tnorm{U^{\thetaBar_{n-\ell}}_{n-\ell,m}-u}{s}\left\lVert
\tfrac{\left(V(t,\approximationX^{m^n,0,0}_{0,t})\right)^{\expSecondNorm}
+
\left(V(s,\approximationX^{m^n,0,0}_{0,s})\right)^{\expSecondNorm}}{2}\right\rVert_{\frac{\exponentV}{\expSecondNorm}}\nonumber\\
&\qquad
\left[
\tfrac{\left\lVert V(s,\approximationX^{m^n,0,0}_{0,s})\right\rVert_{\exponentV}+
\left\lVert V(t,\approximationX^{m^n,0,0}_{0,t})\right\rVert_{\exponentV}
}{2}
\tfrac{
\lvert t-s\rvert^{\nicefrac{1}{2}}}{\sqrt{T}}+
\left\lVert
\tfrac{\left\lVert
\approximationX^{m^n,0,0}_{0,t}-
\approximationX^{m^n,0,0}_{0,s}\right\rVert}{\sqrt{T}}\right\rVert_{\exponentX}
\right]\nonumber\\
&\leq\xeqref{k09} 888 m  e^{6cT}
 e^{m^{p/2}/p}m^{-(n-\ell)/2}74^{n-\ell} e^{4 (n-\ell) cT}
(\exponentLP-1)^{\frac{n-\ell}{2}}\xeqref{s10}(V(0,0))^{\expSecondNorm} \left[
\tfrac{\xeqref{s10}V(0,0)\lvert t-s\rvert^{\nicefrac{1}{2}}}{\sqrt{T}}+\xeqref{h21}\tfrac{V(0,0)\lvert t-s\rvert^{\nicefrac{1}{2}}}{2\sqrt{T}}
\right]\nonumber\\
&= 1332 m  e^{6cT}
 e^{m^{p/2}/p}m^{-(n-\ell)/2}74^{n-\ell} e^{4 (n-\ell) cT}
(\exponentLP-1)^{\frac{n-\ell}{2}} \tfrac{(V(0,0))^{\expSecondNorm+1}\lvert t-s\rvert^{\nicefrac{1}{2}}}{\sqrt{T}}.\label{k11}
\end{align}
Next, \eqref{k06}, \eqref{k09}, \eqref{s10}, Jensen's inequality, and the fact that
$ \expFirstNorm \exponentLP\leq \exponentV$ show for all
$n,m\in\N $, $t\in[0,T]$ that
\begin{align}\begin{split}
&
\left\lVert U^{\thetaBar_n}_{n,m}(t,\approximationX^{m^n,0,0}_{0,t})-u(t,\approximationX^{m^n,0,0}_{0,t})\right\rVert_{\exponentLP}
\leq \xeqref{k06}
\tnorm{U^{\thetaBar_n}_{n,m} -u}{t} \left\lVert\left(V(t, \approximationX^{m^n,0,0}_{0,t})\right)^{\expFirstNorm}\right\rVert_{\exponentLP}
\\
&
\leq\xeqref{k09} 888 m  e^{6cT}
 e^{m^{p/2}/p}m^{-n/2}74^n e^{4 n cT}
(\exponentLP-1)^{\frac{n}{2}} \xeqref{s10}(V(0,0))^{\expFirstNorm}.
\end{split}\label{k11b}\end{align}
This, \cite[Lemma~2.3]{hutzenthaler2021overcoming}
 (applied for every $n,m\in\N$ with 
$V\gets 
\{\mathfrak{Z}\colon \Omega\to\R \colon\mathfrak{Z} \text{ is }
\text{measurable}
\}
$,
$\lVert \cdot \rVert \gets  \lVert\cdot\rVert_{\exponentLP}$,
$\alpha\gets 1/2$,
$(m_\ell)_{\ell\in \{1,2,\ldots,n\}}\gets (m^\ell)_{\ell\in \{1,2,\ldots,n\}}$,
$
(\tau_{\ell,k})_{k\in  \{0,1,\ldots,m_\ell\}, \ell\in \{1,2,\ldots,n\}}
\gets 
(\frac{kT}{m^\ell})_{k\in  \{0,1,\ldots,m^\ell\}, \ell\in \{1,2,\ldots,n\}}
$,
$(Y^0_t)_{t\in[0,T]}\gets (u(t,\approximationX^{m^n,0,0}_{0,t}))_{t\in[0,T]}$,
$((Y^\ell_t)_{t\in[0,T]}
)_{\ell\in[1,n]\cap\N}
\gets
((U^{\thetaBar_\ell}_{\ell,m}(t,\approximationX^{m^n,0,0}_{0,t}))_{t\in[0,T]})_{\ell\in[1,n]\cap\N}$,
$\Yappr \gets \Yappr^{n,m}$ in the notation  of \cite[Lemma~2.3]{hutzenthaler2021overcoming}),
the fact that
$\forall\,m\in\N,\theta\in\Theta\colon U_{0,m}^\theta=0$,
\eqref{k11}, the fact that $V\geq 1$, 
the fact that $888\leq \frac{1332}{\sqrt{2}}$,
and
the fact that $\forall\,n\in \N\colon 
\frac{1332}{\sqrt{2}}\sum_{\ell=0}^{n}74^{n-\ell}= \frac{1332}{\sqrt{2}}\frac{74^{n+1}-1}{73}\leq 955\cdot 74^n$
 show  for all $n,m\in\N$  that
\begin{align}
&
\sup_{t\in[0,T]} \left\lVert\Yappr^{n,m}_t-u(t,\approximationX^{m^n,0,0}_{0,t})\right\rVert_{\exponentLP}
\leq \sup_{t\in[0,T]} \left\lVert U^{\thetaBar_n}_{n,m}(t,\approximationX^{m^n,0,0}_{0,t})-u(t,\approximationX^{m^n,0,0}_{0,t})\right\rVert_{\exponentLP} \nonumber\\
&\quad
 + T^{1/2}2^{-1/2}m^{-n/2}
\left[\sup_{t,s\in[0,T]\colon t\neq s}\tfrac{ \left\lVert u(t,\approximationX^{m^n,0,0}_{0,t})-u(s,\approximationX^{m^n,0,0}_{0,s})\right\rVert_{\exponentLP}}{\lvert t-s\rvert^{\nicefrac{1}{2}}}\right] \nonumber\\
& \quad+
\sum_{\ell=1}^{n-1}\left[T^{1/2}2^{-1/2}m^{-\ell/2}\left[\sup_{t,s\in[0,T]\colon t\neq s}
\tfrac{ \left\lVert
(U^{\thetaBar_{n-\ell}}_{n-\ell,m}-u)(t,\approximationX^{m^n,0,0}_{0,t})
-
(U^{\thetaBar_{n-\ell}}_{n-\ell,m}-u)(s,\approximationX^{m^n,0,0}_{0,s})
\right\rVert_{\exponentLP}}{\lvert t-s\rvert^{\nicefrac{1}{2}}}\right]\right] \nonumber\\
&= \sup_{t\in[0,T]} \left\lVert U^{\thetaBar_n}_{n,m}(t,\approximationX^{m^n,0,0}_{0,t})-u(t,\approximationX^{m^n,0,0}_{0,t})\right\rVert_{\exponentLP} \nonumber\\
&\quad+\sum_{\ell=1}^{n}\left[T^{1/2}2^{-1/2}m^{-\ell/2}\left[\sup_{t,s\in[0,T]\colon t\neq s}
\tfrac{ \left\lVert
(U^{\thetaBar_{n-\ell}}_{n-\ell,m}-u)(t,\approximationX^{m^n,0,0}_{0,t})
-
(U^{\thetaBar_{n-\ell}}_{n-\ell,m}-u)(s,\approximationX^{m^n,0,0}_{0,s})
\right\rVert_{\exponentLP}}{\lvert t-s\rvert^{\nicefrac{1}{2}}}\right]\right]
 \nonumber\\
&\leq \xeqref{k11b}
888 m  e^{6cT}
 e^{m^{p/2}/p}m^{-n/2}74^n e^{4 n cT}
(\exponentLP-1)^{\frac{n}{2}} (V(0,0))^{\expFirstNorm}\nonumber\\
&\quad +\sum_{\ell=1}^{n}\left[T^{1/2}2^{-1/2}m^{-\ell/2}\xeqref{k11}1332 m  e^{6cT}
 e^{m^{p/2}/p}m^{-(n-\ell)/2}74^{n-\ell} e^{4 (n-\ell) cT}
(\exponentLP-1)^{\frac{n-\ell}{2}} \tfrac{(V(0,0))^{\expSecondNorm+1}
}{\sqrt{T}}\right]\nonumber\\
&\leq \sum_{\ell=0}^{n}\left[\tfrac{1332}{\sqrt{2}} m  e^{6cT}
 e^{m^{p/2}/p}m^{-n/2}74^{n-\ell} e^{4 (n-\ell) cT}
(\exponentLP-1)^{\frac{n-\ell}{2}} (V(0,0))^{\expSecondNorm+1}\right]\nonumber\\
&\leq 
 m  e^{6cT}
 e^{m^{p/2}/p}m^{-n/2}\left[\tfrac{1332}{\sqrt{2}}\sum_{\ell=0}^{n}74^{n-\ell}\right]
 e^{4 n cT}
(\exponentLP-1)^{\frac{n}{2}} (V(0,0))^{\expSecondNorm+1}\nonumber\\
&\leq 
955 m  e^{6cT}
 e^{m^{p/2}/p}m^{-n/2}
74^{n} e^{4 n cT}
(\exponentLP-1)^{\frac{n}{2}} (V(0,0))^{\expSecondNorm+1}.\label{t23}
\end{align}
This, the triangle inequality, and
\eqref{k12} show for all $n,m\in\N$, $t\in [0,T]$ that
\begin{align}
&
\left\lVert
\Yappr^{n,m}_{t}- u(t,X^0_{0,t})\right\rVert_{\exponentLP}\leq 
\left\lVert \Yappr^{n,m}_{t}- u(t,\approximationX^{m^n,0,0}_{0,t}) \right\rVert_{\exponentLP}+
\left\lVert
u(t,\approximationX^{m^n,0,0}_{0,t})-
u(t,X^0_{0,t})\right\rVert_{\exponentLP}\nonumber\\
&\leq \xeqref{t23}955 m  e^{6cT}
 e^{m^{p/2}/p}m^{-n/2}
74^{n} e^{4 n cT}
(\exponentLP-1)^{\frac{n}{2}} (V(0,0))^{\expSecondNorm+1}+\xeqref{k12}
4e^{2cT}m^{-n/2}
(V(0,0))^{\expSecondNorm+1}\nonumber\\
&\leq 959 m  e^{6cT}
 e^{m^{p/2}/p}m^{-n/2}
74^{n} e^{4 n cT}
(\exponentLP-1)^{\frac{n}{2}} (V(0,0))^{\expSecondNorm+1}.\label{k14b}
\end{align}
This shows \eqref{k14}.
The proof of \cref{m01}
is thus completed.
\end{proof}

We discuss the 
computational effort of the approximation method in \cref{s02} below; see
\eqref{c10}--\eqref{c09}. First, \eqref{c10} was explained in the paragraph above \cref{c01}. Next, \eqref{c09} is explained as follows.
For every $d,m,n\in\N$ we denote by $\FEY_{d,m,n}$ the upper bound for
the computational effort to have one realization of the process
$(\Yappr^{d,n,m}_{ {kT}/{m^n}})_{k\in\{0,1,2\ldots,m^n\}} $. For every $d,m,n\in\N$ we think of
$a_3(d)(m^n+1)$ as the upper bound for the computational effort 
to have 
one realization of the process 
$(\approximationX^{d,m^n,0,0}_{0,kT/m^n})_{k\in\{0,1,\ldots,m^n\}}$.
Here, we use Euler--Maruyama approximations; see \eqref{p15}.
Furthermore, in \eqref{c10a}
  for every $d,m,n\in\N$, $\ell\in \{0,1,\ldots,n-1\}$ we need
to evaluate 
$U^{d,\ell}_{d,n-\ell,m}$ at $m^\ell+1 $ space-time points.

\begin{corollary}[Complexity analysis]\label{s02}
Let $\lVert \cdot\rVert\colon \bigcup_{k,\ell\in\N}\R^{k\times \ell}\to[0,\infty)$ satisfy for all $k,\ell\in\N$, $s=(s_{ij})_{i\in[1,k]\cap\N,j\in [1,\ell]\cap\N}\in\R^{k\times \ell}$ that
$\lVert s\rVert^2=\sum_{i=1}^{k}\sum_{j=1}^{\ell}\lvert s_{ij}\rvert^2$,
let $T,\delta \in (0,\infty)$,
$\beta,c,p\in [2,\infty)$,
 $\Theta=\cup_{n\in\N}\Z^n$,
$b,a_1,a_2,a_3\colon \N\to\N$ satisfy that
 $p>2/\delta$ and
$
 \inf_{ \gamma\in (0,\infty)}\sup_{d\in\N}
\tfrac{b(d)+
\sum_{i=1}^{3}a_i(d)}{\gamma d^\gamma}\leq 1
,
$
for every $d\in\N$ let
$f_d\in C^2(\R,\R)$, 
$g_d\in C^2(\R^d,\R)$,
$u_d\in C^{1,2}([0,T]\times\R^d,\R)$,
$\mu_d=(\mu_{d,i})_{i\in[1,d]\cap\Z}\in C^2(\R^d,\R^d)$,
$\sigma_d =(\sigma_{d,i,j})_{i,j\in[1,d]\cap\Z}\in C^2(\R^d,\R^{d\times d})$, 
assume for all
$d\in\N$,
 $x,y,z\in\R^d$, 
 $w\in\R$, $t\in[0,T]$ that
\begin{align}\label{p17}
\lvert Tf_d(0)\rvert+\lVert \mu_d(0)\rVert+\lVert\sigma_d(0)\rVert\leq b(d),\quad 
\lvert u_d(t,x)\rvert +
\lvert g_d(x)\rvert\leq \bigl[b(d)+c^2\lVert x\rVert^2\bigr]^{\beta},
\end{align}
\begin{align}
\lvert (\totalD g_d (x))(y)\rvert\leq b(d)\lVert y\rVert,\quad  \lvert (\totalD^2 g_d(x))(y,z)\rvert\leq b(d)\lVert y\rVert
\lVert  z\rVert,\quad \lvert f_d'(w)\rvert+\lvert f_d''(w)\rvert\leq c,\label{p18}
\end{align}
\begin{align}
\lVert (\totalD \mu_d (x))(y)\rVert\leq c\lVert y\rVert,\quad 
\lVert (\totalD \sigma_d (x))(y)\rVert\leq c\lVert y\rVert,\label{p19}
\end{align}
\begin{align}\label{p20}
\lVert (\totalD^2 \mu_d (x))(y,z)\rVert\leq b(d)\lVert y\rVert\lVert z\rVert,\quad 
\lVert (\totalD^2 \sigma_d (x))(y,z)\rVert\leq b(d)\lVert y\rVert\lVert z\rVert,
\end{align}
\begin{align}
(\tfrac{\partial}{\partial t}u_d)(t,x)+\left\langle\mu_d(x), (\nabla_xu_d)(t,x)\right\rangle+\tfrac{1}{2} \mathrm{trace}\left(
\sigma_{d} (x)
(\sigma_{d}(x))^*
(\mathrm{Hess}_xu_d)(t,x)\right)=
-f_d(u_d(t,x)),\label{p10}
\end{align}
and $u_d(T,x)=g_d(x)$,
let $(\Omega,\mathcal{F},\P, (\F_t)_{t\in [0,T]})$ be a filtered probability space which satisfies the usual conditions,
for every $s\in[1,\infty)$, 
$k,\ell\in\N$
and every random variable $\mathfrak{X}\colon \Omega\to\R^{k\times \ell}$ let $\lVert\mathfrak{X}\rVert_s\in[0,\infty]$ satisfy that $\lVert\mathfrak{X}\rVert_s^s=\E [\lVert\mathfrak{X}\rVert^s]$,
let
$\unif^\theta\colon \Omega\to[0,1]$,
$\theta\in \Theta$, be i.i.d.\ random variables which satisfy for all
 $t\in [0,1]$
that
$\P(\unif^0 \leq t)= t$,
let $W^{d,\theta} \colon [0,T]\to \R^d$, $d\in\N$, $\theta\in\Theta$, be independent standard $(\F_t)_{t\in [0,T]}$-Brownian motions with continuous sample paths, assume that
$(\unif^{\theta})_{\theta\in\Theta} $ and
$(W^{d,\theta})_{d\in\N,\theta\in\Theta}$ are independent, for every $d,n\in\N$,
$\theta\in\Theta$,
$s\in[0,T]$,
$x\in\R^d$
let
$(\approximationX^{d,n,\theta,x}_{s,t})_{t\in[s,T]}
\colon [s,T] \times\R^d\times \Omega\to\R^d$ satisfy for all 
$k\in [0,n-1]\cap\Z$,
$t\in\bigl[\max\{s,\frac{kT}{n}\},\max\{s,\frac{(k+1)T}{n}\}\bigr]$
 that $\approximationX^{d,n,\theta,x}_{s,s}=x$ and
\begin{align}\small\begin{split}
\approximationX^{d,n,\theta,x}_{s,t}&=
\approximationX^{d,n,\theta,x}_{s,\max\{s,\frac{kT}{n}\}}+ \mu_d\bigl(
\approximationX^{d,n,\theta,x}_{s,\max\{s,\frac{kT}{n}\}} \bigr)\bigl(t-\max\{s,\tfrac{kT}{n}\}\bigr) +
\sigma_d\bigl(
\approximationX^{d,n,\theta,x}_{s,\frac{kT}{n}} \bigr)\bigl(W^{d,\theta}_t-W^{d,\theta}_{\max\{s,\frac{kT}{n}\}}\bigr),\end{split}\label{p15}
\end{align}
let 
$ 
  {U}_{ n,m}^{d,\theta } \colon [0, T] \times \R^d \times \Omega \to \R
$,
$d\in\N$,
$n,m\in\Z$, $\theta\in\Theta$, satisfy for all $d,n,m\in \N$, $\theta\in\Theta$, $t\in[0,T]$, $x\in\R^d$ that
$
{U}_{-1,m}^{d,\theta}(t,x)={U}_{0,m}^{d,\theta}(t,x)=0$ and
\begin{align}
  {U}_{n,m}^{d,\theta}(t,x)
=  &\sum_{\ell=0}^{n-1}\frac{1}{m^{n-\ell}}
    \sum_{i=1}^{m^{n-\ell}}
\Biggl[
      g_d\bigl(\approximationX^{d,m^\ell,(\theta,\ell,i),x}_{t,T}\bigr)-\1_{\N}(\ell)
g_d \bigl(\approximationX^{d,m^{\ell-1},(\theta,\ell,i),x}_{t,T}\bigr)\nonumber \\
 &+(T-t)
     \bigl( f_d\circ {U}_{\ell,m}^{d,(\theta,\ell,i)}\bigr)
\left(t+(T-t)\unif^{(\theta,\ell,i)},\approximationX_{t,t+(T-t)\unif^{(\theta,\ell,i)}}^{d,m^\ell,(\theta,\ell,i),x}\right)\nonumber \\
&-\1_{\N}(\ell)(T-t)
\bigl(f_d\circ {U}_{\ell-1,m}^{d,(\theta,\ell,-i)}\bigr)
    \left(t+(T-t)\unif^{(\theta,\ell,i)},\approximationX_{t,t+(T-t)\unif^{(\theta,\ell,i)}}^{d,m^{\ell-1},(\theta,\ell,i),x}   \right) \Biggr],\label{p15b}
\end{align}
 let $\lfloor \cdot \rfloor_m \colon \R \to \R$, $ m \in \N $, and 
$\lceil \cdot \rceil_m \colon  \R \to \R$, $ m \in \N $, 
satisfy for all $m \in \N$, $t \in [0,T]$ that
$\lfloor t \rfloor_m = \max( ([0,t]\backslash \{T\}) \cap \{ 0, \frac{ T }{ m }, \frac{ 2T }{ m }, \ldots \} )$
and 
$\lceil t \rceil_m = \min(((t,\infty) \cup \{T\})\cap  \{ 0, \frac{ T }{ m }, \frac{ 2T }{ m }, \ldots \} )$,
let $\Yappr^{d,n,m}=(\Yappr^{d,n,m}_t)_{t\in[0,T]}
\colon [0,T]\times\Omega\to\R $,
 $d,n,m\in\N$, 
satisfy  for all $d,n,m\in\N$, $t\in[0,T]$  that
\begin{align}\label{c10a}
&\Yappr^{d,n,m}_{t}
= \sum_{\ell=0}^{n-1}\Biggl[
\left[ \tfrac{ \lceil t \rceil_{m^{\ell+1}} - t }{ ( T / m^{ \ell + 1 } ) } \right]
U^{d,\ell}_{n-\ell,m}(\lfloor t \rfloor_{m^{\ell+1}}, \approximationX^{d,m^n,0,0}_{0,\lfloor t \rfloor_{m^{\ell+1}}})+
\left[\tfrac{ t-\lfloor t \rfloor_{m^{\ell+1}} }{ ( T / m^{ \ell + 1 } ) }\right]
U^{d,\ell}_{n-\ell,m}(\lceil t \rceil_{m^{\ell+1}}, \approximationX^{d,m^n,0,0}_{0,\lceil t \rceil_{m^{\ell+1}}})
\nonumber
\\
&-\1_{\N}(\ell)\biggl(
\left[ \tfrac{ \lceil t \rceil_{m^{\ell}} - t }{ ( T / m^{ \ell  } ) } \right]U^{d,\ell}_{n-\ell,m}(\lfloor t \rfloor_{m^{\ell}}, \approximationX^{d,m^n,0,0}_{0,\lfloor t \rfloor_{m^{\ell}}})+
\left[ \tfrac{ t-\lfloor  t \rfloor_{m^{\ell}} }{ ( T / m^{ \ell  } ) } \right]
U^{d,\ell}_{n-\ell,m}(\lceil t \rceil_{m^{\ell}}, \approximationX^{d,m^n,0,0}_{0,\lceil t \rceil_{m^{\ell}}})
\biggr)\Biggr],
\end{align}
%
let 
 $
 ( {\FEU}_{d, n,m})_{d,n,m\in\Z}\subseteq  \N_0
$, 
$(\FEY_{d,n,m})_{d,n,m\in\Z}\subseteq \N_0$
satisfy for all $ d,n,m\in\N$ that
$
{\FEU}_{d,0,m}=0$, 
\begin{align}
&   {\FEU}_{d,n,m}\leq  \sum_{\ell=0}^{n-1}\left[
m^{n-\ell}\left[
4a_1(d)+a_2(d) m^\ell+
     {\FEU}_{d,\ell,m}
+\1_{\N}(\ell){\FEU}_{d,\ell-1,m}\right]\right]
,\label{c10}
\end{align}
\begin{align}
& \FEY_{d,n,m}\leq a_3(d)(m^n+1)+\sum_{\ell=0}^{n-1} \left[(m^{\ell+1}+1)\FEU_{d,n-\ell,m}\right],\label{c09}
\end{align} 
and
let
$M\colon \N\to\N$   satisfy that
\begin{align}
\limsup_{n\to\infty}\tfrac{1}{M(n)}=0
\quad\text{and}\quad
 \sup_{n\in\N}\left[\tfrac{M(n+1)}{M(n)}+\tfrac{(M(n))^{p/2}}{n}\right]<\infty.
\label{c08}
\end{align}
Then 
\begin{enumerate}[(i)]
\item\label{s01} for all $d\in\N$
there exists  an up to indistinguishability unique continuous random field
$(X^{d,x}_{s,t})_{s\in[0,T],t\in[s,T],x\in\R^d}
\colon \{(\mathfrak{s},\mathfrak{t})\in [0,T]^2\colon \mathfrak{s}\leq \mathfrak{t}\} \times\R^d\times \Omega\to\R^d
$  
such that 
for all
$s\in[0,T]$, 
$x\in\R^d$ it holds that
$(X^{d,x}_{s,t})_{t\in [s,T]}$ is $ (\F_t)_{t\in [s,T]}$-adapted
and such that
for all 
$s\in[0,T]$,
$t\in[s,T]$, $x\in\R^d$ it holds a.s.\
 that
\begin{align}
X^{d,x}_{s,t}= x+\int_{s}^{t} \mu_d(X^{d,x}_{s,r})\,dr+\int_{s}^{t} \sigma_d(X^{d,x}_{s,r})\,dW^{d,0}_r,
\end{align}
\item\label{s01b} for all $d\in \N$, $t\in[0,T]$, $x\in \R^d$ it holds that
$\E\bigl[\lvert
g_d(X_{t,T}^{d,x})\rvert\bigr]+
\int_{t}^{T} \E\bigl[\lvert f( u_d(r,X_{t,r}^{d,x}))\rvert
\bigr]\,d r<\infty$ and
$
u_d(t,x)=
\E\bigl[
g_d(X_{t,T}^{d,x})\bigr]+
\int_{t}^{T} \E\bigl[f_d( u_d(r,X_{t,r}^{d,x}))
\bigr]\,dr
$,
\item\label{s08} for all $d\in \N$
there exists a unique
 $(\F_{t})_{t\in[0,T]}$-predictable stochastic process $\mathbf{Y}^d=(Y^d,Z^{d,1},Z^{2,d},\ldots,Z^{d,d})\colon [0,T]\times\Omega\to \R^{d+1}
$ 
such that
$
\int_0^T  \E \bigl[\lvert Y_s^d\rvert+\sum_{j=1}^{d}\lvert Z_s^{d,j}\rvert^2\bigr]ds<\infty
$ and such that
for all  $t\in[0,T]$ it holds a.s.\ that
\begin{equation}
Y^d_t=g_d(X^{d,0}_{0,T})+\int_t^T f_d(Y^d_s)\,ds-\sum_{j=1}^{d}\int_t^T Z_s^{d,j}\, dW_s^{d,j},
\end{equation}
\item \label{s04}for all $d\in\N$, $t\in[0,T]$ it holds a.s.\ that
$u_d(t,X_{0,t}^{d,0})=Y^d_t $, and
\item \label{s03}there exist
$\gamma\in(0,\infty)$,
 $\mathsf{n}\colon \N\times(0,1)\to \N$ such that for all 
$d\in\N$,
$\epsilon\in (0,1)$, $n\in [\mathsf{n}(d,\epsilon),\infty)\cap\N$ it holds that
$
\left(\E\!\left[
\sup_{t\in[0,T]}
\left\lvert \Yappr^{d,n,M(n)}_{t}-  Y^d_t \right\rvert^{\exponentLP}\right]\right)^{\nicefrac{1}{p}}<\epsilon$
 and 
$
\epsilon^{2+\delta}
{\FEY}_{d,n,M(n)}
\leq \gamma d^\gamma
$.
\end{enumerate}
\end{corollary}
\begin{proof}[Proof \cref{s02}]First,
a standard result on stochastic differential equations with Lipschitz continuous coefficients
(see, e.g.,
\cite[Theorem~4.5.1]{Kunita1990}) 
and \eqref{p19} show \eqref{s01}.

Throughout the rest of this proof 
let $q,\bar{c}\in \R$ satisfy that $q=12\beta p$ and $\bar{c}=  16q^2c^2$ and for every $d\in\N$ let $\varphi_d\colon \R^d\to[1,\infty) $,
$V_d\colon [0,T]\times\R^d\to[1,\infty)  $
 satisfy 
for all 
$t\in[0,T]$,
$x\in\R^d$ that
\begin{align}
\varphi_d(x)= 2^{2q}\Bigl(\lvert b(d)\rvert^2+ c^2\lVert x \rVert^2\Bigr)^{q}
,\label{p13}
\end{align}
and
\begin{align}\label{p16}
V_d(t,x)=\left[
62 (1+b(d)+c)(c+1)
\left[\sqrt{T}+2q\right]^7
e^{5 c^2\left[\sqrt{T}+2q\right]^2 T}
 e^{\frac{4.5\bar{c} T}{2q}}\right]
e^{\frac{1.5\bar{c}\beta(T-t)}{ q}}
(\varphi_d(x))^{\frac{\beta}{ q}}.
\end{align}
First,
\eqref{p17}, \eqref{p13}, and \eqref{p16} show for all 
$d\in\N$,  $t\in[0,T]$, $x\in\R^d$ that
  \begin{equation}\begin{split}
&\max\bigl\{\lvert T f_d (0)\rvert,\lvert u_d(t,x)\rvert+
\lvert g_d(x)\rvert\bigr\}\leq\xeqref{p17}
\bigl[b(d)+c^2\lVert x\rVert^2\bigr]^{\beta}\leq \xeqref{p13}
(\varphi_d(x))^{\frac{\beta}{q}}\leq\xeqref{p16}  V_d(t,x).
\end{split}\label{s05}\end{equation}
Next, \eqref{p18},
\cite[Lemma 3.3]{HN21},
and the fact that $\forall\,d\in\N\colon b(d)\leq\min \{\frac{V_d}{\sqrt{T}},\frac{V_d}{T}\}$
 show that for all $d\in\N$, $x,\tilde{x},y,\tilde{y}\in\R^d$,
$v,w,\tilde{v},\tilde{w}\in\R$
  it holds that 
\begin{equation}%
\begin{split}
& 
\left\lvert 
(g_d(x)-g_d(y))-( g_d(\tilde{x})-g_d(\tilde{y}))
\right\rvert 
\leq b(d)
\left(
\lVert(x-y)-(\tilde{x}-\tilde{y})\rVert + \tfrac{(\lVert x-y\rVert+\lVert\tilde{x}-\tilde{y}\rVert)\lVert x-\tilde{x}\rVert}{2}\right)
\\
&
\leq 
\tfrac{V_d(T,x)+V_d(T,y)+V_d(T,\tilde{x})+V_d(T,\tilde{y})}{4}
\left(
\tfrac{\lVert(x-y)-(\tilde{x}-\tilde{y})\rVert}{\sqrt{T}} + \tfrac{(\lVert x-y\rVert+\lVert\tilde{x}-\tilde{y}\rVert)\lVert x-\tilde{x}\rVert}{2T}\right)
\end{split}
\end{equation}%
and
\begin{equation}
\begin{split}
&
\left\lvert 
(f_d(v)-f_d(w))-( f_d(\tilde{v})-f_d(\tilde{w}))
\right\rvert 
\leq c\left\lvert (v-w)-(\tilde{v}-\tilde{w})\right\rvert  + \tfrac{c\left(\lvert v-w\rvert + \lvert \tilde{v}-\tilde{w}\rvert\right)\lvert v-\tilde{v}\rvert}{2}.
\end{split}\label{s05b}
\end{equation}%
This implies for all $d\in\N$, $x,y\in\R^d$, $v,w\in\R$ that
\begin{equation}\label{s78}\begin{split}
&
\lvert g_d(x)-g_d(y)\rvert=
\tfrac{\left\lvert (g_d(x)-g_d(y)) -(g_d(x)-g_d(x))
\right\rvert}{2}
+
\tfrac{\left\lvert (g_d(y)-g_d(x)) -(g_d(y)-g_d(y))
\right\rvert}{2}\\
&\leq\tfrac{1}{2} \tfrac{3V_d(T,x)+V_d(T,y)}{4}
\tfrac{\lVert x-y\rVert}{\sqrt{T}}
+ \tfrac{1}{2}
\tfrac{3V_d(T,y)+V_d(T,x)}{4}
\tfrac{\lVert y-x\rVert}{\sqrt{T}}= \tfrac{V_d(T,x)+V_d(T,y)}{2}\tfrac{\lVert x-y\rVert}{\sqrt{T}}
\end{split}\end{equation}and
\begin{equation}\label{s79}
\lvert f_d(v)-f_d(w)\rvert =
\lvert (f_d(v)-f_d(w))-(f_d(v)-f_d(v))\rvert\leq c \lvert v-w\rvert.
\end{equation}
Next, \eqref{p17}, the fact that $\forall\,A,B\in [0,\infty)\colon A+B\leq 2\sqrt{A^2+B^2}$, and \eqref{p13} show 
 for all
$d\in\N$,
$x\in\R^d$  that 
\begin{align}
\lVert\mu_d(0)\rVert+\lVert\sigma_d(0)\rVert+c\lVert x\rVert\leq 
\lvert b(d)\rvert+c\lVert x\rVert
\leq 
2\left[\lvert b(d)\rvert^2+ c^2\lVert x \rVert^2\right]^{\frac{1}{2}}=  (\varphi_d(x))^{\frac{1}{2q}}.\label{p14}
\end{align}
Moreover, 
the fact that
$\forall\, d\in\N,x\in\R^d\colon 
\varphi_d(x)= \bigl(4\lvert b(d)\rvert^2+ 4c^2\lVert x \rVert^2\bigr)^{q}
$ (see \eqref{p13}), the fact that 
$q\geq 3$, 
 and \cite[Lemma~3.1]{HN21}
 (applied for every $d\in\N$ with
$p\gets q$, $a\gets 4 \lvert b(d)\rvert^2$, $c\gets 2c$, $V\gets\varphi_d$
in the notation of 
\cite[Lemma~3.1]{HN21})
show for all 
$d\in\N$,
$x,y\in\R^d$ that
$
\bigl\lvert((\totalD \varphi_d)(x))(y)\bigr\rvert
\leq {4q}{c} (\varphi_d(x))^{\frac{2q-1}{2q}}\lVert y\rVert$ and $
\bigl\lvert ((\totalD^2 \varphi_d)(x))(y,y)\bigr\rvert
\leq 16q^2c^2(\varphi_d (x))^{\frac{2q-2}{2q}}\lVert y\rVert^2.
$
This and the fact that $\bar{c}= \max \{4qc,16q^2c^2\}$
show for all
$d\in\N$,
$x,y,\tilde{x},\tilde{y}\in\R^d$ that 
\begin{equation}
\bigl\lvert((\totalD \varphi_d)(x))(y)\bigr\rvert
\leq \bar{c}(\varphi_d(x))^{\frac{2q-1}{2q}}\lVert y\rVert
\quad\text{and}\quad 
\bigl\lvert ((\totalD^2 \varphi_d)(x))(y,y)\bigr\rvert
\leq \bar{c}(\varphi_d(x))^{\frac{2q-2}{2q}}\lVert  y\rVert^2.\label{m42}
\end{equation}
Next,
\eqref{p18}, \eqref{p19}, and \cite[Lemma 3.3]{HN21} show for all
$d\in\N$,
$\zeta\in \{\mu_d,\sigma_d\}$,
$x,y,\tilde{x},\tilde{y}\in\R^d$ that 
\begin{equation}
\begin{split}
\left\lVert 
(\zeta(x)-\zeta (y))-
(\zeta(\tilde{x})-\zeta (\tilde{y}))\right\rVert 
\leq c
\left\lVert 
(x-y)-(\tilde{x}-\tilde{y})\right\rVert 
+b(d)
\tfrac{\left\lVert x-y\right\rVert 
+\left\lVert \tilde{x}-\tilde{y}\right\rVert 
}{2}\lVert x-\tilde{x}\rVert.
\label{v03}
\end{split}
\end{equation}
This, 
\eqref{p14}, \eqref{s01}, \eqref{p15}, and
\cite[Theorem~3.2]{HN21}
(applied for every $d,n\in\N$ with
$m\gets d$, 
$b\gets b(d)$,
$p\gets 2q$, $\mu\gets\mu_d$, $\sigma\gets \sigma_d$, $V\gets \varphi_d$
 in the notation of \cite[Theorem~3.2]{HN21}), 
the fact that $q\geq 4$,
Jensen's inequality,  the fact that
$1\leq  q/\beta \leq q $, and the fact that
$\forall\,t,s\in [0,T]\colon \lvert t-s\rvert^{\nicefrac{1}{2}}\leq \sqrt{T}\leq \sqrt{T}+p$
show that
\begin{enumerate}[(I)]\itemsep0pt
\item it holds 
for all $d,n\in\N$,
 $s\in[0,T]$, $t\in [s,T]$,
 $x\in\R^d$
  that 
\begin{align}\label{r06b}
\E\bigl[ \varphi_d(\approximationX_{s,t}^{d,n,0,x})\bigr]\leq e^{1.5\bar{c}\lvert t-s\rvert}\varphi_d(x)
,
\end{align}
\item  it holds for all
 $d,n\in\N$,
$s\in[0,T]$, 
$t\in [s,T]$,
$x\in\R^d$ that
\begin{align}\label{m29}
\left\lVert 
\approximationX_{s,t}^{d,n,0,x}
-X_{s,t}^{x}\right\rVert_{\frac{q}{\beta}}
\leq 
 \sqrt{2}
c\left[
\sqrt{T}+2q
\right]^4
e^{ c^2\left[
\sqrt{T}+2q
\right]^2 T}
(e^{1.5\bar{c} T}\varphi_d(x))^{\frac{1}{2q} }
 \tfrac{1}{\sqrt{n}}
,
\end{align}

\item 
 it holds for all $d,n\in\N$,
$ s,\tilde{s}\in [0,T]$,
$t\in[s,T]$,
$\tilde{t}\in[\tilde{s},T]$,
 $x,\tilde{x}\in\R^d$ that
\begin{align}\label{m28}\begin{split}
&
\left \lVert \approximationX^{d,n,0,x}_{s,t}
-
\approximationX^{d,n,0,\tilde{x}}_{\tilde{s},\tilde{t}}\right\rVert_{\frac{q}{\beta}}\leq \sqrt{2}\lVert x-\tilde{x} \rVert 
e^{c^2\left[
\sqrt{T}+2q
\right]^2 T}\\
&+
{5} e^{c^2\left[
\sqrt{T}+2q
\right]^2T}
\left[
\sqrt{T}+2q
\right]
e^{\frac{1.5 \bar{c} T}{2q}}
\frac{(\varphi_d(x))^{\frac{1}{2q}}
+(\varphi_d(\tilde{x}))^{\frac{1}{2q} }}{2}
\left[
\lvert s-\tilde{s}\rvert^{\nicefrac{1}{2}} +
\lvert t-\tilde{t}\rvert^{\nicefrac{1}{2}}\right]
,
\end{split}\end{align}

\item it holds
for all 
$d,n\in \N$, $s\in[0,T]$,
$t, \tilde{t}\in[s,T]$, 
$x,\tilde{x}\in\R^d$ that
\begin{align}\label{m36}
&\left \lVert 
\bigl(x-X_{s,t}^{d,x}\bigr)
-\bigl(\tilde{x}-X^{d,\tilde{x}}_{s,t}
\bigr)\right\rVert_{\frac{q}{\beta}}
\leq c\left[
\sqrt{T}+2q
\right]^2\sqrt{2}
e^{c^2\left[
\sqrt{T}+2q
\right]^2 T} 
\tfrac{\lVert x-\tilde{x}\rVert\lvert t-s\rvert^{\nicefrac{1}{2}}}{\sqrt{T}}
,
\end{align}

\item 
\label{m33}
it holds for all 
$d,n\in\N$, $s\in[0,T]$,
$t\in[s,T]$, 
$x,\tilde{x},y,\tilde{y}\in\R^d$ that
\begin{align}
&\left\lVert 
\bigl(
X^{d,x}_{s,t}
-X^{d,y}_{s,t}
\bigr)
-\bigl(
X^{d,\tilde{x}}_{s,t}
-X^{d,\tilde{y}}_{s,t}
\bigr)
\right\rVert_{\frac{q}{\beta}}
  \leq \sqrt{2}e^{c^2\left[\sqrt{T}+2q\right]^2T}
\left\lVert 
\bigl(
x-y \bigr)
-
\bigl(
\tilde{x}-\tilde{y} \bigr)\right\rVert\nonumber\\
& + 2\sqrt{2}b(d)\left[\sqrt{T}+2q\right]^3 e^{3 c^2\left[
\sqrt{T}+2q
\right]^2 T}
\tfrac{\left(\lVert x-y \rVert+
\lVert \tilde{x}-\tilde{y} \rVert\right) \lVert x-\tilde{x}\rVert}{2\sqrt{T}}
,\label{m33b}
\end{align}
and
\item 
\label{m32}it holds
for all $d,n\in\N$,
 $s,\tilde{s}\in[0,T]$, 
$t\in [s,T]$, $ \tilde{t}\in[\tilde{s},T]$,
$x,\tilde{x}\in\R^d$ that
\begin{align}
&
\left\lVert \bigl(
X_{s,t}^{d,x}
-
X_{\tilde{s},\tilde{t}}^{d,\tilde{x}}
\bigr)
 -
\bigl(\approximationX_{s,t }^{d,n,0,x} 
-\approximationX_{\tilde{s},\tilde{t}}^{d,n,0,\tilde{x}} \bigr)\right\rVert_{\frac{q}{\beta}}\leq 
62 (b(d)+c)(c+1)
\left[\sqrt{T}+2q\right]^7
e^{5 c^2\left[\sqrt{T}+2q\right]^2 T}\nonumber\\
&\qquad\qquad
 e^{\frac{4.5\bar{c} T}{2q}}\tfrac{(\varphi_d(x))^{\nicefrac{1}{q}}
+(\varphi_d(\tilde{x}))^{\nicefrac{1}{q} }}{2}
\left[
\tfrac{\lvert s-\tilde{s}\rvert^{\nicefrac{1}{2}}+
\lvert t-\tilde{t}\rvert^{\nicefrac{1}{2}}}{2}
+\lVert x-\tilde{x}\rVert\right]
\tfrac{1}{\sqrt{n}}.\label{m32b}
\end{align}
\end{enumerate}
This, the Markov property,  \eqref{s01},  and \eqref{p15} show
for all $d\in\N$,
 $s,\tilde{s}\in[0,T]$, 
$t\in[s,T]$,
$\tilde{t}\in[\tilde{s},T] $, $r\in [t,T]$,
$x,\tilde{x},y,\tilde{y}\in\R^d$, 
$A\in ( \mathcal{B}(\R^d))^{\otimes \R^d}$,
$B\in (\mathcal{B}(\R^d))^{\otimes ([t,T]\times\R^d)}
$ 
that
\begin{equation}\begin{split}
\P \!\left(
X_{ t,r}^{d,X_{ s,t }^{d,x} } = X_{s,r}^{d,x}\right)=1
,\quad 
\P \!\left(X_{ s,t }^{d,(\cdot)} \in A,
X_{ t,\cdot}^{d,(\cdot)} \in B\right)= 
\P \!\left(X_{ s,t }^{d,(\cdot)} \in A\right)\P \!\left(
X_{ t,(\cdot)}^{d,(\cdot)} \in B\right).
\end{split}\label{p21}
\end{equation}%
\begin{equation}
\begin{split}
&
\left\lVert \left\lVert 
(X_{s,t}^{d,x}-X_{s,t}^{d,y})
-(X_{s,t}^{d,\tilde{x}}-X_{s,t}^{d,\tilde{y}})\right\rVert
\right\rVert_{\frac{q}{\beta}}\\
&\leq \xeqref{m33b}
\tfrac{V_d(s,x)+V_d(s,y)+V_d(s,\tilde{x})+V_d(s,\tilde{y})}{4}
\left[
\left\lVert 
(x-y)-(\tilde{x}-\tilde{y})
\right\rVert
+ \tfrac{\left(
\lVert x-y\rVert+\lVert\tilde{x}-\tilde{y}\rVert\right)
\lVert x-\tilde{x}\rVert}{2\sqrt{T}}\right]
,
\end{split}
\end{equation}%
\begin{equation}\begin{split}
\left\lVert \left\lVert \bigl( X_{s,t}^{d,x}- 
X_{s,t}^{d,y}\bigr) -
( x-y)\right\rVert\right\rVert_{\frac{q}{\beta}}\leq\xeqref{m36} \tfrac{(V_d(s,x)+V_d(s,y))}{2} \tfrac{\lVert x-y\rVert\lvert t-s\rvert^{\nicefrac{1}{2}}}{\sqrt{T}},
\end{split}\end{equation}%
\begin{equation}\begin{split}
&
\left\lVert 
\left\lVert 
\approximationX_{s,t}^{d,n,0,x}-
\approximationX_{\tilde{s},\tilde{t}}^{d,n,0,\tilde{x}}\right\rVert\right\rVert_{\frac{q}{\beta}}\leq\xeqref{m28}
\tfrac{V_d(s,x)+V_d(\tilde{s},\tilde{x})}{2}
\tfrac{\lvert s-\tilde{s}\rvert^{\nicefrac{1}{2}}+\lvert t-\tilde{t}\rvert^{\nicefrac{1}{2}}}{2}
+
e^{\left(\frac{\ln(2)}{T} +c^2[\sqrt{T}+2q]^2\right)T  }
\lVert x-\tilde{x}\rVert,
\end{split}\label{p23}\end{equation}
\begin{equation}
\begin{split}
&\left\lVert \left\lVert 
 \bigl( \approximationX^{d,n,0,x}_{s,t} -
X^{d,x}_{s,t}\bigr)
-   \bigl(\approximationX^{d,n,0,\tilde{x}}_{\tilde{s},\tilde{t}} - X^{d,\tilde{x}}_{\tilde{s},\tilde{t}}\bigr)\right\rVert\right\rVert_{\frac{q}{\beta}}
\leq \xeqref{m32b}
\tfrac{V_d(s,x)+V_d(\tilde{s},\tilde{x})}{2}\left[\tfrac{\lvert s-\tilde{s}\rvert^{\nicefrac{1}{2}} 
+\lvert t-\tilde{t}\rvert^{\nicefrac{1}{2}}}{2}+
 \lVert x-\tilde{x}\rVert\right]\tfrac{1}{\sqrt{n}},
\end{split}\label{p23c}
\end{equation}
\begin{equation}
\P(X^{d,x}_{s,s}=x)=
\P\!\left(\approximationX^{d,n,0,x}_{s,s}=x\right)=1,
\quad\text{and}\quad
\sqrt{2}e^{ c^2[\sqrt{T}+2q]^2 T} = 
e^{\left(\frac{\ln(2)}{2T} +c^2[\sqrt{T}+2q]^2\right)T  }\leq  V_d(s,x).\label{s06}
\end{equation}
Next, 
\eqref{r06b} shows
for all $d,n\in\N$, $ x\in\R^d$, $s\in[0,T]$, $t\in[s,T]$ that
\begin{align}\begin{split}
&
\left
\lVert
e^{\frac{1.5\bar{c}\beta(T-t)}{ q}}
(\varphi_d(\approximationX_{s,t}^{d,n,0,x}))^{\frac{\beta}{q}}\right\rVert_{\frac{q}{\beta}}=e^{\frac{1.5\bar{c}\beta(T-t)}{ q}} \bigl(\E\bigl[ \varphi_d(\approximationX_{s,t}^{d,n,0,x})\bigr]\bigr)^{\frac{\beta}{q}}\\
&\leq e^{\frac{1.5\bar{c}\beta(T-t)}{ q}} e^{\frac{1.5\beta\bar{c}\lvert t-s\rvert}{q}}(\varphi_d(x))^{\frac{\beta}{q}}=e^{\frac{1.5\bar{c}\beta(T-s)}{ q}} (\varphi_d(x))^{\frac{\beta}{q}}.
\end{split}\end{align}
This, \eqref{p13}, and \eqref{p16} 
show
for all $d\in\N$, $ x\in\R^d$, $s\in[0,T]$, $t\in[s,T]$ that
\begin{align}
\bigl\lVert V_d\bigl(t,\approximationX_{s,t}^{d,n,0,x}\bigr)\bigr\rVert_{\frac{q}{\beta}}\leq V_d(s,x).\label{p22}
\end{align}
This, \eqref{m29},  continuity of $V_d$, $d\in\N$, and Fatou's lemma show for all
$d\in\N$, $ x\in\R^d$, $s\in[0,T]$, $t\in[s,T]$ that
$
X_{s,t}^{d,x}=
\operatornamewithlimits{\P\text{-}\lim}_{n\to\infty}
\approximationX_{s,t}^{d,n,0,x}
$
and
$
\bigl\lVert V_d\bigl(t,X_{s,t}^{d,x}\bigr)\bigr\rVert_{\frac{q}{\beta}}
=
\left\lVert
\operatornamewithlimits{\P\text{-}\lim}_{n\to\infty} V_d\bigl(t,\approximationX_{s,t}^{d,n,0,x}\bigr)\right\rVert_{\frac{q}{\beta}}
\leq \liminf_{n\to\infty}
\bigl\lVert V_d\bigl(t,\approximationX_{s,t}^{d,n,0,x}\bigr)\bigr\rVert_{\frac{q}{\beta}}\leq V_d(s,x).$
This, \eqref{s01}, the Feynman-Kac formula, 
the fact that $f_d$, $d\in\N$, are Lipschitz continuous (see \eqref{p18}), \eqref{s05},
and \eqref{p10} show \eqref{s01b}.

Next, \eqref{m29}, \eqref{s06}, and \eqref{p23} show for all
$d\in\N$, $ x\in\R^d$, $s\in[0,T]$, $t\in[s,T]$ that
\begin{equation}\begin{split}
&
\left\lVert 
\left\lVert 
X_{s,t}^{d,x}-x\right\rVert\right\rVert_{\frac{q}{\beta}}=
 \xeqref{m29} \xeqref{s06}
\lim_{n\to\infty}
\left\lVert 
\left\lVert 
\approximationX_{s,t}^{d,n,0,x}-
\approximationX_{{s},{s}}^{d,n,0,{x}}\right\rVert\right\rVert_{\frac{q}{\beta}}\leq\xeqref{p23}
\tfrac{V_d(s,x)+V_d({s},{x})}{2}
\tfrac{\lvert t-{s}\rvert^{\nicefrac{1}{2}}}{2}.
\end{split}\label{p25}\end{equation}
Moreover, \eqref{m29} and \eqref{s06} show for all
$d\in\N$, $ x\in\R^d$, $s\in[0,T]$, $t\in[s,T]$ that
\begin{equation}\begin{split}
&\left\lVert 
\left\lVert 
X_{s,t}^{d,x}-
X_{{s},{t}}^{d,\tilde{x}}\right\rVert\right\rVert_{\frac{q}{\beta}}=\xeqref{m29}\lim_{n\to\infty} 
\left\lVert 
\left\lVert 
\approximationX_{s,t}^{d,n,0,x}-
\approximationX_{{s},{t}}^{d,n,0,\tilde{x}}\right\rVert\right\rVert_{\frac{q}{\beta}}\\
&\leq\xeqref{m28}
\sqrt{2} 
e^{c^2\left[
\sqrt{T}+2q
\right]^2 T}
\lVert x-\tilde{x}\rVert\leq \xeqref{s06}
\tfrac{V_d(s,x)+V_d(\tilde{s},\tilde{x})}{2}
\lVert x-\tilde{x}\rVert
.
\end{split}\label{p23b}\end{equation}
In addition, \eqref{s06} and \eqref{p23c} show for all
$d,n\in\N$, $ x\in\R^d$, $s\in[0,T]$, $t\in[s,T]$ that
\begin{equation}
\begin{split}
&
\left\lVert \left\lVert 
 \approximationX^{d,n,0,x}_{s,t} -
X^{d,x}_{s,t}\right\rVert\right\rVert_{\frac{q}{\beta}}
=\xeqref{s06}
\left\lVert \left\lVert 
 \bigl( \approximationX^{d,n,0,x}_{s,t} -
X^{d,x}_{s,t}\bigr)
-   \bigl(\approximationX^{d,n,0,x}_{s,s} - X^{d,x}_{s,s}\bigr)\right\rVert\right\rVert_{\frac{q}{\beta}}\\
&
\leq \xeqref{p23c}
V_d(s,x)\tfrac{\lvert t-s\rvert^{\nicefrac{1}{2}} 
}{2}
\tfrac{1}{\sqrt{n}}\leq \tfrac{\sqrt{T}V_d(s,x)}{2\sqrt{n}}.
\end{split}\label{p26}
\end{equation}
Next, for every $d\in\N$ let 
$\mathbf{\tilde{Y}}^d=(\tilde{Y}^d,\tilde{Z}^{d,1},\tilde{Z}^{2,d},\ldots,\tilde{Z}^{d,d})\colon [0,T]\times\Omega\to \R^{d+1}
$ satisfy for all 
$i\in [1,d]\cap\Z$,
$t\in[0,T]$ that
$\tilde{Y}^d_t=u_d(t,X_{0,t}^{d,0}) $ and
$\tilde{Z}_t^{d,i}=(\frac{\partial}{\partial x_j})u_d(t,X_{0,t}^{d,0})  \sigma_{d,j,i} (X_{0,t}^{d,0})$. Then It\^o's formula, the regularity assumptions of $u_d$, $d\in\N$, and the fact that
$\forall\,d\in\N,x\in\R^d\colon g_d(x)=u_d(T,x)$ show that
for all  $t\in[0,T]$ it holds a.s.\ that
$
\tilde{Y}^d_t=g_d({X}^{d,x}_{0,T})+\int_t^T f_d(\tilde{Y}^d_s)\,ds-\sum_{j=1}^{d}\int_t^T \tilde{Z}_s^{d,j}\, dW_s^{d,j}.
$ This and a result on uniqueness of solutions to BSDEs (see, e.g., \cite[Theorem~1.1]{HKN21}) imply that $\tilde{\mathbf{Y}}=\mathbf{Y}$. This implies
\eqref{s08} and \eqref{s04}.


For the rest of this proof
 let $\Delta_{d,n,m}, \mathcal{E}_{d,n,m}\in[0,\infty]$,  $d,m,n\in\N$, satisfy for all
$d,m,n\in\N$
 that
\begin{align}\label{t13c}\begin{split}
&\Delta_{d,n,m}=m^{\min \{\delta,1\} n/2}\Biggl\{
\left(\sup_{t\in [0,T]}
\left\lVert \Yappr^{d,n,m}_{t}-  Y^d_t \right\rVert_{\exponentLP}\right)\\
&\qquad\qquad\qquad\qquad\qquad\qquad
+m^{-n/2}\left[\left(\sup_{t\in [0,T]}\left\lVert
Y^d_t\right\rVert_{p}\right)+
\left(
\sup_{t,s\in [0,T]\colon t\neq s}
\tfrac{\sqrt{T}\left\lVert
Y^d_t
-Y^d_s
\right\rVert_{p}}{\lvert t-s\rvert^{\nicefrac{1}{2}}}\right)\right]\Biggr\}
\end{split}
\end{align}
and
\begin{align}\mathcal{E}_{d,n,m}=
\left(\E\!\left[
\sup\limits_{t\in[0,T]}
\left\lvert \Yappr^{d,n,m}_{t}-  Y^d_t \right\rvert^{\exponentLP}\right]\right)^{\nicefrac{1}{p}},
\end{align}
 let $K\in [0,\infty]$ satisfy that
\begin{align}\begin{split}
&K=\inf\left(\left\{ \gamma\in[2,\infty]\colon 
\forall\,d,n,m\in\N\colon \mathcal{E}_{d,n,m}
\leq \gamma\Delta_{d,n,m}\right\}\cup\{\infty\}\right),
\end{split}\label{s07}
\end{align} and
let $\mathsf{n}\colon\N\times  (0,\infty) \to [0,\infty]$ satisfy for all $d\in\N$, $\epsilon\in(0,\infty)$ that
\begin{align}\label{s09}\textstyle
\mathsf{n}(d,\varepsilon)= \inf\! \left(\left\{n\in\N\colon\sup_{k\in\N\cap  [n,\infty)}\sup_{t\in[0,T]}
\Delta_{d,k,M(k)}<\tfrac{\epsilon}{K}\right\}\cup\{\infty\}\right).
\end{align}
Now 
\cite[Theorem~1.1]{CHJvNW21}
  (applied with $E\gets\R$, $\epsilon\gets \min\{\delta,1\}/2$, $\alpha\gets 0$, $\beta\gets 1/2$ in the notation of \cite[Theorem~1.1]{CHJvNW21})
implies that
$K<\infty$.
In addition,
\cref{m01} (applied for every $d\in\N$ with
$c\gets\frac{\ln(2)}{T} +c^2[\sqrt{T}+2q]^2$,
$f\gets f_d$, 
$g\gets g_d$, 
$V\gets V_d$, $(X^{x}_{s,t})_{s\in[0,T],t\in[s,T],x\in\R^d}\gets (X^{d,x}_{s,t})_{s\in[0,T],t\in[s,T],x\in\R^d}$,
$(\approximationX^{d,n,\theta,x}_{s,t})_{\theta\in\Theta,n\in\N,s\in[0,T],t\in[s,T],x\in\R^d}\gets (\approximationX^{n,\theta,x}_{s,t})_{\theta\in\Theta,n\in\N,s\in[0,T],t\in[s,T],x\in\R^d}$,
$(U^{\theta}_{n,m} )_{n,m\in\Z,\theta\in\Theta}\gets (U^{d,\theta}_{n,m})_{n,m\in\Z,\theta\in\Theta}$,
$\exponentV\gets \frac{q}{\beta}$,
$\delta\gets \min\{\delta,1\}$,
$\exponentX\gets \frac{q}{\beta}$,
$\expFirstNorm\gets 3$, $\expSecondNorm\gets 9$,
$(\Yappr^{n,m})_{n,m\in\N} \gets  (\Yappr^{d, n,m})_{n,m\in\N}$,
$a_1\gets a_1(d)$,
$a_2\gets a_2(d)$, $a_3\gets a_3(d)$,
  $ ( {\FEU}_{ n,m})_{n,m\in\Z}\gets ( {\FEU}_{ d,n,m})_{n,m\in\Z}$, 
$(\FEY_{n,m})_{n,m\in\Z}\gets (\FEY_{d,n,m})_{n,m\in\Z}$
in the notation of \cref{m01}),
 \eqref{s05}--\eqref{s05b}, \eqref{p21}--\eqref{s06}, \eqref{p22}, the fact that
$\frac{9+1}{(q/\beta)}+\frac{2}{(q/\beta)}=\frac{12\beta}{q}= \frac{1}{\exponentLP} $,  the assumptions of \cref{s02}, \eqref{s01b}, and
\eqref{s04} 
show that
there exists $\gamma\in [1,\infty)$ such that for all
$d,n\in\N$ it holds that
$
\Delta_{d,n,M(n)}\leq \gamma^n (M(n))^{-n/2} (V_d(0,0))^{10}.
$
This,  \eqref{s09}, and the fact that $K<\infty$ show for all $d\in\N$, 
$\epsilon\in (0,1)$
 that $\mathsf{n}(d,\epsilon)<\infty$. 
This, \eqref{s07}, \eqref{s09},
 and the fact that $K<\infty$ 
 show for all $d\in\N$, 
$\epsilon\in (0,1)$,
$n\in [\mathsf{n}(d,\epsilon),\infty)$ that
\begin{align}\begin{split}
\mathcal{E}_{d,n,M(n)}\leq\xeqref{s07} K\Delta_{d,n,M(n)}<\xeqref{s09}\epsilon.
\end{split}\label{s13}\end{align}
Next,
\cref{c01} and \eqref{c10}
show for all $d,n,m\in\N$ that
$
{\FEU}_{d,n,m}\leq \max\{4 a_1(d),a_2(d)n\}(5m)^n
$. 
This and \eqref{c09} imply that there exists $\gamma\in [1,\infty)$ such that for all $d,m,n\in\N$ it holds that
\begin{align}\begin{split}
&\FEY_{d,n,m}\leq a_3(d)(m^n+1)+\sum_{\ell=0}^{n-1} \left[(m^{\ell+1}+1)\FEU_{d,n-\ell,m}\right]\\
&\leq 2a_3(d)m^n+ 
\sum_{\ell=0}^{n-1} \left[2 m^{\ell+1} \max\{4 a_1(d),a_2(d)n\}(5m)^{n-\ell}\right]\leq \max\{a_1(d),a_2(d),a_3(d)\}\gamma^n m^{n+1}.\end{split}\label{k18}
\end{align} 
This shows that there exists $\gamma\in [1,\infty)$ such that for all $d,n\in\N$ it holds that
$
\FEY_{d,n+1,M(n+1)}
\leq \max\{a_1(d),a_2(d),a_3(d)\}\gamma^{n+1} (M(n+1))^{n+2}.
$
This and \eqref{c08} imply that there exists $\gamma\in [1,\infty)$ such that
$
\FEY_{d,n+1,M(n+1)}\leq \max\{a_1(d),a_2(d),a_3(d)\}(\gamma M(n))^{n+2}.
$ This, \eqref{s07}, and the fact that $K<\infty$
show that there exist $\gamma, C\in [1,\infty)$ such that for all
$d,n\in\N$ it holds that
\begin{align}\begin{split}
&\FEY_{d,n+1,M(n+1)}
\left\lvert \mathcal{E}_{d,n,M(n)}\right\rvert^{2+\delta}\leq \xeqref{s07}
\FEY_{d,n+1,M(n+1)}
\left\lvert K
\Delta_{d,n,M(n)}\right\rvert^{2+\delta}\\
&
\leq \max\{a_1(d),a_2(d),a_3(d)\}(\gamma M(n))^{n+2}
\left( K\gamma^n (M(n))^{-n/2} (V_d(0,0))^{10}\right)^{\delta+2}\\
&\leq \max\{a_1(d),a_2(d),a_3(d)\}\gamma^{(\delta+2)(2n+2)}K^{\delta+2} (M(n))^2(M(n))^{-0.5n\delta}
(V_d(0,0))^{10(\delta+2)}\\
&
\leq \max\{a_1(d),a_2(d),a_3(d)\}C(V_d(0,0))^{10(\delta+2)}.
\end{split}\end{align}
This and \eqref{k18} show that there exists $C\in (0,\infty)$ such that
for all $d\in\N$, $\epsilon\in (0,1)$ it holds that
\begin{align}\begin{split}
&
\FEY_{d,\mathsf{n}(d,\epsilon),M(\mathsf{n}(d,\epsilon))}\epsilon^{2+\delta}
\leq 
\FEY_{d,1,M(1)}+\1_{[2,\infty)}(\mathsf{n}(d,\epsilon))
\FEY_{d,\mathsf{n}(d,\epsilon),M(\mathsf{n}(d,\epsilon))}\epsilon^{2+\delta}\\
&\leq \FEY_{d,1,M(1)}+
\1_{[2,\infty)}(\mathsf{n}(d,\epsilon))\left[
\FEY_{d,n+1,M(n+1)}
\left\lvert
\mathcal{E}_{d,n,M(n)}\right\rvert^{2+\delta}\right]\Bigr|_{n=\mathsf{n}(d,\epsilon)-1}\\
&\leq \max\{a_1(d),a_2(d),a_3(d)\}C(V_d(0,0))^{10(\delta+2)}.
\end{split}\end{align}
This, \eqref{s13}, and the fact that
$ \exists \,\gamma\in (0,\infty)\,\forall\,d\in \N\colon 
 V_d(0,0) +\sum_{i=1}^{3}\lvert a_i(d)\rvert\leq \gamma d^\gamma$ 
(see \eqref{p13} and \eqref{p16} and recall the fact that
$ \exists\, \gamma\in (0,\infty)\,\forall\,d\in \N\colon 
b(d)+
\sum_{i=1}^{3}a_i(d)\leq \gamma d^\gamma
$)
imply  \eqref{s03}. The proof of \cref{s02} is thus completed.
\end{proof}

\paragraph*{Acknowledgement}
This work has been
funded by the Deutsche Forschungsgemeinschaft (DFG, German Research Foundation) through
the research grant HU1889/6-2.

{
\small
\bibliographystyle{acm}
\bibliography{bib-XBSDE-114}
}

\end{document}